\documentclass[10pt]{article}
\usepackage{amsthm,amsfonts,amsmath,amscd,amssymb,epsfig}
\usepackage[all]{xy}

\usepackage{enumerate,array}
\theoremstyle{plain}
\newtheorem{theorem}{Theorem}[section]
\newtheorem{thm}[theorem]{Theorem}
\newtheorem{corollary}[theorem]{Corollary}
\newtheorem{cor}[theorem]{Corollary}
\newtheorem{proposition}[theorem]{Proposition}
\newtheorem{prop}[theorem]{Proposition}
\newtheorem{lemma}[theorem]{Lemma}

\theoremstyle{remark}
\newtheorem*{conjecture}{Conjecture}
\newtheorem*{question}{Question}
\newtheorem*{remark}{Remark}

\newtheorem{example}[theorem]{Example}

\theoremstyle{definition}
\newtheorem*{definition}{Definition}
\newcommand{\p}{\partial}
\newcommand{\om}{\omega}
\newcommand{\Om}{\Omega}
\newcommand{\eps}{\varepsilon}

\newcommand{\la}{\langle}
\newcommand{\ra}{\rangle}

\newcommand{\Z}{{\mathbb{Z}}}
\newcommand{\R}{{\mathbb{R}}}

\newcommand{\T}{{\mathbb{T}}}

\newcommand{\bS}{{\bf S}}

\newcommand{\BH}{{\mathbb{H}}}

\renewcommand{\min}{{\rm min}}
\renewcommand{\max}{{\rm max}}

\newcommand{\LL}{\mathcal{L}}

\renewcommand{\AA}{\mathcal{A}}

\newcommand{\g}{{\mathfrak g}}         

\newcommand{\comment}[1]{}
\newcommand{\x}{\times}

\newcommand{\pa}{{p_{\alpha}}}
\newcommand{\pb}{{p_{\beta}}}
\newcommand{\pc}{{p_{\gamma}}}
\def\co{\colon\thinspace}
\title{Symplectic topology of Ma\~n\'e's critical values}
\author{
K.~Cieliebak\footnote{Partially supported by DFG grant CI 45/1-3},
U.~Frauenfelder and  G.P.~Paternain}

\begin{document}
\maketitle
{\it Abstract: }We study the dynamics and symplectic topology of energy
hypersurfaces of mechanical Hamiltonians on twisted cotangent bundles.
We pay particular attention to periodic orbits, displaceability,
stability and the contact type property, and the changes that occur at
the Ma\~n\'e critical value $c$. Our main tool is Rabinowitz Floer
homology. We show that it is defined for hypersurfaces that are either
stable tame or virtually contact, and that it is invariant under
homotopies in these classes. If the configuration space admits a
metric of negative curvature, then Rabinowitz Floer homology does not
vanish for energy levels $k>c$ and, as a consequence, these
level sets are not displaceable. We provide a large class of examples
in which Rabinowitz Floer homology is non-zero for energy levels
$k>c$ but vanishes for $k<c$, so levels above and below $c$ cannot be
connected by a stable tame homotopy. Moreover, we show that for
strictly $1/4$-pinched negative curvature and non-exact magnetic fields all
sufficiently high energy levels are non-stable, provided that the dimension
of the base manifold is even and different from two.
\tableofcontents

\section{Introduction}\label{sec:intro}

In this paper we study the significance of Ma\~n\'e's critical value to
the symplectic topology of energy hypersurfaces. Steps in this
direction were taken in \cite{B,PPS,Vi2}, but here we will have
a different focus. We will attempt to relate Ma\~n\'e's critical value
with stable Hamiltonian structures and a new type of Floer homology that
we develop along the lines of \cite{CF}, namely {\it Rabinowitz Floer homology}.

On the cotangent bundle $\tau\co T^*M\to M$ of a closed manifold $M$ we
consider autonomous Hamiltonian systems defined by a convex Hamiltonian
$$
   H(q,p) = \frac{1}{2}|p|^2 + U(q)
$$
and a twisted symplectic form
$$
   \om = dp\wedge dq + \tau^*\sigma.
$$
Here $dp\wedge dq$ is the canonical symplectic form in canonical
coordinates $(q,p)$ on $T^*M$, $|p|$ denotes
the dual norm of a Riemannian metric $g$ on $M$, $U\co M\to\R$ is a
smooth potential, and $\sigma$ is a closed 2-form on $M$. This
Hamiltonian system describes the motion of a particle on $M$ subject
to the conservative force $-\nabla U$ and the magnetic field
$\sigma$.

We assume that $\sigma$ vanishes on $\pi_2(M)$, so its pullback
$\pi^*\sigma$ to the universal cover $\pi\co \widetilde M\to M$ is exact. The
{\em Ma\~n\'e critical value}
\footnote{More generally, there is a Ma\~n\'e critical value associated to
any cover of $M$ on which $\sigma$ becomes exact. In this paper, we
will mostly restrict to the universal and abelian covers.}
is defined as
$$
   c = c(g,\sigma,U) := \inf_\theta\sup_{q\in \widetilde{M}} \widetilde{H}(q,\theta_q),
$$
where the infimum is taken over all 1-forms $\theta$ on $\widetilde M$
with $d\theta=\pi^*\sigma$ and $\widetilde{H}$ is the lift of $H$.

We wish to understand how dynamical and symplectic topological
properties of regular level sets $\Sigma_k=H^{-1}(k)$ change as $k$
passes the Ma\~n\'e critical value. More specifically, we will investigate
the following properties.

{\em Dynamics. }
The dynamics usually changes drastically at the Ma\~n\'e critical value;
we will provide abundant examples of this in Section \ref{sec:ex}
of this paper. We will pay particular attention to the existence or
non-existence of periodic orbits in given free homotopy classes of
loops. The study of the existence of closed orbits of a charged
particle in a magnetic field was initiated by V.I. Arnold \cite{A}
and S.P. Novikov \cite{No} in the 1980's.

{\em Displaceability. }
A subset $A$ of a symplectic manifold $(V,\om)$ is {\em displaceable}
if there exists a Hamiltonian diffeomorphism (time-1 map of a
time-dependent compactly supported Hamiltonian system) $\phi$ with $\phi(A)\cap
A=\emptyset$. Among the many consequences of displaceability, the most
relevant for this paper is the following result of Schlenk~\cite{Sch}:
If for $H$ as above and some $k_0$ the set $\{H\leq k_0\}$ is
displaceable, then it has finite Hofer-Zehnder capacity
(see~\cite{HZ}) and for almost every $k\leq k_0$ the energy level
$\Sigma_k$ carries a contractible periodic orbit.

{\em Contact type. }
A hypersurface $\Sigma$ in a symplectic manifold $(V^{2n},\om)$
is of {\em contact type} if $\om|_\Sigma=d\lambda$ for a contact form
$\lambda$ on $\Sigma$, i.e.~a 1-form such that
$\lambda\wedge(d\lambda)^{n-1}$ is
nowhere vanishing. This property was introduced by Weinstein~\cite{We}
and has important dynamical consequences. For example, for large
classes of contact type hypersurfaces the existence of a periodic
orbit has been proved (``Weinstein conjecture'', see
e.g.~\cite{HZ,Vi1}). Moreover, algebraic counts of periodic orbits can
be organized into invariants of such hypersurfaces such as contact
homology or symplectic field theory~\cite{EGH}.

We say that a hypersurface $\Sigma$ is {\em virtually contact} if
$\pi^*\om|_{\widetilde\Sigma}=d\lambda$ for a contact form
$\lambda$ on a cover $\pi\co \widetilde\Sigma\to\Sigma$ satisfying
$$
   \sup_{x\in\widetilde\Sigma}|\lambda_x|\leq C<\infty, \qquad
   \inf_{x\in\widetilde\Sigma}\lambda(R)\geq\eps >0,
$$
where $|\cdot|$ is a metric on $\Sigma$ and $R$ is a vector field
generating $\ker(\om|_\Sigma)$ (both pulled back to $\widetilde\Sigma$).
A {\em virtually contact homotopy} is a smooth homotopy
$(\Sigma_t,\lambda_t)$ of virtually contact hypersurfaces together
with the corresponding 1-forms on the covers such that the preceding
conditions hold with constants $C,\eps$ independent of $t$.

{\em Stability and tameness. }
A hypersurface $\Sigma$ in a symplectic manifold $(V^{2n},\om)$
is {\em stable} if there exists a 1-form $\lambda$ on $\Sigma$ which
is nonzero on $\ker(\om|_\Sigma)$ and satisfies
$\ker(\om|_\Sigma)\subset\ker(d\lambda)$.
Note that contact type implies stability, but stability is more
general, e.g.~it allows for $\om|_\Sigma$ to be non-exact.
The notion of stability was introduced
by Hofer and Zehnder~\cite{HZ} and shares many consequences of the
contact type condition, e.g.~existence results for periodic
orbits~\cite{HZ} and algebraic invariants arising from symplectic
field theory~\cite{EGH}.
A {\em stable homotopy} is a smooth homotopy $(\Sigma_t,\lambda_t)$ of
stable hypersurfaces together with the corresponding 1-forms.

Suppose now $\om$ vanishes on $\pi_{2}(V)$, and let $X(\Sigma)$
be the set of all closed characteristics in $\Sigma$ which
are contractible in $V$.  We define the
function
$$\Omega \colon X(\Sigma) \to \mathbb{R}$$
by choosing for $v \in X(\Sigma)$ a filling disk $\bar{v}$ in $V$
which exists since $v$ is contractible in $V$, and setting
\[\Omega(v)=\int \bar{v}^* \omega.\]
The pair $(\Sigma,\lambda)$ is said to be {\it tame} if
there exists a positive constant $c$ such that
\[\left|\int v^{*}\lambda\right|\leq c|\Omega(v)|\]
for all $v\in X(\Sigma)$. Again, abundant examples of tame
stable hypersurfaces will be given in Section \ref{sec:ex}, but
we should mention that there are also examples of stable non-tame
hypersurfaces, see \cite[Example 5.1]{CMP} and \cite{CV}.

A stable homotopy $(\Sigma_{t},\lambda_{t})$
is said to be {\it tame} if each  $(\Sigma_{t},\lambda_{t})$ is
tame and the constant $c$ is independent of $t$.

We remark that when $V$ is a cotangent bundle $T^*M$,
then $X(\Sigma)$ coincides with the closed orbits in $\Sigma$ whose
projection to $M$ is contractible. Moreover, if $\dim M\geq 3$ and there is no potential, it follows
from the homotopy sequence of a fibration that $\pi_{1}(\Sigma)$ injects into
$\pi_1(T^*M)$ and thus
$X(\Sigma)$ coincides
with the set of closed orbits in $\Sigma$ which are contractible in $\Sigma$
itself.

{\em Rabinowitz Floer homology. }
In~\cite{CF} the first two authors associated to a contact type
hypersurface $\Sigma$ in a symplectic manifold $(V,\om)$ (satisfying
suitable conditions, see Section~\ref{sec:RFH}) its {\em Rabinowitz
Floer homology} $RFH(\Sigma)$. This is a graded $\Z_2$-vector
space which is invariant under contact type homotopies of
$\Sigma$. Its relevance for our purposes rests on the following two
results in~\cite{CF}: Displaceability of $\Sigma$ implies vanishing of
$RFH(\Sigma)$, and vanishing of $RFH(\Sigma)$ implies the existence of
a periodic orbit on $\Sigma$ which is contractible in $V$. To make these results
applicable to the situation considered in this paper, we generalize in
Section~\ref{sec:RFH} the construction of Rabinowitz Floer homology to
hypersurfaces that are stable tame or virtually contact, and to the
corresponding homotopies. The generalization to the stable case is
definitely non-trivial. The generalization to the virtually contact case is
straightforward, but necessary, as we will try to explain in due time.
In what follows we shall assume that $\Sigma$ is separating, i.e. $V\setminus\Sigma$ consists of two connected components.

We call a symplectic manifold $(V,\om)$ {\em
  geometrically bounded} if it admits a $\omega$-compatible almost complex structure $J$ and a complete Riemannian metric such that
\begin{enumerate}
\item there are positive constants $c_1$ and $c_2$ such that
\[\omega(v,Jv)\geq c_{1}|v|^2,\;\;\;\;\;\;|\omega(u,v)|\leq c_{2}|u||v|\]
for all tangent vectors $u$ and $v$ to $V$;
\item The sectional curvature of the metric is bounded above and its
injectivity radius is bounded away from zero.
\end{enumerate}
This is slightly stronger than the corresponding notion
in~\cite{ALP} since we demand $\omega$-compatibility. 
It is proved in~\cite{Lu} that twisted cotangent bundles are
geometrically bounded. 

\begin{thm}\label{thm:RFH}
Let $(V,\om)$ be a geometrically bounded symplectic manifold such that
$\om|_{\pi_2(V)}=0$ (e.g.~a twisted cotangent bundle as above). Then:

(a) Rabinowitz Floer homology $RFH(\Sigma)$ is defined for each tame stable
hypersurface $\Sigma$ and invariant under tame stable homotopies.

(b) Rabinowitz Floer homology $RFH(\Sigma)$ is defined for each
virtually contact hypersurface $\Sigma$ and invariant under virtually
contact homotopies provided that $\pi_{1}(\Sigma)$ injects into $\pi_{1}(V)$.

(c) If $\Sigma$ is as in (a) or (b) and is displaceable, then
$RFH(\Sigma)=0$.

(d) If $\Sigma$ is as in (a) or (b) satisfies $RFH(\Sigma)=0$, then it
carries a periodic orbit which is contractible in $V$.
\end{thm}

We remark at once that Schlenk's result alluded to above can be derived
right away from Theorem \ref{thm:RFH} if we assume that the
hypersurface is tame. However, the compactness result proved
in Section \ref{sec:RFH} will show that one can recover Schlenk's result
fully, i.e. we can do without the tameness assumption, this is explained
in Subsection \ref{ss:existence}.

One of the goals of this paper is to provide supporting evidence for the
following {\em paradigms}:

$(k>c)$: Above the Ma\~n\'e critical value, $\Sigma_k$ is virtually
contact. It may or may not be stable.
Its Rabinowitz Floer homology $RFH(\Sigma_k)$ is defined and
nonzero, so $\Sigma_k$ is non-displaceable. The dynamics on $\Sigma_k$
is like that of a geodesic flow; in particular, it has a periodic
orbit in every nontrivial free homotopy class.

$(k=c)$: At the Ma\~n\'e critical value, $\Sigma_k$ is non-displaceable
and can be expected to be non-stable (hence
non-contact). (When $M$ is the 2-torus, an example is given in \cite{CMP}
in which $\Sigma_c$ is of contact type.)

$(k<c)$: Below the Ma\~n\'e critical value, $\Sigma_k$ may or may not be
of contact type.
It is stable and displaceable (provided that $\chi(M)=0$), so its
Rabinowitz Floer homology $RFH(\Sigma_k)$ is defined and vanishes. In
particular, $\Sigma_k$ has a contractible periodic orbit.

Some of these statements will be proved, others verified for various
classes of examples, while some remain largely open. It should be said that
the paradigms for energies $k\leq c$
are only rough approximations to the true picture. For example, work
in progress \cite{MPa} shows that there are convex superlinear Hamiltonians
on $\mathbb T^n$ (for any $n\geq 2$) for which $\Sigma_c$ is actually stable.

The first supporting
evidence for the paradigms comes from the following

\begin{thm}\label{thm:sigma=0}
For $\sigma=0$ all the paradigms are true. Moreover, in this case all
regular level sets are of contact type.
\end{thm}

A well-known example with $\sigma\neq 0$ where all the paradigms are
true is a closed hyperbolic surface with $\sigma$ the area form \cite{G1}.
Here for $k>c=1/2$, $\Sigma_k$ is contact and Anosov with all closed orbits
non-contractible; for $k<1/2$, $\Sigma_k$ is stable (in fact contact
with the opposite orientation) and the flow is completely periodic
with contractible orbits; for $k=1/2$, $\Sigma_k$ is unstable and
the flow is the horocycle flow without periodic orbits. We briefly discuss
this example, as well as its generalization to higher dimensions, in
Section~\ref{ss:hyp}.

\begin{thm}\label{thm:non-displaceable}
Suppose that $M$ admits a metric of negative curvature. Then for $k>c$
the Rabinowitz Floer homology $RFH(\Sigma_k)$ is defined and
does not vanish. In particular, $\Sigma_k$ is non-displaceable for
$k\geq c$.
\end{thm}

\begin{remark}
We point out that the chosen metric on $M$ need not have negative
curvature (we merely ask the existence of such a metric on $M$)
and the dynamics on $\Sigma_k$ need not be a small perturbation of the
geodesic flow. The proof of Theorem~\ref{thm:non-displaceable} uses
invariance of Rabinowitz Floer homology in a crucial way. In fact,
existence of a metric of negative curvature is a technical hypothesis
that can most likely be removed by establishing
invariance of Rabinowitz Floer homology under more general deformations,
e.g.~deforming the symplectic form in a suitable class or adding an
additional Hamiltonian term. 
\footnote{Note added in proof: Non-vanishing of $RFH(\Sigma_k)$ has in
  the meantime been proven without the hypothesis of a metric of
  negative curvature, for $\sigma$ exact and $k>c_0$ in~\cite{AbS},
  and for general $\sigma$ and all $k>c$ in~\cite{Me2}.} 
\end{remark}

If $\sigma$ is exact one can define another Ma\~n\'e critical value
$$
   c_0 = c_0(g,\sigma,U) := \inf_\theta\sup_{q\in M} H(q,\theta_q),
$$
where the infimum is taken over all 1-forms $\theta$ on $M$
with $d\theta=\sigma$. Note that $c_0\geq c$. For $k>c_0$ the level
set $\Sigma_k$ encloses the Lagrangian graph ${\rm gr}(-\theta)$ and is
therefore non-displaceable by Gromov's theorem~\cite{Gr}.
However, there are examples \cite{PP} with $c<c_0$.
Non-displaceability for the gap $c\leq k\leq c_0$ is new.
In Theorem \ref{thm:gap} we show that the gap appears rather frequently and
in Section~\ref{sec:ex} we will explain how a gap of size $1/4$ appears
quite explicitly in the geometry $PSL(2,\R)$.

We set $c_0:=\infty$ if $\sigma$ is non-exact. Then the contact type
property behaves as follows with respect to the values $c$ and $c_0$.
Suppose $\dim M\geq 3$. Then:
\begin{itemize}
\item for $k>c_0$, $\Sigma_k$ is of contact type
  (Lemma~\ref{lem:exact});
\item for $c<k\leq c_0$, $\Sigma_k$ is virtually contact but
  not of contact type (Lemma~\ref{virtc} and~\cite[Theorem B.1]{Co}).
\end{itemize}
It turns out that {\em whole intervals of unstable levels} may appear in the
gap. The first example where this phenomenon occurs was constructed
in~\cite{PP} on surfaces of negative curvature and with exact
$\sigma$, see Theorem~\ref{thm:PP} below. With non-exact $\sigma$, this
phenomenon {\em always} occurs for strictly $1/4$-pinched negative
sectional curvature.

\begin{theorem}\label{thm:pinched}
Let $(M,g)$ be a closed Riemannian manifold of even dimension
different from two whose sectional curvature satisfies the pinching
condition $-4\leq K< -1$.
Let $\sigma$ be a closed 2-form with cohomology class
$[\sigma]\neq 0$. Then $c<c_0=\infty$ and for any $k>c$ sufficiently
large, the hypersurface $\Sigma_k$ is not stable.
\end{theorem}

The theorem becomes false if the strict pinching condition is replaced
by the weak one $-4\leq K\leq -1$: Compact quotients of complex hyperbolic
space with $\sigma$ given by the K\"ahler form are stable for high energies
(cf. Subsection \ref{ss:hyp}).

The next result is an easy consequence of the results by Schlenk
and will be useful to show non-stability at the
Ma\~n\'e critical value in the various examples.

\begin{thm}\label{thm:crit}
Suppose that at the Ma\~n\'e critical value $c$, $\Sigma_{c}$ has
no contractible periodic orbits.
Then $\Sigma_c$ is non-stable provided that
all level sets $\Sigma_k$ with $k<c$ are displaceable.
\end{thm}


More evidence for the displaceability of $\Sigma_k$ with $k<c$
comes from results of F. Laudenbach and J.-C. Sikorav \cite{LS}
and L. Polterovich \cite{Po} which assert that the zero section of
$(T^{*}M,\omega)$ is actually displaceable if $\sigma$
is non-zero (assuming $\chi(M)=0$). Moreover, for a large class of
cotangent bundles of solvable
manifolds with $[\sigma]\neq 0$, displaceability for all $k<c=\infty$
is proved in \cite{BP}.

As an illustration of the homogeneous examples considered
in Section \ref{sec:ex}, let $G$ be the 3-dimensional Heisenberg
group of matrices
\[\left(\begin{array}{ccc}

1&x&z\\
0&1&y\\
0&0&1\\

\end{array}\right),\]
where $x,y,z\in \R$. The 1-form $\gamma:=dz-xdy$ is left-invariant and we
let $\sigma:=d\gamma$ be the exact magnetic field. If $\Gamma$ is a co-compact
lattice in $G$, $M:=\Gamma\setminus G$ is a closed 3-manifold and $\sigma$
descends to an exact 2-form on $M$. Now let $H$ be the left-invariant
Hamiltonian defined in dual coordinates by
\[2H:=p_{x}^2+(p_{y}+xp_{z})^2+p_{z}^2.\]

\begin{thm}\label{thm:nil}
Consider $M$, $H$ and $\sigma$ as above. Then all the paradigms are true. More
precisely: Each level $\Sigma_k$ except the Ma\~n\'e critical
level $c=c_0=1/2$ is stable and tame. For $k>1/2$ it is contact and
$RFH(\Sigma_k)\neq 0$. For $k<1/2$ it is non-contact and displaceable, so
$RFH(\Sigma_k)=0$. $\Sigma_k$ has no contractible periodic orbits for
$k\geq 1/2$. The level set $\Sigma_{1/2}$ is not stable.
\end{thm}

By invariance of Rabinowitz Floer homology under tame stable homotopies,
this implies

\begin{cor}
In the example of Theorem~\ref{thm:nil} two level sets
$\Sigma_k,\Sigma_\ell$ with $k<1/2<\ell$ are smoothly homotopic and
tame stable, but not tame stably homotopic.
\end{cor}

An even more intriguing example arises on compact quotients
$M:=\Gamma\setminus PSL(2,\R)$. Let $(x,y,\theta)$ be coordinates on
$PSL(2,\R)$ arising from its identification with $S\mathbb H^2\cong \mathbb H^2\times
S^1$, the unit tangent bundle of the upper half plane. The
left-invariant 1-form $\gamma = dx/y+d\theta$ gives rise to an exact
magnetic field $\sigma:=d\gamma$ on $M$. Let $H$ be the left-invariant
Hamiltonian defined in dual coordinates by
\[2H:= (yp_x-p_\theta)^2+(yp_{y})^2+p_\theta^2.\]

\begin{thm}\label{thm:PSL}
Consider $M$, $H$ and $\sigma$ as above. Then all the paradigms are true. More
precisely: Each level $\Sigma_k$ except the Ma\~n\'e critical
levels $c=1/4$ and $c_0=1/2$ is stable and tame. For $k>1/2$ it is contact and
$RFH(\Sigma_k)\neq 0$. For $1/4<k<1/2$ it is non-contact and
$RFH(\Sigma_k)\neq 0$. For $k<1/4$ it is non-contact and displaceable, so
$RFH(\Sigma_k)=0$. $\Sigma_k$ has no contractible periodic orbits for
$k\geq 1/4$.
The level sets $\Sigma_{1/4}$ and $\Sigma_{1/2}$ are not stable.
\end{thm}

\begin{remark}
Non-stability of the Ma\~n\'e critical level $c$ in
Theorems~\ref{thm:nil} and~\ref{thm:PSL} follows from
Theorem~\ref{thm:crit}. Non-stability of the level $c_0$ in
Theorem~\ref{thm:PSL} follows from a detailed analysis of the
dynamics on the level set. Note that the level $c_0$ is virtually contact;
in fact for any $k>1/4$, $\Sigma_k$ is virtually contact.
\end{remark}

Finally, we also point out the following examples with
infinite Ma\~n\'e critical value. An analogous picture arises on
Sol-manifolds discussed in Section~\ref{ss:sol}.

\begin{thm}\label{thm:torus}
Consider the $n$-torus $M=\T^n$ and a nonzero constant 2-form
$\sigma$. Then the Ma\~n\'e  critical value is $c=\infty$ and all level
sets $\Sigma_k$ are non-contact, stable, tame and displaceable, so
$RFH(\Sigma_{k})=0$ for all $k$.
\end{thm}

We conclude this paper with a discussion of very high and low energy
levels. For levels $k>c$ the only remaining issue is:

\begin{conjecture}
For $k>c$ the Rabinowitz Floer homology $RFH(\Sigma_k)$
does not vanish. In particular, $\Sigma_k$ is non-displaceable for
$k\geq c$.
\end{conjecture}

It is known that for $k>c$, $\Sigma_k$ carries a periodic orbit
in every nontrivial free homotopy class of loops (this is proved
in \cite{P} under a mild technical condition on $\pi_{1}(M)$,
which is removed in \cite{Me1}.)

For $k<c$ the dynamics is much less well understood, even for
very small values of $k$. If $\sigma\neq 0$, the results of
Polterovich and Schlenk mentioned above yield a $k_0>0$ such that
$\Sigma_k$ is displaceable for all $k\leq k_0$ and carries a
contractible periodic orbit for almost all $k\leq k_0$. However, the
following basic question is wide open:

\begin{question}
Is $\Sigma_k$ stable for sufficiently small $k$, at least
in the case that $\sigma$ is a symplectic form on $M$?
\end{question}

A positive answer to the question would give an alternative proof
of the existence of closed contractible orbits on every low energy level
for $\sigma$ symplectic. This has been recently proved by
V.~Ginzburg and B.~G\"urel~\cite{GG}.
In Section~\ref{ss:symp} (Proposition~\ref{prop:symp}) we give
an affirmative answer in the homogeneous symplectic case:
{\em Let $M=\Gamma\setminus G$ be a compact homogeneous space with a
left-invariant metric and a left-invariant symplectic form $\sigma$.
Then there exists $k_0>0$ such that for all $k<k_0$ the
hypersurface $\Sigma_k$ is stable.}
\medskip

\noindent{\it Acknowledgements:}
We wish to thank L.~Macarini, F.~Schlenk and E.~Volkov for several
useful discussions related to this paper. We also thank W.~Merry for
comments and corrections on previous drafts. We thank the organizers
of the wonderful Conference on Symplectic Geometry, Kyoto 2007, which
inspired the present work. Finally, we thank the referee for numerous
corrections.

\section{Stable Hamiltonian structures and
  hypersurfaces}\label{sec:stable}

{\bf Stable Hamiltonian structures. }
Let $\Sigma$ be a closed oriented manifold of dimension $2n-1$. A {\em
Hamiltonian structure} on $\Sigma$ is a closed 2-form $\om$ such that
$\om^{n-1}\neq 0$. So its kernel $\ker(\om)$ defines a 1-dimensional
foliation which we call the {\em characteristic foliation} of
$\om$. We orient the characteristic foliation by a 1-form $\lambda$
such that $\lambda\wedge\om^{n-1}>0$.

A Hamiltonian structure is called {\em stable} if there exists a 1-form
$\lambda$ such that $\ker\om\subset \ker d\lambda$ and
$\lambda\wedge\om^{n-1}>0$. We call $\lambda$ a {\em stabilizing
1-form}. Define the {\em Reeb vector field} $R$ by
$\lambda(R)=1$ and $i_R\om=0$ (which implies $i_Rd\lambda=0$).

The following two results give equivalent formulations of stability.

\begin{thm}[Wadsley~\cite{Wa}]\label{thm:wadsley}
A Hamiltonian structure $(\Sigma,\om)$ is stable if and only if its
characteristic foliation is {\em geodesible}, i.e.~there exists a
Riemannian metric such that all leaves are geodesics.
\end{thm}

\begin{thm}[Sullivan~\cite{Su}]\label{thm:sullivan}
A Hamiltonian structure $(\Sigma,\om)$ is non-stable if and only if
there exists a foliation cycle which can be arbitrarily well
approximated by boundaries of singular $2$-chains tangent to the
foliation.
\end{thm}

The simplest obstruction to stability that can appear in Sullivan's
theorem is a {\em Reeb component}, i.e.~an embedded annulus tangent to
the characteristic foliation such that its boundary components are
closed leaves with opposite orientations.

The following criterion for stability (whose proof is immediate) will
be useful in later examples.

\begin{lemma}\label{lem:proj-stable}
Let $(\Sigma^\pm,\om^\pm)$ be stable Hamiltonian structures and
$f\co \Sigma^+\to \Sigma^-$ a smooth (not necessarily injective) map
which maps leaves diffeomorphically onto leaves. If
a 1-form $\lambda$ stabilizes $(\Sigma^-,\om^-)$, then $f^*\lambda$
stabilizes $(\Sigma^+,\om^+)$.
\end{lemma}

For a Hamiltonian structure $(\Sigma,\om)$ we denote by
$$
   \Lambda(\Sigma,\omega) := \{\lambda \in \Omega^1(\Sigma):
   \mathrm{ker}\omega \subset \mathrm{ker}d\lambda,\,\,\lambda \wedge
   \omega^{n-1}>0\}
$$
the space of stabilizing 1-forms. It obviously satisfies

\begin{lemma}\label{cone}
The space $\Lambda(\Sigma,\omega)$ is a convex cone in $\Omega^1(\Sigma)$.
In particular, if it is nonempty, then it is contractible.
\end{lemma}

{\bf Stable hypersurfaces. }
We call a closed oriented connected hypersurface $\Sigma$ in a connected
symplectic manifold $(V^{2n},\om)$ {\em stable} if the following holds:
\begin{description}
 \item[(i)] $\om|_\Sigma$ defines a stable Hamiltonian structure;
 \item[(ii)] $\Sigma$ is {\em separating}, i.e.~$V\setminus \Sigma$ consists of
 two connected components.
\end{description}

Condition (a) in the following lemma gives a more dynamical
formulation of stability and justification for its name.

\begin{lemma}[\cite{CM}]\label{lem:HZ-stable}
For a closed hypersurface $\Sigma$ in a symplectic manifold $(V,\om)$ the
following are equivalent:

(a) $\Sigma$ is {\em stable} in the
sense of~\cite{HZ}, i.e.~there exists a tubular
neighborhood $(-\eps,\eps)\times \Sigma$ of $\Sigma=\{0\}\times \Sigma$ such that
the Hamiltonian line fields on $\{r\}\times \Sigma$ are conjugate for
all $r\in(-\eps,\eps)$.

(b) There exists a vector field $Y$ transverse to $\Sigma$ such that
$\ker(\om|_\Sigma)\subset\ker(L_Y\om|_\Sigma)$.

(c) $(\Sigma,\om|_\Sigma)$ is a stable Hamiltonian structure.
\end{lemma}

A {\em stable homotopy of hypersurfaces} in $(V,\om)$ is a smooth
homotopy $(\Sigma_t,\lambda_t)$ of stable hypersurfaces together with
stabilizing 1-forms.

{\bf Stable tubular neighbourhoods. }
Now assume that $\Sigma$ is a stable hypersurface in a symplectic
manifold $(V,\omega)$. We abbreviate
$$\omega_{\Sigma}=\omega|_{\Sigma}.$$
If $\lambda_\Sigma \in \Lambda(\Sigma, \omega_\Sigma)$ we call the
pair $(\Sigma,\lambda_\Sigma)$ a \emph{stabilized hypersurface}. For
a stabilized hypersurface $(\Sigma,\lambda_\Sigma)$ a  \emph{stable tubular
neighbourhood} is a pair $(\rho,\psi)$ where $\rho>0$ and
$\psi \colon [-\rho,\rho] \times \Sigma \to V$ is an embedding
satisfying
$$\psi|_{\{0\}\times \Sigma}=\mathrm{id}|_{\Sigma}, \quad
\psi^*\omega=\omega_\Sigma+d(r\lambda_\Sigma)=\omega_\Sigma+r
d\lambda_\Sigma+dr \wedge \lambda_\Sigma.$$
Note that a stable tubular neighbourhood satisfies condition (a) of
Lemma~\ref{lem:HZ-stable}.
We abbreviate by $\mathcal{T}(\Sigma,\lambda_\Sigma)$ the space
of stable tubular neighbourhoods of $(\Sigma,\lambda_\Sigma)$.
We further denote by
$$\mathcal{T}(\Sigma)=\bigcup_{\lambda_\Sigma \in \Lambda(\Sigma,\omega|_\Sigma)}
\mathcal{T}(\Sigma,\lambda_\Sigma)$$
the space of stable tubular neighbourhoods for the
stable (but not stabilized) hypersurface $\Sigma$. If
$(\Sigma_\sigma,\lambda_\sigma)$ for $\sigma \in [0,1]$
is a stable homotopy of hypersurfaces we abbreviate by
$\mathcal{T}(\{\Sigma_\sigma\})$ the space of 
smooth families $(\rho_\sigma,\psi_\sigma)$ of stable tubular
neighbourhoods. 

\begin{prop}\label{prop:stable-homotopy}
(a) Assume that $(\Sigma_\sigma,\lambda_\sigma)$ is a
stable homotopy of compact hypersurfaces. Then
$\mathcal{T}(\{\Sigma_\sigma\})$ is nonempty.

(b) If $\Sigma$ is a compact stable hypersurface, then
$\mathcal{T}(\Sigma)$ is nonempty and path-connected.
\end{prop}

\begin{proof}
(a)
Choose a smooth family of vector fields $X_\sigma$ on
$V$ satisfying
\begin{equation}\label{vecfi}
\iota_{X_\sigma} \omega_x=\lambda_\sigma, \quad x \in \Sigma_\sigma.
\end{equation}
Since $\Sigma_\sigma$ is compact
the flow $\phi^r_{X_\sigma}$ exists locally near $\Sigma_\sigma$.
We observe that plugging in the Reeb vector field into
(\ref{vecfi}) implies that $X_\sigma$ is
transverse to $\Sigma_\sigma$.
Hence we can define a smooth family  of diffeomorphisms
$\tilde{\psi}_\sigma \colon (-\tilde{\rho}_\sigma, \tilde{\rho}_\sigma)
\times \Sigma_\sigma \to V$ for $\tilde{\rho}_\sigma>0$ by the formula
$$\tilde{\psi}_\sigma(r,x)=\phi^r_{X_{\sigma}}(x), \quad
(r,x) \in (-\tilde{\rho}_\sigma,\tilde{\rho}_\sigma) \times \Sigma_\sigma.$$
We abbreviate
$$\omega_\sigma=\omega_{\Sigma_\sigma}+d(r\lambda_\sigma).$$
Perhaps after shrinking $\tilde{\rho}_\sigma$ it follows that
$\omega_\sigma$ is a symplectic structure on
$(-\tilde{\rho}_\sigma,\tilde{\rho}_\sigma)\times \Sigma_\sigma$.
Moreover, it follows from (\ref{vecfi}) that
$\omega_\sigma$ and $\tilde{\psi}_\sigma^* \omega$ agree at points of
$\{0\}\times \Sigma_\sigma$. Applying Moser's argument,
see for example \cite[Lemma 3.14]{MS}, we find a smooth family
of $\rho_\sigma>0$ and a smooth family of embeddings
$\phi_\sigma\co (-\rho_\sigma,\rho_\sigma)\times \Sigma_\sigma\to
(-\tilde\rho_\sigma,\tilde\rho_\sigma)\times \Sigma_\sigma$ satisfying
$$\phi_\sigma|_{\{0\}\times \Sigma_\sigma}=\mathrm{id},
\quad \phi^*_\sigma \tilde{\psi}^*_\sigma \omega=\omega_\sigma.$$
Now set
$$\psi_\sigma= \tilde{\psi}_\sigma\circ \phi_\sigma.$$
Then $(\rho_\sigma, \psi_\sigma)$ lies in $\mathcal{T}(\{\Sigma_\sigma\})$.
In particular, $\mathcal{T}(\{\Sigma_\sigma\})$ is nonempty.

(b) In view of part (a), $\mathcal{T}(\Sigma)$ is nonempty.
To prove that it is path-connected we first show that
for each $\lambda \in \Lambda(\Sigma,\omega_\Sigma)$
the space $\mathcal{T}(\Sigma,\lambda)$ is path-connected.
Let $(\rho_1,\psi_1), (\rho_2,\psi_2) \in \mathcal{T}(\Sigma,\lambda)$.
By hypothesis, 
$$\omega_\lambda=\omega|_{\Sigma}+d(r\lambda)$$
is symplectic on $U:=(-\rho,\rho)$ for some
$\rho>\max(\rho_1,\rho_2)$. 
There exist neighbourhoods $U_1$ and $U_2$ of $\{0\}\times \Sigma$ in
$U$ such that
$\psi_2^{-1}\circ\psi_1 \colon U_1 \to U_2$ is an isomorphism and
$$\big(\psi_2^{-1}\circ\psi_1\big)^*\omega_\lambda|_{U_2}=
\omega_\lambda|_{U_1}$$
Note further that
$$\psi_2^{-1}\circ\psi_1|_{\{0\}\times \Sigma}=\mathrm{id}.$$
Hence after choosing $U_1$ even smaller we can in a Weinstein neighbourhood
identify the graph $\Gamma_{\psi_2^{-1}\circ\psi_1} \subset (U \times
U,\om_\lambda\oplus -\om_\lambda)$
with an open subset of
the graph $\Gamma_\sigma\subset T^*U$ of a closed one-form $\sigma$ on
$U$. 
Considering the homotopy of graphs $\Gamma_{t\sigma}$ for $t \in [0,1]$
we find an $\epsilon>0$ and a path in
$\mathcal{T}(\Sigma,\lambda)$ between
$(\epsilon,\psi_1|_{(-\epsilon,\epsilon)\times \Sigma})$ and
$(\epsilon,\psi_2|_{(-\epsilon,\epsilon)\times \Sigma})$.
Now concatenating this path with shrinking paths between
$(\rho_1,\psi_1)$ and $(\epsilon,\psi_1|_{(-\epsilon,\epsilon)\times \Sigma})$
as well as between $(\rho_2,\psi_2)$ and
$(\epsilon,\psi_2|_{(-\epsilon,\epsilon)\times \Sigma})$
we obtain a path between $(\rho_1,\psi_1)$ and $(\rho_2,\psi_2)$.
This proves that for fixed $\lambda \in \Lambda(\Sigma,\omega_\Sigma)$
the space $\mathcal{T}(\Sigma,\lambda)$ is path-connected.

Now assume more generally that  $(\rho_0,\psi_0), (\rho_1,\psi_1)
\in \mathcal{T}(\Sigma)$. Then there exist
$\lambda_0,\lambda_1 \in \Lambda(\Sigma,\omega_\sigma)$
such that
$$(\rho_0,\psi_0) \in \mathcal{T}(\Sigma,\lambda_0),\quad
(\rho_1,\psi_1) \in \mathcal{T}(\Sigma,\lambda_1).$$
Since we have already seen that $\mathcal{T}(\Sigma,\lambda_0)$
and $\mathcal{T}(\Sigma,\lambda_1)$ are path-connected it suffices
to connect arbitrary elements in these spaces by a path
in $\mathcal{T}(\Sigma)$. To do that we make use of
Lemma~\ref{cone} giving us a path
$\lambda_\sigma$ in $\Lambda(\Sigma,\omega_\Sigma)$ connecting
$\lambda_0$ and $\lambda_1$. Hence we can apply part (a) to the stable
homotopy $(\Sigma,\lambda_\sigma)$. This proves the proposition.
\end{proof}

{\bf Contact structures. }
A Hamiltonian structure $(\Sigma,\om)$ is called {\em contact} if
there exists a 1-form $\lambda$ such that $d\lambda=\om$ and
$\lambda\wedge\om^{n-1}>0$. In particular, $\lambda$ is a stabilizing
1-form and $\lambda$ is a (positive) {\em contact form},
i.e.~$\lambda\wedge(d\lambda)^{n-1}>0$. Note that $(\Sigma,\om)$ can
be contact only if $\om$ is exact.

Sullivan's theory in~\cite{Su} also provides a necessary and
sufficient condition for an exact Hamiltonian structure
$(\Sigma,\om)$ being contact. Fix a positive vector field $R$ generating
$\ker\om$. Every Borel probability measure $\mu$
on $\Sigma$ invariant under the flow of $R$ gives rise to a
1-current via
$$
   \la\mu,\beta\ra = \int_\Sigma\beta(R)d\mu,\qquad
   \beta\in\Om^1(\Sigma).
$$
We say $\mu$ is {\em exact} if it is exact as a 1-current,
i.e.~$\la\mu,\beta\ra=0$ for all closed 1-forms $\beta$.

\begin{thm}[McDuff \cite{McD}]\label{thm:mcduff}
An exact Hamiltonian structure $(\Sigma,\om)$ is non-contact if and
only if there exists a nontrivial exact positive invariant Borel
measure $\mu$ 
such that $\la\mu,\alpha\ra=0$ for some (and hence every) 1-form
$\alpha$ with $d\alpha=\om$.
\end{thm}

The simplest obstruction to the contact property arises if $\mu$ is
supported on a closed orbit: If there exists a null-homologous closed
orbit $\gamma$ of $\ker\om$ such that $\int_\gamma\alpha=0$ for a
primitive $\alpha$ of $\om$, then $(\Sigma,\om)$ is non-contact.

The following immediate consequence of the theorem will be useful
below.

\begin{cor}\label{cor:mcduff}
An exact Hamiltonian structure $(\Sigma,\om)$ is non-contact provided
there exist two nontrivial exact invariant Borel measures $\mu_\pm$
such that $\la\mu_+,\alpha\ra\geq 0$ and $\la\mu_-,\alpha\ra\leq 0$
for some (and hence every) 1-form $\alpha$ with $d\alpha=\om$.
\end{cor}

\begin{proof}
For $\mu_\pm$ as in the corollary, some positive linear combination
$\mu=a_+\mu_++a_-\mu_-$ satisfies the condition in
Theorem~\ref{thm:mcduff}.
\end{proof}

A particular invariant measure is given by the {\em Liouville measure}
associated to a Hamiltonian structure $(\Sigma,\om,R)$ with a chosen
vector field $R$ generating $\ker\om$. It is defined by the unique
volume form $\mu\in\Om^{2n-1}(\Sigma)$ satisfying
$$
   i_R\mu = \frac{\om^{n-1}}{(n-1)!}.
$$

\begin{lemma}\label{lem:liouville}
Consider a twisted cotangent bundle $(T^*M,\om = dp\wedge dq +
\tau^*\sigma)$ with a convex Hamiltonian $H(q,p) = \frac{1}{2}|p|^2 +
U(q)$ as in the Introduction. If $M\neq \mathbb T^2$, then on every regular
level set $(\Sigma_k=H^{-1}(k),\om|_{\Sigma_k},R=X_H)$
the Liouville measure is exact as a current.
\end{lemma}

\begin{proof}
We claim that $\om^{n-1}|_{\Sigma_k}$ is exact if $M\neq \mathbb T^2$. To see
this, let $\theta:=p\,dq$ be the Liouville form and compute
$$
   \om^{n-1} = (d\theta+\tau^*\sigma)^{n-1} = (d\theta)^{n-1} +
   (n-1)\tau^*\sigma\wedge(d\theta)^{n-2}.
$$
On the right-hand side the first term is exact. For $n\geq 3$ the
second term is exact as well and the claim follows, so it remains to
consider the case $n=2$. If $k<\max \,U$ the projection
$\tau_k:=\tau|_{\Sigma_k}\co \Sigma_k\to M$ is not surjective, so
$\sigma$ is exact on the image of $\tau_k$ and the claim follows. If
$k>\max \,U$ the Gysin sequence of the circle bundle $\tau_k\co \Sigma_k\to
M$ yields
$$
   H^0(M;\R)\overset{\cup e}{\longrightarrow} H^2(M;\R)
   \overset{\tau_k^*}{\longrightarrow} H^2(\Sigma_k;\R),
$$
where $e$ is the Euler class of the cotangent bundle of $M$. If $M\neq
\T^2$ this Euler class is nonzero, so the first map is an isomorphism
and $\tau_k^*$ the zero map. This proves the claim.

Now let $\Theta$ be a primitive of $\om^{n-1}|_{\Sigma_k}$ and
$\beta\in\Om^1(\Sigma_k)$ be closed. Then
$$
   (n-1)!\beta(R)\mu = (n-1)!\beta\wedge i_R\mu =
   \beta\wedge\om^{n-1}|_{\Sigma_k} = \beta\wedge d\Theta =
   -d(\beta\wedge\Theta)
$$
is exact, so its integral over $\Sigma_k$ vanishes. This proves
exactness of $\mu$ and hence the lemma.
\end{proof}

The following immediate consequence of Lemma~\ref{lem:liouville} will
be used repeatedly in this paper.

\begin{corollary}\label{cor:liouville}
In the situation of Lemma~\ref{lem:liouville}, there exists no 1-form
$\lambda$ on $\Sigma_k$ with $d\lambda=0$ and $\lambda(R)>0$.
\end{corollary}

\section{Tame hypersurfaces}\label{sec:tame}

In this section we introduce the notion of weakly tame
hypersurfaces in a symplectic manifold and the notion
of tameness for stable hypersurfaces. As the name suggests,
tame stable hypersurfaces are also weakly tame. We further
explain what a tame, stable homotopy is. In the
forthcoming section~\ref{sec:RFH} we then show how for
weakly tame, stable hypersurfaces Rabinowitz Floer homology
can be defined and that Rabinowitz Floer
homology is invariant under tame, stable
homotopies.

Given a closed hypersurface $\Sigma$ in a symplectically
aspherical symplectic  manifold $(V,\omega)$ we
denote by $X(\Sigma)$ the set of closed characteristics in
$\Sigma$ which are contractible in $V$. We define the
function
$$\Omega \colon X(\Sigma) \to \mathbb{R}$$
by choosing for $v \in X(\Sigma)$ a filling disk $\bar{v}$ in $V$
which exists since $v$ is contractible in $V$ and putting
\begin{equation}\label{OE}
\Omega(v)=\int \bar{v}^* \omega.
\end{equation}
Note that since $V$ is symplectically aspherical the function
$\Omega$ is well-defined independent of the choice of the
filling disk. We refer to the function $\Omega$ as the
\emph{$\omega$-energy} of a closed characteristic.
For $a \leq b$ we abbreviate
$$X_a^b(\Sigma)=\big\{v \in X(\Sigma): a \leq \Omega(v) \leq b\big\}.$$

The hypersurface $\Sigma$ is called \emph{weakly tame} if
for each $a \leq b$ the space $X_a^b(\Sigma)$ is compact (with the topology
of uniform convergence).

An example of weakly tame hypersurfaces are hypersurfaces of restricted
contact type. Indeed if $\omega=d\lambda$ is exact and the
restriction $\lambda|_\Sigma$ is a contact form on $\Sigma$, then
each closed characteristic can be parametrised as a periodic
orbit of the Reeb vector field $R_\lambda$ and hence
$$\Omega(v)=\int \bar{v}^*\omega=\int v^* \lambda=T_\lambda(v),$$
where $T_\lambda(v)$ is the period of $v$ as a periodic orbit of the
Reeb vector field $R_\lambda$. Therefore the theorem of
Arzela-Ascoli implies that $\Sigma$ is weakly tame. On the other
hand an obstruction for being weakly tame is the existence of a
closed characteristic of vanishing $\omega$-energy. Indeed, taking
iterates of it we get a sequence of closed characteristics of
vanishing $\omega$-energy having no convergent subsequence.

If the hypersurface $\Sigma$ is stable we can define for
each $\lambda \in \Lambda(\Sigma,\omega)$ another function
$$T_\lambda \colon X(\Sigma) \to (0,\infty), \quad
v \mapsto \int v^*\lambda.$$ The function $T_\lambda$ associates to
$v$ its period as a Reeb orbit of the Reeb vector field $R_\lambda$.
We note that for two different $\lambda_1, \lambda_2 \in
\Lambda(\Sigma,\omega)$ the functions $T_{\lambda_1}$ and
$T_{\lambda_2}$ are proportional in the sense that there exist
positive constants $\underline{c}_{\lambda_1,\lambda_2}$ and
$\overline{c}_{\lambda_1,\lambda_2}$ such that
\begin{equation}\label{prolam}
\underline{c}_{\lambda_1,\lambda_2}T_{\lambda_2} \leq T_{\lambda_1}
\leq \overline{c}_{\lambda_1,\lambda_2}T_{\lambda_2}.
\end{equation}
Indeed, the Reeb vector fields $R_{\lambda_1}$ and $R_{\lambda_2}$
are pointing in the same direction at each point of $\Sigma$ so that there
exists a positive function
$f_{\lambda_1,\lambda_2} \in C^\infty(\Sigma,(0,\infty))$ such that
the formula
$$R_{\lambda_1}=f_{\lambda_1,\lambda_2}R_{\lambda_2}$$
holds. Since $\Sigma$ is compact the function
$f_{\lambda_1,\lambda_2}$ attains a positive maximum and a
positive minimum on $\Sigma$ and we set
$$\underline{c}_{\lambda_1,\lambda_2}=
\frac{1}{\max f_{\lambda_1,\lambda_2}}, \quad
\overline{c}_{\lambda_1,\lambda_2}=\frac{1}
{\min f_{\lambda_1,\lambda_2}}$$
for which (\ref{prolam}) holds.

A stable hypersurface is called {\it tame} if there exists
$\lambda \in \Lambda(\Sigma,\omega)$ and a constant $c_\lambda$
such that for all $v \in X(\Sigma)$
\begin{equation}\label{hlam}
T_\lambda(v) \leq c_\lambda |\Omega(v)|.
\end{equation}

We remark that it follows from (\ref{hlam}) and the
theorem of Arzela-Ascoli that a tame stable hypersurface
is weakly tame. We refer to the constant $c_\lambda$
as a \emph{taming constant for $\lambda$}. We note
that if there exists a taming constant for
one $\lambda_1 \in \Lambda(\Sigma,\omega)$ then there
exists also a taming constant for every other
stabilizing one-form $\lambda_2 \in \Lambda(\Sigma,\omega)$.
Indeed, it follows from (\ref{prolam}) that
$$c_{\lambda_2}=\frac{c_{\lambda_1}}
{\underline{c}_{\lambda_1,\lambda_2}}$$
is a taming constant for $\lambda_2$.
We finally introduce the tameness condition for stable homotopies.

A stable homotopy $(\Sigma_t,\lambda_t)$ for
$t \in [0,1]$ is called \emph{tame}, if there
exists a taming constant $c> 0$ such that
$$T_{\lambda_{t}}(v) \leq c|\Omega(v)|, \quad
\forall\,\,t \in [0,1],\,\,\forall\,\,
v \in X(\Sigma_t).$$

\section{Rabinowitz Floer homology}\label{sec:RFH}

Rabinowitz Floer homology as the Floer homology
of Rabinowitz' action functional was, for the restricted
contact type case, introduced by the first two
authors in \cite{CF}. In this section we generalize as
far as possible the construction of Rabinowitz Floer homology
to the stable case. This generalization is not straightforward.
The compactness proof for the moduli space of gradient flow
lines has to be modified considerably in the stable case.
The crucial point is the occurence of a second Liapunov
functional on the moduli space of gradient flow lines
which is defined via a stabilizing 1-form. One could define
Rabinowitz Floer homology for a stabilized hypersurface
by using a filtration from both action functionals. However,
since the second action functional depends on the stabilizing
form, such homology groups might depend on the stabilizing 1-form.
To avoid this difficulty we only define Rabinowitz Floer homology
in the case of weakly tame stable hypersurfaces, in which case
this trouble does not occur.

As in the restricted contact type case the Rabinowitz Floer homology
groups have the property that they vanish if the hypersurface is
displaceable and that they coincide with the singular homology of
the hypersurface in the case that there are no contractible Reeb
orbits on the hypersurface. In particular, using Rabinowitz Floer
homology one can recover a theorem due to F.\,Schlenk, that on
displaceable stable hypersurfaces there always exists a contractible
Reeb orbit. Schlenk's theorem actually does not need the weakly tame
condition. And indeed, if one looks only at the local Rabinowitz
Floer homology around the action value zero, the general compactness
theorem for the moduli space of gradient flow lines of Rabinowitz
action functional suffices, to recover Schlenk's theorem also in the
not weakly tame case.

A further trouble in the stable case which one does not have to
worry about in the restricted contact type case, is the difficulty
that the stability condition is not open \cite{CFP} and hence
Rabinowitz action functional cannot be assumed to be generically
Morse-Bott. We therefore have to introduce an additional
perturbation which makes a generically perturbed Rabinowitz action
functional Morse. It is however not clear that compactness for the
moduli space of gradient flow lines for the perturbed Rabinowitz
action functional continues to hold. Nevertheless, one can partially
overcome this trouble be a procedure which is somehow reminiscent of
the Conley index in the finite dimensional case. For a fixed action
window it is possible to choose the perturbation so small that one
can get a boundary operator by counting only gradient flow lines of
the perturbed action functional which are sufficiently close to the
gradient flow lines of the unperturbed action functional. Since the
perturbation depends on the action window the drawback of this
construction is, that one cannot define the full Rabinowitz Floer
chain complex by using Novikov sums as in \cite{CF}. Instead one
first defines filtered Rabinowitz Floer homology groups and then
takes their inverse and direct limits. Because inverse and direct
limits do not necessarily commute one obtains in this way actually
two types of Rabinowitz Floer homology groups which we denote by
$\overline{RFH}$ and $\underline{RFH}$. One can show, see
\cite{CF2}, that in the restricted contact type case the Rabinowitz
Floer homology $RFH$ as defined in \cite{CF} via Novikov sums
coincides with $\overline{RFH}$. The two homology groups are
connected by a canonical homomorphism
$$\kappa \colon \overline{RFH} \to \underline{RFH}.$$
It is an open question if $\kappa$ is always an isomorphism.
The results of \cite{CF2} suggest that it should at least be
surjective. But even that is not clear yet, since in
\cite{CF2} the first two authors
made heavy use of a bidirect system of chain
complexes which we do not have in the stable case.

The Rabinowitz Floer homology groups are invariant under tame,
stable homotopies. We prove this via an adiabatic version of Floer's
continuation homomorphism. For this we need a compactness theorem
for gradient flow lines in the case that Rabinowitz action
functional is also allowed to depend on time.

This section is organised in the following manner. In
section~\ref{ss:action} we recall the definition of Rabinowitz
action functional, show how its critical points are given by Reeb
orbits and derive the gradient flow equation.

In section~\ref{ss:multiplier} the main ingredient for the
compactness proof of the moduli spaces of gradient flow lines is
established. Rabinowitz action functional is a Lagrange multiplier
action functional and the main difficulty is to obtain a uniform
bound on the Lagrange multiplier along gradient flow lines. Once
this bound is established, the compactness for the moduli spaces of
gradient flow lines follows from standard arguments well-known in
Floer theory.

In section~\ref{ss:existence} we give a new proof
of Schlenk's theorem about the existence of a
contractible Reeb orbit on a displaceable, stable
hypersurface. This proof can also be used to
derive the vanishing of Rabinowitz Floer homology
for displaceable, weakly tame, stable hypersurfaces.

In section~\ref{ss:approvable} we introduce a class of perturbations
of Rabinowitz action functional and show that for a generic
perturbation Rabinowitz action functional becomes Morse. We then
explain how for a fixed action window one can find small
perturbations such that the moduli space of gradient flow lines can
be written as the disjoint union of two closed parts where one part
(which we refer to as the essential part) is compact. We then
explain how the essential part can be used to define a boundary
operator.

In section~\ref{ss:RFH} we define the two
Rabinowitz Floer homology groups $\overline{RFH}$
and $\underline{RFH}$ for weakly tame, stable hypersurfaces.
We show that both homology groups vanish for displaceable
hypersurfaces and that both homology groups coincide
with the singular homology of the hypersurface in the
case that the hypersurface carries no contractible Reeb
orbit.

In section~\ref{ss:invariance} we finally establish invariance
of Rabinowitz Floer homology for tame, stable homotopies.

\subsection{Rabinowitz action functional}\label{ss:action}

We first give the general definition of Rabinowitz action functional
and its gradient flow, postponing technical details to later
subsections.

Consider a symplectically aspherical manifold $(V,\om)$ and a separating closed
hypersurface $\Sigma\subset V$.
Choose a Hamiltonian $\bar{H}$ with $\bar{H}^{-1}(0)=\Sigma$.
Such Hamiltonians exist since $\Sigma$ is assumed to be separating.
Denote by $\LL \subset C^\infty(S^1,V)$ the component of
contractible loops in the free loop space of $V$. {\em Rabinowitz
action functional}
$$
   \AA^{\bar{H}} \colon \LL\times \R \to \R
$$
is defined as
$$
   \AA^{\bar{H}}(v,\eta) := \int_0^1\bar v^*\om-\eta
   \int_0^1 \bar{H}(v(t))dt,\quad (v,\eta) \in \LL\times \R,
$$
where $\bar v\co D^2\to V$ is an extension of $v$ to the disk.
One may think of $\AA^{\bar{H}}$ as the Lagrange multiplier
functional
of the unperturbed action functional of classical mechanics also studied
in Floer theory to a mean value constraint of the loop. The critical
points of $\AA^{\bar{H}}$ satisfy
\begin{equation}\label{crit0}
\left. \begin{array}{cc}
\partial_t v(t)= \eta X_{\bar{H}}(v(t)),
& t \in \R/\Z, \\
\bar{H}(v(t))=0. & \\
\end{array}
\right\}
\end{equation}
Here we used the fact that $\bar{H}$ is invariant under its
Hamiltonian flow.

It is also useful to consider Rabinowitz action functional for
``weakly'' time dependent Hamiltonians. Here weakly means that the
Hamiltonian is just the product of a time independent Hamiltonian
with a function depending only on time, i.e.
$$H(t,y)=\chi(t)\bar{H}(y), \quad y \in V,\,\,t \in S^1.$$
Normalizing, we assume in addition that
\begin{equation}\label{chi}
\int_0^1 \chi(t)dt=1.
\end{equation}
We define Rabinowitz action functional $\mathcal{A}^H$ as before
with $\bar{H}$ replaced by $H$. The critical point equation then
becomes
\begin{equation}\label{crit1}
\left. \begin{array}{cc}
\partial_t v(t)= \eta \chi(t)X_{\bar{H}}(v(t)),
& t \in \R/\Z, \\
\bar{H}(v(t))=0, & \\
\end{array}
\right\}
\end{equation}
i.e.\,the new critical points can just be obtained by
reparametrisation of the previous ones. Moreover, the
action value remains constant.

We next describe a class of metrics on $\LL\times \mathbb{R}$.
Pick a smooth $\omega$-compatible almost
complex structure $J$ on $V$. For such a $J$ we define a
metric $m$ on $\LL\times \mathbb{R}$
for $(v,\eta) \in \LL\times \mathbb{R}$ and
$(\hat{v}_1,\hat{\eta}_1), (\hat{v}_2,\hat{\eta}_2) \in
T_{(v,\eta)}(\LL\times \mathbb{R})$ by
$$m\big((\hat{v}_1,\hat{\eta}_1),(\hat{v}_2,\hat{\eta}_2)\big)
=\int_0^1 \omega(\hat{v}_1,J(v)\hat{v}_2) dt +
\hat{\eta}_1 \cdot \hat{\eta}_2.$$
The gradient of $\mathcal{A}^H$ with respect to this metric is given by
$$
   \nabla \mathcal{A}^H=\nabla_m \mathcal{A}^H =
   \left(\begin{array}{cc}
   -J(v)\bigl(\partial_t v-\eta X_H(v)\bigr)\\
   -\int_0^1 H(t,v(\cdot,t)) dt
   \end{array}\right)\;.
$$
Thus (positive) {\em gradient flow lines} of $\nabla \mathcal{A}^H$
are solutions  $(v,\eta) \in C^\infty(\mathbb{R}\times S^1,V \times
\mathbb{R})$ of the following problem
\begin{equation}\label{flowline}
\left. \begin{array}{cc}
\partial_s v+J(v)(\partial_t v-\eta X_H(v))=0\\
\partial_s \eta+\int_0^1 H(t,v(\cdot,t)dt=0.
\end{array}
\right\}
\end{equation}

The boundary operator of Rabinowitz Floer homology counts gradient
flow lines connecting critical points of $\mathcal{A}^H$.
In order to prove that this is well defined one
has to show that the moduli spaces of gradient flow line are compact
modulo breaking. There are three difficulties one has to solve.
\begin{itemize}
 \item An $L^\infty$-bound on the loop $v \in \LL$.
 \item An $L^\infty$-bound on the Lagrange multiplier $\eta \in \mathbb{R}$.
 \item An $L^\infty$-bound on the derivatives of the loop $v$.
\end{itemize}
The first and the third point are standard problems in Floer theory
one knows how to deal with: The $L^\infty$-bound for the loop follows
from suitable assumptions on $(V,\om)$ such as convexity at infinity or
geometrical boundedness, and the derivatives can be controlled if
$V$ is symplectically aspherical, meaning that
$\om$ vanishes on $\pi_2(V)$, and hence there is no bubbling of
pseudo-holomorphic spheres.

The new feature is the bound on the Lagrange multiplier $\eta$, which
can only hold under some additional hypothesis on the hypersurface
$\Sigma$. It was derived in~\cite{CF} for $\Sigma$ of restricted
contact type. In the next subsection we will bound the Lagrange
multiplier provided $\Sigma$ is {\em stable}.

\subsection{Bound on the Lagrange multiplier}\label{ss:multiplier}

In this section we discuss the bound on the Lagrange multiplier.
In the restricted contact case this was carried out in
\cite{CF}. There the first two authors of this paper
used the fact that in the restricted contact case the action
value of Rabinowitz action functional at a critical point
coincides with the Lagrange multiplier. This is not true
anymore in the stable case. Nevertheless, using stability one
can define a modified Rabinowitz action functional whose
action value at critical points still is given by the Lagrange
multiplier. However, the modified action is not necessarily
a Liapunov function for gradient flow lines, so that the
proof in \cite{CF} cannot be mimicked in the stable case by
just using the modified version. The crucial observation
is that the difference of the two action functionals is
actually a Liapunov function for gradient flow lines and this
fact can be used to still get a bound on the Lagrange multiplier.

In section~\ref{ss:liapunov} we explain how the additional Liapunov
functional arises and why it can be used to get the bound on the
Lagrange multiplier. Technical details of section~\ref{ss:liapunov}
are postponed to section~\ref{at}. In section~\ref{ss:TDC} we also
derive a bound on the Lagrange multiplier for gradient flow lines,
when the Rabinowitz action functional itself is allowed to depend on
time. For this we need a short stable homotopy. Using the notion of
the stable pseudodistance, explained in  section~\ref{ss:pd}, we can
make precise what ``short'' means.

\subsubsection{A second Liapunov function for gradient flow lines}
\label{ss:liapunov}

$\,$From now on we assume that $\Sigma \subset V$ is a stable
hypersurface. We first explain the general strategy to obtain a
bound on the Lagrange multiplier. Assume that $\lambda \in
\Lambda(\Sigma,\omega)$ and let $R=R_\lambda$ be the Reeb vector
field of the stabilizing 1-form $\lambda$. We will use Hamiltonians
$\bar{H}$ such that $\bar{H}^{-1}(0)=\Sigma$ and in addition
$X_{\bar{H}}=R$ along $\Sigma$. To obtain the bound on the Lagrange
multiplier we also have to choose $\bar{H}$ carefully outside the
stable hypersurface. For that we have to assume that the stabilizing
1-form $\lambda$ is small enough, but we will give a detailed
explanation of this construction later in section~\ref{at}. The
additional condition implies that the critical point equations
(\ref{crit1}) are equivalent to
\begin{equation}\label{crit}
\left. \begin{array}{cc}
\partial_t v(t)= \eta \chi(t)R(v(t)),
& t \in \R/\Z, \\
v(t)\in \Sigma, & t \in \R/\Z, \\
\end{array}
\right\}
\end{equation}
i.e. $v$ is up to reparametrisation
a periodic orbit of the Reeb vector field on
$\Sigma$ with period $\eta$.
\footnote{The period $\eta$ may be negative or zero.
We refer in this paper
to Reeb orbits moved backwards as Reeb orbits with negative period
and to constant orbits as Reeb orbits of period zero.}

The bound on the Lagrange multiplier is derived by comparing
the Rabinowitz action functional to a different action functional. This
modified version of Rabinowitz action functional is
obtained using an extension of the stabilizing 1-form $\lambda$ on
$\Sigma$ to a compactly supported 1-form $\beta_\lambda$ on $V$.
We also postpone the precise construction of $\beta_\lambda$.
But given $\beta_\lambda$ we
define the auxiliary action functional $\widehat\AA^H\co \LL\times\R\to\R$
by
$$
   \widehat\AA^H(v,\eta) := \int_0^1\bar v^*d\beta_\lambda-\eta
   \int_0^1 H(t,v(t))dt,
$$
We will further use the difference of the two Rabinowitz action
functionals
$$
   \AA:=\AA^H-\widehat\AA^H\co \LL\times\R\to\R,\qquad
   (v,\eta)\mapsto\int_0^1\bar v^*(\om-d\beta_\lambda).
$$
Note that $\AA(v,\eta)$ is in fact independent of $\eta$. We further
use an $\omega$-compatible almost complex structure $J$ on $V$ which
has a special form on a stable tubular neighbourhood of $\Sigma$,
see section~\ref{at}. For that $J$ we consider the following two
bilinear forms on the tangent bundle $T(\LL\times\R)$:
\begin{align*}
   m\Bigl((\hat v_1,\hat\eta_1),(\hat v_2,\hat\eta_2)\Bigr)
   &:= \int_0^1\om(\hat v_1,J\hat v_2)dt + \hat\eta_1\hat\eta_2, \cr
   \widehat m\Bigl((\hat v_1,\hat\eta_1),(\hat v_2,\hat\eta_2)\Bigr)
   &:= \int_0^1d\beta_\lambda(\hat v_1,J\hat v_2)dt + \hat\eta_1\hat\eta_2.
\end{align*}
Here the bilinear form $m$ (which has already been defined in the
previous subsection) is positive definite. The main
point in the choice of $H$, $\beta_\lambda$ and $J$
is to make sure that the following Proposition becomes true.
\begin{prop}\label{mod}
If $(v,\eta) \in \LL \times \R$ and  $(\hat{v},\hat{\eta}) \in
T_{(v,\eta)}(\LL \times \R)$ then the following two assertions
hold
\begin{description}
 \item[(i)] $d\widehat\AA^H(v,\eta)(\hat v,\hat\eta) =
   \widehat m\Bigl(\nabla_m\AA^H(v,\eta),(\hat v,\hat\eta)\Bigr),$
 \item[(ii)] $
   (m-\widehat m)\Bigl((\hat v,\hat\eta),(\hat v,\hat\eta)\Bigr)
   \geq 0.$
\end{description}
\end{prop}
We prove Proposition~\ref{mod}
in section~\ref{at}.
Assertion (i) might be interpreted as
$$\nabla_{\widehat{m}}\widehat{\AA}^H=\nabla_m\AA^H.$$
However, we point out that the gradient with respect
to the bilinear form $\widehat{m}$ is not uniquely
determined since $\widehat{m}$ is not positive definite.
We can use Proposition~\ref{mod} to show that the
difference action functional $\AA$ is a Liapunov function
for the gradient flow lines of $\AA^H$.
\begin{corollary}\label{liap}
The functional $\AA$ is non-decreasing along gradient flow lines of
Rabinowitz action functional.
\end{corollary}
\begin{proof}
Let $w \in C^\infty(\mathbb{R},\LL\times \R)$ be a gradient
flow line of Rabinowitz action functional, i.e., a solution of
$$\partial_s w(s)=\nabla_m\AA^H\big(w(s)\big), \quad
s \in \mathbb{R}.$$
We estimate using Proposition~\ref{mod}
\begin{align*}
   \frac{d}{ds}\AA(w)
   &= d\AA^H(w)(\partial_s w) -
   d\widehat\AA^H(w)(\partial_s w) \cr
   &= m\Big(\nabla_m\AA^H(w),\nabla_m\AA^H(w)\Big) - \widehat
   m\Big(\nabla_m\AA^H(w),\nabla_m\AA^H(w)\Big) \cr
   &\geq 0.
\end{align*}
This concludes the proof of the Corollary.
\end{proof}

Since
the restriction of $\beta_\lambda$ to $\Sigma$ equals $\lambda$
we obtain using Stokes'
theorem and (\ref{crit}) the following period-action equality
for critical points $(v,\eta)$ of $\AA^H$,
$$
  \widehat\AA^H(v,\eta)=\eta.
$$
An elaboration of this observation is the second assertion
of the following Lemma which might be thought of
as a period-action inequality for almost Reeb orbits.
For the statement of the Lemma we use the abbreviation
$$||H||_\infty=\int_0^1\sup_V |H| dt=||\bar{H}||_\infty.$$

\begin{lemma}\label{abound}If the stabilizing 1-form $\lambda$
is small enough in the $C^0$-topology, then for every $\epsilon>0$
and for every $\gamma \in [0,1)$ we can choose the Hamiltonian
$\bar{H}$ in such a way that for every $\chi$ satisfying (\ref{chi})
the following two conditions hold for $H=\chi \bar{H}$.
\begin{description}
\item[(i)] $||H||_\infty \leq \gamma+\epsilon$.
\item[(ii)]For $(v,\eta) \in \mathcal{L}\times \mathbb{R}$ the
 following implication holds
$$
||\nabla\mathcal{A}^H(v,\eta)||\leq \frac{2\gamma}{3}
\quad \Rightarrow \quad |\eta| \leq \frac{1}{1-\gamma}
\Big|\widehat{\mathcal{A}}^H(v,\eta)\Big|+\frac{2\gamma}{3(1-\gamma)}.
$$
\end{description}
Here the norms and the gradient are taken with respect to $m$.
\end{lemma}
\begin{proof}
The Lemma follows from Lemma~\ref{abound2} below.
\end{proof}
\begin{proposition}\label{abo}
Suppose that the stabilizing 1-form $\lambda$ is small enough
and $H$ is as in Lemma~\ref{abound} for
$\epsilon>0$ and $\gamma \in (0,1-\epsilon]$.
Assume that
$w=(v,\eta) \in C^\infty(\mathbb{R},\mathcal{L}\times \mathbb{R})$
is a gradient flow line of  $\nabla \mathcal{A}^H$
for which there exist $a \leq b$ such that
\begin{equation}\label{aabound}
\mathcal{A}^H(w)(s),\,\,\mathcal{A}(w)(s) \in [a,b], \quad
\forall\,\,s \in \mathbb{R}.
\end{equation}
Then the $L^\infty$-norm of $\eta$ is bounded by
$$||\eta||_\infty \leq
\bigg(\frac{9-5\gamma}{4\gamma(1-\gamma)}\bigg)(b-a)
+\frac{2\gamma}{3(1-\gamma)}+\frac{9(b-a)}{4\gamma^2}\epsilon.$$
\end{proposition}
\begin{proof}
For $\sigma \in \mathbb{R}$ we set
$$\tau(\sigma)=\inf\bigg\{\tau \geq 0:
\big|\big|\nabla\mathcal{A}^H\big(w(\sigma+\tau)\big)\big|\big|
\leq \frac{2\gamma}{3}\bigg\}.$$
To obtain a bound on $\tau(\sigma)$ we estimate using the gradient
flow equation and (\ref{aabound})
\begin{eqnarray*}
b-a &\geq&
\mathcal{A}^H(w(\sigma))-\mathcal{A}^H(w(\sigma+\tau(\sigma))
\\
&=&\int_\sigma^{\sigma+\tau(\sigma)} \frac{d}{ds}\Big(
\mathcal{A}^H(w)\Big)ds\\
&=&\int_\sigma^{\sigma+\tau(\sigma)} d\mathcal{A}^H(w)(\partial_s w)
ds\\
&=&\int_\sigma^{\sigma+\tau(\sigma)}\big|\big|\nabla
\mathcal{A}^H(w)\big|\big|^2ds\\
&\geq&\bigg(\frac{2\gamma}{3}\bigg)^2 \tau(\sigma)
\end{eqnarray*}
from which we deduce
$$\tau(\sigma) \leq \frac{9(b-a)}{4\gamma^2}.$$
Furthermore, for every $s\in \mathbb{R}$ the modified version
of Rabinowitz action functional can be estimated from above
using (\ref{aabound}) again by
$$\big|\widehat{\mathcal{A}}^H(w(s))\big|=\big|
\mathcal{A}^H(w(s))-\mathcal{A}(w(s))\big| \leq b-a.$$
The above two inequalities together with Lemma~\ref{abound} and the
equation $\partial_s \eta(s)=-\int_0^1H(t,v(s,t))dt$ following
from the gradient flow equation imply
\begin{align*}
   |\eta(\sigma)| &\leq |\eta(\sigma+\tau(\sigma))|
   +\int_\sigma^{\sigma+\tau(\sigma)}|\partial_s \eta(s)| ds \cr
   &\leq \frac{1}{1-\gamma}\big(b-a\big)+\frac{2\gamma}{3(1-\gamma)}
   +(\gamma+\epsilon)\tau(\sigma) \cr
   &\leq \bigg(\frac{9-5\gamma}{4\gamma(1-\gamma)}\bigg)(b-a)
+\frac{2\gamma}{3(1-\gamma)}+\frac{9(b-a)}{4\gamma^2}\epsilon.
\end{align*}
Since $\sigma \in \mathbb{R}$ was arbitrary the proposition
follows.
\end{proof}
\begin{corollary}\label{obo}
Suppose that the stabilizing 1-form is small enough. Then for given
$a \leq b$ the Hamiltonian $\bar{H}=\bar{H}(b-a)$ can be chosen in
such a way that for every $\chi$ satisfying (\ref{chi}) and for
every gradient flow line $w=(v,\eta) \in
C^\infty(\mathbb{R},\mathcal{L}\times \mathbb{R})$ of $\nabla
\mathcal{A}^H$ which satisfies (\ref{aabound}) the $L^\infty$-norm
of $\eta$ is bounded by
$$||\eta||_\infty \leq \bigg(\frac{5}{2}\bigg)^2(b-a)
+2.$$
\end{corollary}
\begin{proof}
Choose $\gamma=\frac{3}{5}$ and
$\epsilon \leq \frac{4}{25(b-a)}$ and let
$\bar{H}=\bar{H}(b-a)$ be the corresponding Hamiltonian. With these choices
the Corollary follows from Proposition~\ref{abo}.
\end{proof}
\begin{remark}
The Corollary is somehow the optimal Corollary from
Proposition~\ref{abo}, since the function
$$f \in C^\infty((0,1),\mathbb{R}), \quad
\gamma \mapsto \frac{9-5\gamma}{4\gamma(1-\gamma)}$$
attains at the point $\gamma=\frac{3}{5}$ its unique
minimum.
\end{remark}

\subsubsection[Admissible tuples]{Admissible tuples}\label{at}

In this section we prove Proposition~\ref{mod} and
Lemma~\ref{abound}.

For $\lambda \in \Lambda(\Sigma,\omega)$ denote by
$\mathcal{I}(\Sigma,\lambda)$ the space of
$\om_\Sigma$-compatible complex structures on the bundle
$\xi_\lambda=\ker\lambda$. Note that actually
$\mathcal{I}(\Sigma,\lambda)$ only depends on the ray
of $\lambda \in \Lambda(\Sigma,\omega)$ since
$\xi_{\sigma\lambda}=\xi_\lambda$ for $\sigma>0$.
For $I \in \mathcal{I}(\Sigma,\lambda)$
we introduce the following quantity
$$\kappa(\lambda,I):=\sup_{\substack{w \in \xi_\lambda\\w \neq 0}}
\frac{|d\lambda(w,I w)|}{\omega(w,I w)}\geq 0.$$
We define the upper semicontinuous function
$$\kappa \colon \Lambda(\Sigma,\omega) \to [0,\infty), \quad
\lambda \mapsto \inf_{I \in \mathcal{I}(\Sigma,\lambda)}
\kappa(\lambda,I).$$
Note that $\kappa$ has the following scaling behaviour
$$\kappa(\sigma\lambda)=\sigma\kappa(\lambda), \quad \sigma>0,\,\,\lambda
\in \Lambda(\Sigma,\omega).$$
Recall from section~\ref{sec:stable} the notion of a stable
tubular neighbourhood.
We further introduce the lower semicontinuous function
$$\varrho \colon \Lambda(\Sigma,\omega) \to (0,\infty)$$
measuring the size of stable tubular neighbourhoods which is defined
for $\lambda \in \Lambda(\Sigma,\omega)$ by
$$\varrho(\lambda)=\sup\big\{\rho: (\rho,\psi) \in
\mathcal{T}(\Sigma,\lambda)\big\}.$$
The scaling behaviour for the function $\varrho$ is given by
$$\varrho(\sigma \lambda)=\frac{1}{\sigma}\varrho(\lambda).$$
\begin{definition}
A stabilizing 1-form $\lambda \in \Lambda(\Sigma,\omega)$ is called
\emph{small} if
$$\max\bigg\{\kappa(\lambda),\frac{1}{\varrho(\lambda)}
\bigg\} <1.$$
We abbreviate by
$\Lambda_s(\Sigma,\omega) \subset \Lambda(\Sigma,\omega)$
the subset of small stabilizing 1-forms.
\end{definition}
\begin{remark}
It follows from the scaling behaviour of the functions
$\kappa$ and $\varrho$ that for every $\lambda \in \Lambda(\Sigma,\omega)$
there exists $\tau>0$ such that $\tau \lambda$ is small.
\end{remark}

We introduce three spaces of functions. Given $\rho>1$
the first space
$$\mathcal{F}_1(\rho) \subset C^\infty(\mathbb{R},\mathbb{R})$$
consists of functions $\phi$
for which there exists $\rho_0 \in (1,\rho)$ such
that the following conditions are met
\begin{equation}\label{cutoff}
\left.\begin{array}{cc}
\phi(r)=r+1, & r \in [-\rho_0,\rho_0]\\
\phi'(r)\leq 1, & r \in \mathbb{R}\\
\mathrm{supp}(\phi) \subset (-\rho,\rho).&
\end{array}\right\}
\end{equation}
Given $\gamma \geq 0$ and $\epsilon>0$
the second space
$$\mathcal{F}_2(\gamma,\epsilon) \subset C^\infty(\mathbb{R},
\mathbb{R})$$
consists of functions $h$ for which the following holds
\begin{equation}\label{mono}
\left.\begin{array}{cc}
h(r)=r, & r \in [-\gamma,\gamma]\\
0 \leq h'(r) \leq 1+\epsilon, & r \in \mathbb{R}\\
h(r)=\gamma+\epsilon & r \geq \gamma+\epsilon\\
h(r)=-\gamma-\epsilon & r \leq -\gamma-\epsilon
\end{array}\right\}.
\end{equation}
The third space
$$\mathcal{F}_3 \subset C^\infty(S^1,\mathbb{R})$$
consists of functions
$\chi \in C^\infty(S^1,\mathbb{R})$
satisfying
$$\int_0^1 \chi(t)dt=1.$$
\begin{remark} All three spaces $\mathcal{F}_1(\rho)$,
$\mathcal{F}_2(\gamma,\epsilon)$
and $\mathcal{F}_3$ are convex and nonempty.
\end{remark}
\begin{definition}
A stable tubular neighbourhood $\tau=(\rho,\psi) \in
\mathcal{T}(\Sigma,\lambda)$ is called \emph{large}
if $\rho>1$. We abbreviate by
$\mathcal{T}_\ell(\Sigma,\lambda) \subset
\mathcal{T}(\Sigma,\lambda)$ the subset of large
stable tubular neighbourhoods.
\end{definition}
We remark that if $\lambda$ is a small stabilizing
1-form, then the space of large tubular neighbourhoods
is not empty by definition of small.
\begin{definition} Given $\epsilon>0$ and
$\gamma \in [0,1-\epsilon]$
a \emph{$(\gamma,\epsilon)$-admissible quintuple}
for the stable hypersurface
$\Sigma$ in $V$
$$\alpha=(\lambda,\tau,\phi,h,\chi)$$
consists of $\lambda \in \Lambda_s(\Sigma,\omega)$,
$\tau=(\rho,\psi) \in \mathcal{T}_\ell(\Sigma,\lambda)$,
$\phi \in \mathcal{F}_1(\rho)$,
$h \in \mathcal{F}_2(\gamma,\epsilon)$, and
$\chi \in \mathcal{F}_3$.
\end{definition}
\begin{remark}
We sometimes omit the reference to $\gamma$ and $\epsilon$
and refer to a $(\gamma,\epsilon)$-admissible quintuple
as an \emph{admissible quintuple}.
\end{remark}
We abbreviate by
$$\mathfrak{A}^{\gamma,\epsilon}=
\mathfrak{A}^{\gamma,\epsilon}(\Sigma,V)$$
the space of all $(\gamma,\epsilon)$-admissible quintuples for
$\Sigma$ in $V$.
We denote by
$$\pi \colon \mathfrak{A}^{\gamma,\epsilon} \to
\Lambda_s(\Sigma,\omega)$$
the projection to the first factor. Note that for
each $\lambda \in \Lambda_s(\Sigma,\omega)$
the fiber
$$\mathfrak{A}^{\gamma,\epsilon}_\lambda=\pi^{-1}(\lambda)$$
is nonempty.
As we have seen in Corollary~\ref{obo} and the Remark
following it the somehow optimal choice to bound the Lagrange
multiplier is $\gamma=\frac{3}{5}$. Hence for
$\epsilon \in (0,\frac{2}{5}]$ we set
$$\mathfrak{A}^\epsilon=\mathfrak{A}^{\frac{3}{5},\epsilon}.$$
We further abbreviate
$$\mathfrak{A}=\bigcup_{\epsilon \in (0,\frac{2}{5}]}
\mathfrak{A}^\epsilon.$$

For a stable tubular neighbourhood $\tau=(\rho,\psi)
\in \mathcal{T}(\Sigma,\lambda)$ we denote by
$$U_\tau=\psi\big([-\rho,\rho]\times \Sigma)$$
its image in $V$.
We introduce a compactly supported
1-form on $V$ for an admissible quintuple
which extends the 1-form $\lambda$ on $\Sigma$ by
$$\beta_\alpha(y)=\left\{
\begin{array}{cc}
\phi(r)\lambda(x) & y=\psi(x,r) \in U_\tau\\
0 & y \in V \setminus U_\tau.
\end{array}
\right.$$
We will also sometimes suppress some of the dependencies of
$\beta_\alpha$ and write
$$\beta=\beta_\lambda=\beta_\alpha.$$

Using the notion of admissible quintuples we are now in position to
give the precise definition of the Rabinowitz action functional. For
a $(\gamma,\epsilon)$-admissible quintuple
$\alpha=(\lambda,\tau,\phi,h,\chi)$ with $\tau=(\rho,\psi)$ we
define $H \in C^\infty(V\times S^1,\mathbb{R})$ for $y \in V$ by the
condition
$$H(y,t)=H_\alpha(y,t)=\left\{\begin{array}{cc}
\chi(t)h(r) & y=\psi(x,r) \in U_\tau,\,\, t\in S^1\\
H(y,t)=\pm\chi(t)(\gamma+\epsilon) & y \in V \setminus U_\tau,
\,\,t \in S^1.
\end{array}\right.$$
For later reference we introduce the abbreviation
$$\bar{H}(y)=\int_0^1 H(t,y).$$
We sometimes use the notation
$$\mathcal{A}^\alpha=\mathcal{A}^H,\quad
\widehat{\mathcal{A}}^\alpha=\widehat{\mathcal{A}}^H,\quad
\mathcal{A}_\alpha=\mathcal{A}$$
to make the dependency of the three action functionals on the
parameters explicit.

To define the bilinear form $m$ and $\widehat{m}$ we
have to pick  a suitable subspace of
$\omega$-compatible almost complex structures on $V$.
We first make the following definition.
\begin{definition}
If $\lambda \in \Lambda(\Sigma,\omega)$ then the space
of \emph{small} $\omega_\Sigma$-compatible almost
complex structures on $\xi_\lambda$ is defined as
$$\mathcal{I}_s(\Sigma,\lambda)=
\{I \in \mathcal{I}(\Sigma,\lambda):
\kappa(\lambda,I)\leq1\}.$$
\end{definition}
We note that if $\lambda$ is small, then by definition the
space $\mathcal{I}_s(\Sigma,\lambda)$ is nonempty.

Recall from the introduction that a
compatible almost complex structure $J$ is called {\em geometrically
  bounded} if $\omega(\cdot,J\cdot)$ is a complete Riemannian metric
with bounded sectional curvature and injectivity radius bounded away
from zero. 

If $\alpha$ is an admissible quintuple, and
$I \in \mathcal{I}_s(\Sigma, \pi(\alpha))$
we denote by $\mathcal{J}(\alpha,I)$ the subspace of
$\omega$-compatible geometrically bounded almost complex structures
$J$ on $V$ which split on the stable tubular neighbourhood
$U_\tau=\psi\big([-\rho,\rho] \times \Sigma\big)
\subset V$
with respect to the decomposition
$$TV|_{U_\tau}=\xi_\lambda \oplus \xi_\lambda^\omega$$
as
$$J=\left(\begin{array}{cc}
I & 0\\
0 & J_0
\end{array}\right)$$
where
$J_0$ is the standard complex structure on the symplectic
complement of $\xi_\lambda$ spanned by the Reeb vector field $R_\lambda$
and $\partial/\partial r$. We further set
$$\mathcal{J}(\alpha)=
\bigcup_{I \in \mathcal{I}_s(\Sigma,\pi(\alpha))}
\mathcal{J}(\alpha,I).$$

The following Lemma is the main technical point to establish
Proposition~\ref{mod}.

\begin{lemma}
Assume that $\alpha=(\lambda,\tau,\phi,h,\chi)$
is an admissible quintuple
for the stable hypersurface $\Sigma$ in $V$,
and $J \in \mathcal{J}(\alpha)$. Then for every $w \in TV$ the following
estimate holds
\begin{equation}\label{loine}
   d\beta_\alpha(w,Jw)\leq \om(w,Jw).
\end{equation}
\end{lemma}

\textbf{Proof: } The inequality (\ref{loine})
is clear on the complement of the stable tubular neighbourhood $U_\tau$ since
there $d\beta_\alpha$ vanishes and
$\omega(\cdot, J\cdot)$ is positive
definite. For $\tau=(\rho,\psi)$ we identify via $\psi$ the neighbourhood
$U_\tau$ of $\Sigma$ in $V$ symplectically with
$[-\rho,\rho]\times \Sigma$ and $\omega=\omega_\Sigma+d(r\lambda)$.
Then the exterior derivative of $\beta_\alpha$ is given by the formula
$$
   d\beta_\alpha|_{U_\tau}=
 \phi(r)d\lambda + \phi'(r)dr\wedge\lambda.
$$
For  $u=(r,x) \in U_\tau$ and $w \in T_u V$
we write $w=w_1+w_2$ with respect to the decomposition
$TV|_{U_\tau}=\xi_\lambda \oplus \xi_\lambda^\omega$. Note that since
$J \in \mathcal{J}(\alpha)$ there exists
$I \in \mathcal{I}_s(\Sigma,\lambda)$
such that
$$Jw=I w_1+J_0 w_2.$$
We first observe that the following implications follow from (\ref{cutoff})
$$\phi(r) \geq 0\,\,\Rightarrow\,\,\phi(r) \leq r+1, \quad
\phi(r) \leq 0\,\,\Rightarrow\,\,\phi(r) \geq r+1.$$
Moreover, since $I \in \mathcal{I}_s(\Sigma,\lambda)$ it
follows that
$$|d\lambda(w_1,Iw_1)| \leq \omega_\Sigma(w_1,Iw_1).$$
Using further $\phi' \leq 1$ following from (\ref{cutoff}) we estimate
\begin{eqnarray*}
d\beta_\alpha(w,Jw)&=&\phi(r)d\lambda(w_1,I w_1)+\phi'(r) dr
\wedge\lambda(w_2,J_0w_2)\\
&\leq& \max\big\{(r+1)d\lambda(w_1,Iw_1),0\big\}+\omega(w_2,J w_2)\\
&=&\max\big\{rd\lambda(w_1,Iw_1)+\omega_\Sigma(w_1,Iw_1),0\big\}
+\omega(w_2,J w_2)\\
&=&\omega(w_1, Jw_1)
+\omega(w_2,J w_2)\\
&=&\omega(w,Jw).
\end{eqnarray*}
This proves the Lemma. \hfill $\square$
\\ \\
\textbf{Proof of Proposition~\ref{mod}}
The inequality (\ref{loine}) implies that
$$
   (m-\hat m)\Bigl((\hat v,\hat\eta),(\hat v,\hat\eta)\Bigr)
   =  (m-\hat m)\Bigl((\hat v,0),(\hat v,0)\Bigr)\geq 0.
$$
Recall that the gradient of $\AA^H$ with respect to the metric $m$,
defined by the equation
$$
   d\AA^H(v,\eta)(\hat v,\hat\eta) =
   m\Bigl(\nabla_m\AA^H(v,\eta),(\hat v,\hat\eta)\Bigr),
$$
is given by
$$
   \nabla_m\AA^H(v,\eta) = \Bigl(-J(v)\bigl(\dot v-\eta
   X_H(v)\bigr),\int_0^1H(v)dt\Bigr).
$$
Here the Hamiltonian vector field $X_H$, defined by the equation
$dH=-i_{X_H}\om$, equals zero outside the region $[-\rho,\rho]\times
\Sigma$ and on this region it is given by $X_H=\chi(t)h'(r)R$. Hence
the stability condition $i_Rd\lambda=0$ together with the first
equality in (\ref{cutoff}) yields $dH=-i_{X_H}d\beta_\alpha$, which
in turn implies that
$$
   d\widehat\AA^H(v,\eta)(\hat v,\hat\eta) =
   \hat m\Bigl(\nabla_m\AA^H(v,\eta),(\hat v,\hat\eta)\Bigr).
$$
This proves the Proposition. \hfill $\square$

We finally show Lemma~\ref{abound}
by proving the following Lemma
whose proof is an elaboration of the proof
of \cite[Proposition 3.1]{CF}.
To see how this Lemma actually implies Lemma~\ref{abound}
it only remains to note if $\alpha$ is
a $(\gamma,\epsilon)$-admissible quintuple, then
$$||H_\alpha||_\infty=\gamma+\epsilon.$$
\begin{lemma}\label{abound2}
Assume that $\alpha=(\lambda,\tau,\phi,h,\chi)$ is a
$(\gamma,\epsilon)$-admissible quintuple
with $\tau=(\rho,\psi)$. Then
the following implication holds
$$
||\nabla\mathcal{A}^\alpha(v,\eta)||\leq \frac{2\gamma}{3}
\quad \Rightarrow \quad |\eta| \leq \frac{1}{1-\gamma}
\bigg(\Big|\widehat{\mathcal{A}}^\alpha(v,\eta)\Big|+
||\nabla \mathcal{A}^\alpha(v,\eta)||\bigg).
$$
\end{lemma}
\textbf{Proof: }For $\rho' \leq \rho$ we use the following notation
for the stable tubular subneighbourhood of $\tau$
$$\tau_{\rho'}=(\rho',\psi|_{[-\rho',\rho']\times \Sigma}).$$
The Lemma is proved in two steps.
\\ \\
\textbf{Step\,1: }\emph{Assume that $v(t)$ for every $t \in S^1$
is contained in the stable tubular neighbourhood
$U_{\tau_\gamma}$. Then the inequality
for $|\eta|$ holds.}
\\ \\
We first note that the Hamiltonian vector field and the one-form
satisfy on $U_{\tau_\gamma}$ the following relation
$$\beta_\alpha(X_H)|_{S^1 \times U_{\tau_\gamma}}=
H|_{S^1 \times U_{\tau_\gamma}}+\chi.$$ In the following estimate we
denote by $||\,||_1$ and $||\,||_2$ the $L^1$ respectively
$L^2$-norm on the circle.
\begin{eqnarray*}
|\widehat{\mathcal{A}}^\alpha(v,\eta)|
&=&\bigg|\int_0^1 v^*\beta_\alpha-\eta\int_0^1 H(t,v(t))dt\bigg|\\
&=&\bigg|\eta\int_0^1 \beta_\alpha
(X_H(v))dt+\int_0^1\beta_\alpha\big(
\partial_t v-\eta X_H(v)\big)dt-\eta\int_0^1 H(t,v(t))dt\bigg|\\
&\geq&\bigg|\eta\int_0^1 \chi
dt\bigg|-\bigg|\int_0^1\beta_\alpha\big(
\partial_t v-\eta X_H(v)\big)dt\bigg|\\
&\geq&|\eta|-(1+\gamma)||\partial_t v-\eta X_H(v)||_1\\
&\geq&|\eta|-(1+\gamma)||\partial_t v-\eta X_H(v)||_2\\
&\geq&|\eta|-(1+\gamma)||\nabla \mathcal{A}^\alpha(v,\eta)||.
\end{eqnarray*}
This proves Step\,1.
\\ \\
\textbf{Step\,2: }\emph{Assume that $||\nabla\mathcal{A}^\alpha(v,\eta)||
\leq \frac{2\gamma}{3}$. Then
$v(t) \in U_{\tau_\gamma}$ for every
$t \in S^1$.}
\\ \\
We argue by contradiction and exclude the following two cases.
\\ \\
\textbf{Case\,1: }\emph{There exists $t_0, t_1 \in S^1$ such that
$v(t_0) \in U_{\tau_{\frac{2\gamma}{3}}}$ and
$v(t_1) \in V \setminus
U_{\tau_\gamma}$.}
\\ \\
\textbf{Case\,2: }\emph{For all $t \in S^1$ it holds that $v(t) \in
V \setminus U_{\tau_\gamma}$.}
\\ \\
We first observe that in Case\,1 there
exist two disjoint intervals $I^1=[s^1_0,s^1_1] \subset S^1$
and $I^2=[s^2_0,s^2_1]$ such that
$$v(s^1_0)\in \partial U_{\tau_\gamma},\quad
v(s^1_1) \in \partial U_{\tau_{\frac{2\gamma}{3}}}, \quad
v(s^2_0) \in \partial U_{\tau_{\frac{2\gamma}{3}}}, \quad
v(s^2_1) \in \partial U_{\tau_\gamma},$$
and
$$v(s) \in U_{\tau_\gamma} \setminus
U_{\tau_{\frac{2\gamma}{3}}},\,\,s \in I^1 \cup I^2.$$
Identifying $U_{\tau_\gamma}$ with
$[-\gamma,\gamma] \times \Sigma$ we write for
$s \in I^1 \cup I^2$
$$v(s)=\big(r(s),u(s)\big) \in [-\gamma,\gamma]
\times \Sigma.$$
We estimate
\begin{eqnarray*}
||\nabla \mathcal{A}^\alpha(v,\eta)||
&>&||\partial_t v-\eta \chi X_H(v)||_2\\
&\geq&||\partial_t v-\eta \chi X_H(v)||_1\\
&\geq&\int_{s^1_0}^{s^1_1}|\partial_t r(s)| ds
+\int_{s^2_0}^{s^2_1}|\partial_t r(s)|ds\\
&\geq& \frac{2\gamma}{3}.
\end{eqnarray*}
This contradiction excludes Case\,1.

To exclude Case\,2 we estimate
$$||\nabla \mathcal{A}^\alpha|| \geq
\bigg|\int_0^1 H(t,v(t))dt\bigg| > \frac{2\gamma}{3}.$$ This
inequality contradicts the assumption. Thus Step\,2 and hence the
Lemma follow. \hfill $\square$

\subsubsection{The stable pseudo-distance}\label{ss:pd}

We later prove invariance of Rabinowitz Floer homology
via an adiabatic homotopy argument. For that we
need short homotopies of stable hypersurfaces. In order
to say what ``short'' means we introduce the stable
pseudo-distance.

We define the stable pseudo-distance as the
infimum of the length of paths
of stable quadruples between two stable hypersurfaces.
To begin with here is the definition of a stable
quadruple.
\begin{definition}
A \emph{stable quadruple} in $V$
$$\mathfrak{S}=(\Sigma,\lambda,\tau,I)$$
consists of a stable hypersurface $\Sigma \subset V$,
a small stabilizing 1-form
$\lambda \in \Lambda_s(\Sigma,\omega)$, a large
tubular neighbourhood
$\tau \in \mathcal{T}_\ell(\Sigma,\lambda)$ and
a small $\omega_\Sigma$-compatible almost complex
structure $I \in \mathcal{I}_s(\Sigma,\lambda)$. We also
refer to a stable quadruple as above as a \emph{stable
quadruple for the stable hypersurface $\Sigma$}.
\end{definition}
Let $\mathfrak{P}=\{\mathfrak{S}_\zeta\}=\{(\Sigma_\zeta,
\lambda_\zeta,\tau_\zeta,I_\zeta)\}$ for
$\zeta \in [0,1]$ be a smooth path
of stable quadruples. To be precise, we require
for a smooth path $\tau_\zeta$ of large
stable tubular neighbourhoods that for a fixed
$\rho>1$ we have a smooth family of maps
$\psi_\zeta \colon [-\rho,\rho] \times \Sigma \to V$
where $\Sigma$ is a fixed manifold diffeomorphic
to every $\Sigma_\zeta \subset V$ such that
$\tau_\zeta=(\rho,\psi_\zeta) \in \mathcal{T}_\ell(\Sigma_\zeta,
\lambda_\zeta)$ for every $\zeta \in [0,1]$.
We associate to such a path a distance in
the following way. For fixed $\zeta \in [0,1]$ define a vector
field on $[-\rho,\rho]\times \Sigma$ by
$$X_{\psi_\zeta}(r,x)=\psi_\zeta^*\frac{d}{d\theta}
\bigg|_{\theta=0}\psi_{\zeta+\theta}(r,x), \quad
(r,x) \in [-\rho,\rho]\times \Sigma.$$
We refer to the expression
$$\mathfrak{V}_{\mathfrak{P}}(\zeta)=$$
$$\rho\Big(\max_{[-\rho,\rho]\times \Sigma_\zeta}\big|
\omega(X_{\psi_\zeta},R_{\lambda_\zeta})\big|+
\max_{[-\rho,\rho]\times \Sigma_\zeta}\big|d \iota_{X_{\psi_\zeta}}
\lambda_\zeta(R_{\lambda_\zeta})\big|
+\max_{\Sigma_\zeta}\big|\dot{\lambda}_\zeta(R_{\lambda_\zeta})\big|\Big)$$
as the \emph{speed} of the path $\mathfrak{P}$ at $\zeta$. We define
the \emph{length} of the path $\mathfrak{P}$ by the formula
$$\Delta(\mathfrak{P})=
\int_0^1\mathfrak{V}_{\mathfrak{P}}(\zeta)d\zeta.$$
The stable pseudo-distance between two stable hypersurfaces
$\Sigma_0$ and $\Sigma_1$
$$\Delta(\Sigma_0,\Sigma_1) \in [0,\infty]$$
is defined as the infimum of the length of all paths of stable
quadruples whose endpoints are stable quadruples for
the stable hypersurfaces $\Sigma_0$ respectively $\Sigma_1$.
Here we understand that the stable pseudo-distance is infinite
if there is no such path.

\subsubsection {The time-dependent case}\label{ss:TDC}

In this section we establish a bound on the Lagrange
multiplier for gradient flow lines, when the
Rabinowitz action functional is allowed to depend itself
on time.

Given two stable hypersurfaces $\Sigma^-, \Sigma^+ \subset V$
and $\alpha^\pm \in \mathfrak{A}(\Sigma^\pm,V)$
we denote by
\begin{equation}\label{mah}
\mathcal{H}(\alpha^-,\alpha^+)
\subset C^\infty(V \times \mathbb{R})
\end{equation}
the space of time dependent Hamiltonians $H \in
C^\infty(V \times \mathbb{R})$ for which there exists
$R>0$ with the property that the family of
Hamiltonians $H_s=H(\cdot,s)$ becomes
constant for $|s| \geq R$, and
$$H_s=H_{\alpha^\pm} \in C^\infty(V), \quad \pm s \geq R.$$
\begin{theorem}\label{tdc}
For $\epsilon>0$ there exists a constant $\Delta(\epsilon)>0$
with the following property. Assume that
$\Sigma^-,\Sigma^+ \subset V$ are stable hypersurfaces such that
$$\Delta(\Sigma^-,\Sigma^+) \leq \Delta(\epsilon).$$
Then there exist admissible
quintuples $\alpha^\pm  \in \mathfrak{A}(\Sigma^\pm,V)$,
a time dependent Hamiltonian
$H \in \mathcal{H}(\alpha^-,\alpha^+)$, and a time-dependent metric
$m=\{m_s\}_{s \in \mathbb{R}}$ on $\mathcal{L}\times \mathbb{R}$
which is constant for $|s|$ large
such that the following
holds true. For every $a \leq b$ and for every flow line
$w=(v,\eta) \in C^\infty(\mathbb{R},\mathcal{L}\times \mathbb{R})$
of the time dependent gradient $\nabla_{m_s}\mathcal{A}^{H_s}$
which converges asymptotically
$\lim_{s \to \pm \infty}w(s)=
w^\pm$ to critical points of $\mathcal{A}^{\alpha^\pm}$
satisfying
$$\mathcal{A}^{\alpha^\pm}(w^\pm),\mathcal{A}_{\alpha^\pm}(w^\pm) \in
[a,b],$$
the following $L^\infty$-estimate holds for
$\eta$
$$||\eta||_\infty \leq \bigg(\frac{5}{2}+\epsilon\bigg)^2(b-a)+2.$$
\end{theorem}
\begin{proof}
Assume that $\Sigma^-, \Sigma^+$ are stable
hypersurfaces whose stable pseudodistance
is $\Delta$. Hence for every $\epsilon_0>0$ there exists
a path $\mathfrak{P}=\mathfrak{P}_{\epsilon_0}
=\{\mathfrak{S}_\zeta\}$ of stable quadruples
$\mathfrak{S}_\zeta=(\Sigma_\zeta,\lambda_\zeta,
\tau_\zeta,I_\zeta)$ for
$\zeta \in [0,1]$ such that
$$\Delta(\mathfrak{P})\leq\Delta+\epsilon_0$$
and
$$\Sigma_0=\Sigma^-, \quad \Sigma_1=\Sigma^+.$$
Let $\rho>1$ be such that $\tau_\zeta=(\rho,\psi_\zeta)$.
Recall from section~\ref{at} the spaces $\mathcal{F}_1(\rho)$,
$\mathcal{F}_2(\gamma,\epsilon)$ and $\mathcal{F}_3$. Choose
$$\phi \in \mathcal{F}_1(\rho), \quad
\chi \in \mathcal{F}_3.$$
Moreover, assume that $\epsilon_0 \leq \frac{2}{5}$ and pick
$$h \in \mathcal{F}_2\bigg(\frac{3}{5},\epsilon_0\bigg).$$
Then the family
$$\alpha_\zeta=(\lambda_\zeta,\tau_\zeta,\phi,h,\chi),
\quad \zeta \in [0,1]$$
is a smooth family of
$\big(\frac{3}{5},\epsilon_0\big)$-admissible
quintuples for $\Sigma_\zeta$. Choose further a
smooth function
$$\zeta_1 \in C^\infty(\mathbb{R},[0,1])$$
such that
$$0 \leq \zeta_1'(s) \leq 1, \quad s \in \mathbb{R}$$
and
$$\zeta_1(s)=\left\{\begin{array}{cc}
0 & s \leq -1\\
1 & s \geq 1.
\end{array}\right.$$
For $R>0$ define $\zeta_R \in C^\infty(\mathbb{R},[0,1])$
by
$$\zeta_R(s)=\zeta_1\bigg(\frac{s}{R}\bigg), \quad s \in \mathbb{R}.$$
Note that
$$||\zeta_R'|| \leq \frac{1}{R}.$$
Set
$$\alpha^-=\alpha_0, \quad \alpha^+=\alpha_1$$
and define
$$H^R \in \mathcal{H}(\alpha^-,\alpha^+)$$
by
$$H^R_s=H_{\alpha_{\zeta_R(s)}}, \quad s\in \mathbb{R}.$$
To define the time-dependent metric $m=\{m_s\}_{s\in \mathbb{R}}$
choose a smooth family
$$J_\zeta \in \mathcal{J}(\alpha_\zeta,I_\zeta).$$
If $(v,\eta) \in \mathcal{L}\times \mathbb{R}$ and
$(\hat{v}_1,\hat{\eta}_1),(\hat{v}_2,\hat{\eta}_2)
\in T_{(v,\eta)}(\mathcal{L}\times \mathbb{R})$
we set
$$m_s^R\Big((\hat{v}_1,\hat{\eta}_1),(\hat{v}_2,\hat{\eta}_2)
\Big):=\int_0^1 \omega(\hat{v}_1,J_{\zeta_R(s)}\hat{v}_2)dt
+\hat{\eta}_1 \hat{\eta}_2.$$
Now assume that $w=(v,\eta) \in C^\infty(\mathbb{R},
\mathcal{L}\times \mathbb{R})$ is a gradient flow line
of $\nabla_{m_s^R}\mathcal{A}^{H^R_s}$. We show
in four steps that
for $R$ large enough and $\Delta$ small
enough the required $L^\infty$-bound
on $\eta$ holds.
\\ \\
\textbf{Step\,1: }\emph{For every $\sigma \in \mathbb{R}$ the
following inequalities hold}
\begin{equation}\label{1s}
a-2(\Delta+\epsilon_0)||\eta||_\infty \leq
\mathcal{A}^{H^R_\sigma}(w(\sigma))
\leq b+2(\Delta+\epsilon_0)||\eta||_\infty.
\end{equation}
\\
We first claim that for any $y \in V$ the estimate
\begin{equation}\label{1s1}
\big|\partial_s \bar{H}_s^R(y)\big|
\leq 2\partial_s \zeta_R(s)
\mathfrak{V}_{\mathfrak{P}}(\zeta_R(s))
\end{equation}
holds. This is clear if $y$ does not lie in the stable tubular
neighbourhood $U_{\tau_{\zeta_R(s)}}$ since there
$\partial_s \bar{H}_s^R$ vanishes. On the other hand,
if $y=\psi_{\zeta_R(s)}(x,r)
\in U_{\tau_{\zeta_R(s)}}$, then
$$\big|\partial_s \bar{H}_s^R(y)\big|=
\big|h'(r)\partial_s \zeta_R(s)\omega(X_{\psi_{\zeta_R(s)}},
R_{\lambda_{\zeta_R(s)}})\big|
\leq (1+\epsilon_0)\partial_s \zeta_R(s)
\mathfrak{V}_{\mathfrak{P}}(\zeta_R(s))$$
which implies (\ref{1s1}) since $\epsilon_0$ is already
chosen to be less than or equal $\frac{2}{5}$.
We estimate using (\ref{1s1})
\begin{eqnarray}\label{1s2}
\mathcal{A}^{H^R_\sigma}(w(\sigma))-
\mathcal{A}^{\alpha^-}(w^-)
&=&\int_{-\infty}^\sigma\frac{d}{ds}
\Big(\mathcal{A}^{H^R_s}(w(s))\Big)ds\\ \nonumber
&=&\int_{-\infty}^\sigma \big(\partial_s
\mathcal{A}^{H^R_s}\big)(w(s))ds\\ \nonumber
& &+\int_{-\infty}^\sigma d\mathcal{A}^{H^R_s}(w(s))
\partial_s w(s) ds\\ \nonumber
&=&-\int_{-\infty}^\sigma\int_0^1 \eta(s) \big(\partial_s
H^R_s\big)(v(s,t))dt ds\\ \nonumber & &+\int_{-\infty}^\sigma
m_s^R\big(\nabla_{m_s^R} \mathcal{A}^{H^R_s}(w(s),\partial_s
w(s)\big)ds\\ \nonumber &\geq&-2||\eta||_\infty
\int_{-\infty}^\sigma
\partial_s \zeta_R(s) \mathfrak{V}_{\mathfrak{P}}(\zeta_R(s))ds
\\ \nonumber
& &+\int_{-\infty}^\sigma m_s^R\big(\partial_s w(s),
\partial_s w(s)\big)ds\\ \nonumber
&\geq&-2||\eta||_\infty \int_{-\infty}^\infty
\partial_s \zeta_R(s) \mathfrak{V}_{\mathfrak{P}}(\zeta_R(s))ds
\\ \nonumber
&=&-2||\eta||_\infty \int_0^1
\mathfrak{V}_{\mathfrak{P}}(\zeta)d\zeta\\ \nonumber
&=&-2||\eta||_\infty \Delta(\mathfrak{P})\\ \nonumber
&=&-2||\eta||_\infty(\Delta+\epsilon_0)
\end{eqnarray}
Now using the fact that $\mathcal{A}^{\alpha^-}(w^-) \geq a$
shows the estimate from below in (\ref{1s}). The
estimate from above is derived in a similar manner.
This finishes the proof of Step~1.
\\ \\
To formulate Step\,2 we recall that the energy of
$w$ with respect to the metric $m^R_s$ is defined as
$$E(w)=\int_{-\infty}^\infty m_s^R\big(\partial_s w(s),
\partial_s w(s)\big) ds.$$
\textbf{Step\,2: } \emph{There exists a constant
$c_{\mathfrak{P}}$ such that for every $\sigma \in \mathbb{R}$
the following inequalities
$$a-4(\Delta+\epsilon_0)||\eta||_\infty
-\frac{c_{\mathfrak{P}}\big(E(w)+1\big)}{\sqrt{R}}
\leq \mathcal{A}_{\alpha_{\zeta_R(\sigma)}}(w(\sigma))$$
and
$$\mathcal{A}_{\alpha_{\zeta_R(\sigma)}}(w(\sigma))
\leq b+4(\Delta+\epsilon_0)||\eta||_\infty
+\frac{c_{\mathfrak{P}}\big(E(w)+1\big)}{\sqrt{R}}$$
are satisfied.}
\\ \\
For the extension of the stabilizing 1-form
$\lambda_{\zeta^R(s)}$ we use the abbreviation
$$\beta^R_s=\beta_{\lambda_{\zeta^R(s)}}.$$
We claim that for any $y \in V$ the following estimate holds
\begin{equation}\label{2s0}
\big|\partial_s \beta^R_s\big(X_{\bar{H}^R_s}(y)\big)\big| \leq
4\partial_s \zeta^R(s)\mathfrak{V}_{\mathfrak{P}} (\zeta^R(s)).
\end{equation}
If $y \notin U_{\tau_{\zeta^R(s)}}$, then $\beta^R_s$ vanishes and
the estimate obviously holds. We abbreviate
$$\dot{\beta}^R_s= \partial_{\zeta^R} \beta^R_s.$$
Denoting by $L$ the Lie derivative we
compute for $y=\psi_{\zeta_R(s)}(x,r) \in U_{\tau_{\zeta_R(s)}}$
\begin{eqnarray*}
\dot{\beta}^R_s\big(X_{\bar{H}_s^R}(y)\big)
&=&h'(r)\dot{\beta}^R_s\big(R_{\lambda_{\zeta^R(s)}}\big)\\
&=&h'(r)L_{X_{\psi_{\zeta^R(s)}}} \beta^R_s
\big(R_{\lambda_{\zeta^R(s)}}\big)+
h'(r)\phi(r)\dot{\lambda}_{\zeta^R(s)}
\big(R_{\lambda_{\zeta^R(s)}}\big)\\
&=&h'(r)d\iota_{X_{\psi_{\zeta^R(s)}}}
\beta_s^R\big(R_{\lambda_{\zeta^R(s)}}\big)+h'(r)
\iota_{X_{\psi_{\zeta^R(s)}}}d\beta_s^R
\big(R_{\lambda_{\zeta^R(s)}}\big)\\
& &+h'(r)\phi(r)\dot{\lambda}_{\zeta^R(s)}
\big(R_{\lambda_{\zeta^R(s)}}\big)\\
&=&h'(r)d\big[\phi(r)\iota_{X_{\psi_{\zeta^R(s)}}}
\lambda_{\zeta^R(s)}\big]\big(R_{\lambda_{\zeta^R(s)}}\big)\\
& &+h'(r)\iota_{X_{\psi_{\zeta^R(s)}}}\big[\phi'(r)dr \wedge
\lambda_{\zeta^R(s)}+\phi(r)d\lambda_{\zeta^R(s)}\big]
\big(R_{\lambda_{\zeta^R(s)}}\big)\\
& &+h'(r)\phi(r)\dot{\lambda}_{\zeta^R(s)}
\big(R_{\lambda_{\zeta^R(s)}}\big)\\
&=&h'(r)\phi(r)d\big[\iota_{X_{\psi_{\zeta^R(s)}}}
\lambda_{\zeta^R(s)}\big]\big(R_{\lambda_{\zeta^R(s)}}\big)\\
& &+h'(r)\phi'(r)\iota_{X_{\psi_{\zeta^R(s)}}}\lambda_{\zeta^R(s)}
dr\big(R_{\lambda_{\zeta^R(s)}}\big)\\
& &+h'(r)\iota_{X_{\psi_{\zeta^R(s)}}}\big[
\omega_\Sigma+\phi'(r)dr \wedge
\lambda_{\zeta^R(s)}+\phi(r)d\lambda_{\zeta^R(s)}\big]
\big(R_{\lambda_{\zeta^R(s)}}\big)\\
& &+h'(r)\phi(r)\dot{\lambda}_{\zeta^R(s)}
\big(R_{\lambda_{\zeta^R(s)}}\big)\\
&=&h'(r)\phi(r)d\big[\iota_{X_{\psi_{\zeta^R(s)}}}
\lambda_{\zeta^R(s)}\big]\big(R_{\lambda_{\zeta^R(s)}}\big)\\
& &+h'(r)\omega\big(X_{\psi_{\zeta^R(s)}},
R_{\lambda_{\zeta^R(s)}}
\big)\\
& &+h'(r)\phi(r)\dot{\lambda}_{\zeta^R(s)}
\big(R_{\lambda_{\zeta^R(s)}}\big).
\end{eqnarray*}
In the fifth equality we have used that the Reeb vector field lies
in the kernel of $\omega_\Sigma$. Using
$$\max \phi \leq \rho+1 \leq 2\rho, \quad \max h' \leq
1+\epsilon_0 \leq 2$$
we obtain from that the estimate
$$\big|\dot{\beta}_s^R\big(X_{\bar{H}_s^R}(y)\big)\big|
\leq 4 \mathfrak{V}_{\mathfrak{P}}(\zeta_R(s))$$
implying (\ref{2s0}).

As in the time-independent case we consider the
following time-dependent bilinear form on
$T(\mathcal{L}\times \mathbb{R})$
$$\widehat{m}_s^R\Big((\hat{v}_1,\hat{\eta}_1),
(\hat{v}_2,\hat{\eta}_2)\Big):=
\int_0^1 d\beta_s^R(\hat{v}_1,
J_{\zeta_R(s)}\hat{v}_2)dt+\hat{\eta}_1\hat{\eta}_2.$$
By Proposition~\ref{mod} we have for any $s \in \mathbb{R}$
$$\nabla_{m^R_s} \mathcal{A}^{H^R_s}=
\nabla_{\widehat{m}^R_s} \widehat{\mathcal{A}}^{H^R_s}$$
Hence the computation in Corollary~\ref{liap} shows that
\begin{eqnarray}\label{2s1}
\frac{d}{ds}\Big(\mathcal{A}_{\alpha_{\zeta_R(s)}}(w(s))\Big)
&=&\big(\partial_s \mathcal{A}_{\alpha_{\zeta_R(s)}}\big)(w(s))\\
\nonumber & &+\big(m^R_s-\widehat{m}^R_s\big)\Big(
\nabla_{m^R_s}\mathcal{A}^{H^R_s}(w(s)),
\nabla_{m^R_s}\mathcal{A}^{H^R_s}(w(s))\Big)\\ \nonumber
&\geq&\big(\partial_s \mathcal{A}_{\alpha_{\zeta_R(s)}}\big)(w(s)).
\end{eqnarray}
For $\zeta \in [0,1]$ let $\gamma_\zeta$
be the compactly supported 1-form
$$\gamma_\zeta=\dot{\beta}_{\lambda_\zeta} \circ J_\zeta$$
and define the constant $c_{\mathfrak{P}}$ as
$$c_{\mathfrak{P}}=\max_{\substack{\zeta \in [0,1]\\
y \in V}} ||\gamma_\zeta(y)||$$
where the norm is taken with respect to the metric
$\omega(\cdot,J_\zeta \cdot)$.
Again we only show the estimate from below. This is derived
similarly as the estimate (\ref{1s2})
in Step\,1. Using
(\ref{2s0}) and (\ref{2s1}) we obtain
\begin{eqnarray*}
\mathcal{A}_{\alpha_{\zeta_R(\sigma)}}(w(\sigma))
-\mathcal{A}_{\alpha^-}(w^-) &\geq&-\int_{-\infty}^\sigma \big(
\partial_s \mathcal{A}_{\alpha_{\zeta_R(s)}}\big)(w(s))ds\\
&=&-\int_{-\infty}^\sigma \partial_s \zeta_R(s)\bigg(\int_0^1
 v^*\dot{\beta}_s^Rdt\bigg) ds\\
&=&-\int_{-\infty}^\sigma \partial_s \zeta_R(s)\eta(s)\bigg(\int_0^1
\dot{\beta}_s^R\big(\eta X_{H^R_s}(v(s,t))\big)dt\bigg) ds\\
& &-\int_{-\infty}^\sigma \partial_s \zeta_R(s) \bigg(\int_0^1
\dot{\beta}_s^R
\big(J_{\zeta_R(s)}\partial_s v(s,t)\big)dt\bigg)ds\\
&\geq&-4||\eta||_\infty \int_{-\infty}^\sigma
\partial_s \zeta_R(s) \mathfrak{V}_{\mathfrak{P}}(\zeta_R(s))
ds\\
& &-c_{\mathfrak{P}}\int_{-\infty}^\sigma
\int_0^1\partial_s \zeta_R(s)
||\partial_s v(s,t)||dt ds\\
&\geq&-4||\eta||_\infty \int_{-\infty}^\infty
\partial_s \zeta_R(s) \mathfrak{V}_{\mathfrak{P}}(\zeta_R(s))
ds\\
& &-c_{\mathfrak{P}}\int_{-\infty}^\infty \partial_s \zeta_R(s)
\int_0^1
\bigg(\sqrt{R}||\partial_s v(s,t)||^2+\frac{1}{\sqrt{R}}
\bigg)dt ds\\
&\geq&-4||\eta||_\infty \int_0^1
\mathfrak{V}_{\mathfrak{P}}(\zeta)d\zeta
-\frac{c_{\mathfrak{P}}}{\sqrt{R}}\int_{-\infty}^\infty
\partial_s \zeta_R(s)ds\\
& &-\frac{c_{\mathfrak{P}}}{\sqrt{R}}
\int_{-\infty}^\infty ||\partial_s w(s)||^2 ds\\
&\geq&-4||\eta||_\infty(\Delta+\epsilon_0)-
\frac{c_{\mathfrak{P}}\big(E(w)+1\big)}{\sqrt{R}}
\end{eqnarray*}
This finishes the proof of Step\,2.
\\ \\
\textbf{Step\,3: }\emph{The energy can be estimated by
$$E(w) \leq \mathcal{A}^{\alpha^+}(w^+)-
\mathcal{A}^{\alpha^-}(w^-)+2||\eta||_\infty(\Delta+\epsilon_0)
\leq b-a+2||\eta||_\infty(\Delta+\epsilon_0).$$}
\\
A careful look at (\ref{1s2}) reveals Step\,3.
\\ \\
\textbf{Step\,4: } \emph{We prove the theorem.}
\\ \\
As in the proof of Proposition~\ref{abo}, we set for
$\sigma \in \mathbb{R}$
$$\tau(\sigma)=\inf\Big\{\tau \geq 0:
\big|\big|\nabla \mathcal{A}^{H^R_\sigma}\big(w(\sigma+\tau)\big)
\big|\big|\leq \frac{2}{5}\Big\}.$$
In the time dependent case we still can estimate
\begin{equation}\label{4s1}
\tau(\sigma) \leq \frac{25 E(w)}{4}.
\end{equation}
Since $\alpha_\zeta$ is $\big(\frac{3}{5},\epsilon_0\big)$-admissible
for every $\zeta \in [0,1]$ we conclude from
Lemma~\ref{abound2} applied to $\gamma=\frac{3}{5}$ together
with the fact that $||H^R_s||_\infty \leq \frac{3}{5}+\epsilon_0$
for every $s \in \mathbb{R}$ and (\ref{4s1}) the following estimate
\begin{eqnarray}\label{4s2}
|\eta(\sigma)| &\leq& |\eta(\sigma+\tau(\sigma))|
+\int_{\sigma}^{\sigma+\tau(\sigma)}|\partial_s \eta(s)|ds\\ \nonumber
&\leq&\frac{5}{2}\sup \big |\widehat{\mathcal{A}}^{H^R} \circ w\big|
+1+\tau(\sigma)||H^R||_\infty\\ \nonumber
&\leq&\frac{5}{2}\sup \big |\widehat{\mathcal{A}}^{H^R} \circ w\big|
+1+\frac{25E(w)}{4}\bigg(\frac{3}{5}+\epsilon_0\bigg).
\end{eqnarray}
$\,$From Step\,1 and Step\,2 we obtain  the estimate
\begin{eqnarray}\label{4s3}
\sup \big |\widehat{\mathcal{A}}^{H^R} \circ w\big|
&=&\sup \big|\mathcal{A}^{H^R}\circ w-
\mathcal{A}_{\alpha_{\zeta_R}} \circ w\big|\\ \nonumber
&\leq&b-a+6(\Delta+\epsilon_0)||\eta||_\infty+
\frac{c_{\mathfrak{P}}(E(w)+1)}{\sqrt{R}}.
\end{eqnarray}
Combining (\ref{4s2}) and (\ref{4s3}) with Step\,3 we obtain
\begin{eqnarray}\label{4s4}
||\eta||_\infty &\leq& \frac{5}{2}(b-a)+
1+\frac{5c_{\mathfrak{P}}}{2\sqrt{R}}+\bigg(\frac{15}{4}
+\frac{25\epsilon_0}{4}+\frac{5c_{\mathfrak{P}}}{2\sqrt{R}}\bigg)
E(w)\\ \nonumber & &+15(\Delta+\epsilon_0)||\eta||_\infty\\
\nonumber
&\leq&\bigg(\frac{25(1+\epsilon_0)}{4}+\frac{c_{\mathfrak{P}}}
{\sqrt{R}}\bigg)(b-a)+1+\frac{5c_{\mathfrak{P}}}{2\sqrt{R}}\\
\nonumber & &+\bigg(\frac{45}{2}+\frac{25\epsilon_0}{2}+
\frac{5c_{\mathfrak{P}}}{\sqrt{R}}\bigg)(\Delta+\epsilon_0)
||\eta||_\infty.
\end{eqnarray}
We now choose $\Delta=\Delta(\epsilon)$ in such a way that
$$0<\Delta<\min\bigg\{\frac{2\epsilon}{125+100\epsilon},
\frac{1}{60}\bigg\}.$$
The positive number $\epsilon_0$ which already had to be
chosen in such a way that $\epsilon_0 \leq \frac{2}{5}$
is now supposed to satisfy
$$0<\epsilon_0 \leq \min\bigg\{\frac{1}{60}-\Delta,
\frac{2\epsilon}{125+100\epsilon}-\Delta\bigg\}.$$
Moreover, we choose $R$ such that
$$R \geq \max\bigg\{
\bigg(\frac{5c_{\mathfrak{P}}}{2(1-60(\Delta+\epsilon_0))}\bigg)^2,
\bigg(\frac{c_{\mathfrak{P}}}{\epsilon}\bigg)^2\bigg\}.$$
With these choices we get from (\ref{4s4})
$$||\eta||_\infty \leq \bigg(\frac{25}{4}+2\epsilon\bigg)(b-a)
+2\big(1-30(\Delta+\epsilon_0)\big)+
30(\Delta+\epsilon_0)||\eta||_\infty.$$ Using $\Delta+\epsilon_0
\leq \frac{2\epsilon}{125+100\epsilon}$ we obtain
\begin{eqnarray*}
||\eta||_\infty &\leq&
\frac{1}{1-30(\Delta+\epsilon_0)}\bigg(\frac{25}{4}+2\epsilon\bigg)
(b-a)+2\\
&\leq&\bigg(\frac{25}{4}+5\epsilon\bigg)(b-a)+2\\
&\leq&\bigg(\frac{5}{2}+\epsilon\bigg)^2(b-a)+2.
\end{eqnarray*}
This finishes the proof of the theorem.
\end{proof}

\subsection{An existence result for a periodic Reeb orbit}
\label{ss:existence}

In this subsection we show how compactness for gradient
flow lines of Rabinowitz action functional leads
to existence of a periodic Reeb orbit on stable displaceable
hypersurfaces. This result is not new. Indeed, F.\,Schlenk
proved it before in \cite{Sch} using quite different methods.
Before stating the theorem we recall some well-known notions.

A hypersurface $\Sigma$ in a symplectic manifold $(V,\omega)$
is called \emph{displaceable}, if there exists a
compactly supported Hamiltonian $F \in C^\infty(V \times S^1)$
such that the time-one flow $\phi_F$ of the time dependent
Hamiltonian vector field $X_{F_t}$ with $F_t=F(\cdot,t)
\in C^\infty(V)$ satisfies
$$\phi_F(\Sigma) \cap \Sigma =\emptyset.$$
The positive and the negative part of the Hofer norm for
the compactly supported Hamiltonian $F$ are given by
$$||F||_+=\int_0^1 \max_V F_t dt, \quad
||F||_-=-\int_0^1 \min_V F_t dt$$
and the Hofer norm itself by
$$||F||=||F||_++||F||_-.$$
If $\Sigma \subset V$ is a displaceable hypersurface its
\emph{displacement energy} is given by
$$e(\Sigma)=\inf\{||F||: \phi_F(\Sigma) \cap \Sigma=\emptyset\}.$$
Recall from (\ref{OE}) the $\omega$-energy for
closed characteristics.
\begin{theorem}[Schlenk]\label{felix}
Assume that $\Sigma$ is a stable, displaceable
hypersurface
in a symplectically aspherical, geometrically bounded,
symplectic manifold $(V,\omega)$. Then $\Sigma$ has a closed
characteristic $v$ which is contractible in $V$ and satisfies
$$\Omega(v) \leq e(\Sigma).$$
\end{theorem}
Our proof of Theorem~\ref{felix} is based on a homotopy stretching
argument for a time dependent perturbation of Rabinowitz action
functional where the perturbation is given by the displacing
Hamiltonian.

The crucial point is that the perturbed Rabinowitz action
functional has no critical points anymore. This is
actually true for any separating closed hypersurface $\Sigma$
in a symplectically aspherical,
symplectic manifold $(V,\omega)$. If $\Sigma=\bar{H}^{-1}(0)$
for a Hamiltonian $\bar{H} \in C^\infty(V)$
we choose $\chi \in C^\infty(S^1,\mathbb{R})$ of integral
one which in addition meets the condition
$$\mathrm{supp} \chi \subset (0,1/2)$$
and set as usual
$$H(t,y)=\chi(t)\bar{H}(y), \quad y \in V,\,\,t \in S^1.$$
Without changing the Hofer norm we furthermore can reparametrize
the flow of the displacing Hamiltonian $F_t$ such that we can
assume that
$$F_t=0, \quad t \in [0,1/2].$$
The perturbation of Rabinowitz action functional we consider
is the functional $\mathcal{A}^H_F \colon \mathcal{L}
\times \mathbb{R} \to \mathbb{R}$ defined by
$$\mathcal{A}^H_F(v,\eta)=\mathcal{A}^H(v,\eta)-
\int_0^1 F_t(v(t))dt, \quad
(v,\eta) \in \mathcal{L}\times \mathbb{R}.$$
We further denote by
$$\mathfrak{S}(X_{\bar{H}})=\mathrm{cl}\{y \in V:
X_{\bar{H}}(y) \neq 0\}$$
the support of the Hamiltonian vector field of $\bar{H}$.
The following Lemma is proved in \cite{CF} where it appears
as Lemma 3.10. It immediately implies that a suitable perturbed
Rabinowitz action functional has no critical points anymore.
\begin{lemma}\label{mu}
Assume that $\mathfrak{S}(X_{\bar{H}})$ is compact and
$\phi_F\big(\mathfrak{S}(X_{\bar{H}})\big) \cap \mathfrak{S}(X_{\bar{H}})
=\emptyset$. Then for every $\omega$-compatible almost complex structure
$J$ on $V$ there exists a constant $\mu=\mu(J)>0$ such that
if $\nabla$ is the gradient and $||\cdot||$ is the norm of the metric
on $\mathcal{L}\times \mathbb{R}$ induced from $J$, then for every
$(v,\eta) \in \mathcal{L}\times \mathbb{R}$ it holds that
$$||\nabla \mathcal{A}^H_F(v,\eta)|| \geq \mu.$$
\end{lemma}
With the help of this Lemma we are now armed for the proof
of Schlenk's Theorem.
\\ \\
\textbf{Proof of Theorem~\ref{felix}: }Given $\epsilon_0>0$
we choose a time dependent Hamiltonian $F_t$ satisfying the
following conditions
$$F_t=0,\,\,t \in [0,1/2], \quad \phi_F(\Sigma) \cap \Sigma=
\emptyset, \quad ||F||<e(\Sigma)+\epsilon_0.$$
We further choose a $(\gamma,\epsilon)$-admissible quintuple
$\alpha=(\lambda,\tau,\phi,h,\chi)$ which satisfies
$$\mathrm{supp}\chi \subset (0,1/2), \quad
\phi_F(U_\tau) \cap U_\tau=\emptyset.$$
The first condition can obviously be achieved. To see that one can also
assume without loss of generality the second condition
we note that since
$\Sigma$ is compact there exists an open neighboorhood $U_\Sigma$
of $\Sigma$ in $V$ which is also displaced by $\phi_F$. Now
by choosing $\lambda$ arbitrarily small we can arrange that
even a large tubular neighbourhood $U_\tau$ is contained
in $U_\Sigma$.

We further choose a smooth family of cutoff functions
$\beta_r \in C^\infty(\mathbb{R},[0,1])$ for
$r \in [0,\infty)$ with the following properties
\begin{equation}\label{cuto}
\left.\begin{array}{cc}
\beta_r(s)=0 & |s| \geq r\\
\beta_r(s)=1 & |s| \leq r-1\\
s\beta'_r(s) \leq 0 & \forall\,\,r, \forall\,\,s.
\end{array}
\right\}
\end{equation}
We now consider the $r$-parametrised family of time dependent
perturbed Rabinowitz action functionals defined by
$$\mathcal{A}^\alpha_r(v,\eta,s)=\mathcal{A}^\alpha(v,\eta)-
\beta_r(s)\int_0^1 F_t(v(t))dt, \quad
v \in \mathcal{L}, \,\,\eta \in \mathbb{R},\,\,
s \in \mathbb{R},\,\,r \in [0,\infty).$$
We note that
$$\mathcal{A}^\alpha_0=\mathcal{A}^\alpha$$
is independent of the $s$-variable.
We further choose an $\omega$-compatible
almost complex structure $J \in \mathcal{J}(\alpha)$ and
denote by $\nabla$ the gradient with respect to the
metric $m$ on $\mathcal{L}\times \mathbb{R}$ induced from
$J$. We fix a point $x \in \Sigma$ and think of it
as a loop in $\mathcal{L}$. We are studying solutions
$(w,r)=(v,\eta,r) \in C^\infty(\mathbb{R},\mathcal{L}\times \mathbb{R})
\times [0,\infty)$ of the following problem
\begin{equation}\label{str}
\partial_s w(s)=\nabla \mathcal{A}^\alpha_r(s,w(s)),\,\,
s \in \mathbb{R}, \quad \lim_{s\to -\infty}w(s)=(x,0), \quad
\lim_{s\to \infty}w(s) \in \Sigma \times \{0\}.
\end{equation}
We use the following abbreviation for its moduli space
$$\mathcal{M}=\big\{(w,r):(w,r)\,\,\textrm{solution of (\ref{str})}
\big\}.$$
To prove the Theorem we argue by contradiction and assume
\begin{equation}\label{wid}
\Omega(v)>||F||, \quad \forall\,\,v \in X(\Sigma).
\end{equation}
To see how this leads to a contradiction, we first show the
following claim.
\\ \\
\textbf{Claim: } \emph{If (\ref{wid}) holds, then $\mathcal{M}$
is compact.}
\\ \\
We prove the Claim in four steps. For the first step
recall that the energy of $w$ is defined by
$$E(w)=\int_{-\infty}^\infty||\partial_s w||^2 ds$$
where the norm is taken with respect to the metric $m$ induced
from the $\omega$-compatible almost complex structure $J$.
\\ \\
\textbf{Step\,1:} If $(w,r) \in \mathcal{M}$ then $E(w) \leq ||F||$.
\\ \\
We estimate using (\ref{str})
\begin{eqnarray*}
E(w)&=&\int_{-\infty}^\infty d\mathcal{A}^\alpha_r(w)(\partial_s w) ds\\
&=&\int_{-\infty}^\infty\frac{d}{ds}\mathcal{A}^\alpha_r(w)ds
-\int_{-\infty}^\infty\big(\partial_s \mathcal{A}^\alpha_r)(w)ds\\
&=&0+\int_{-\infty}^\infty \beta'_r(s)
\bigg(\int_0^1 F_t(v)dt\bigg)ds\\
&\leq&||F||_+\int_{-\infty}^0\beta'_r(s)ds-
||F||_-\int_0^\infty \beta'_r(s)ds\\
&\leq&||F||_++||F||_-\\
&=&||F||.
\end{eqnarray*}
This finishes the proof of Step\,1.
\\ \\
\textbf{Step\,2:} There exists $r_0 \in \mathbb{R}$ such that
if $(w,r) \in \mathcal{M}$ then $r \leq r_0$.
\\ \\
Combining Lemma~\ref{mu} with Step\,1 we obtain the
estimate
\begin{eqnarray*}
||F|| &\geq& \int_{-r}^{r}
||\nabla \mathcal{A}^\alpha_F(w)||^2 ds\\
&\geq&2\mu^2r
\end{eqnarray*}
implying that
$$r \leq \frac{||F||}{2\mu^2}=:r_0$$
This finishes the proof of Step\,2.
\\ \\
\textbf{Step\,3: }\emph{There exists a constant $c>0$
such that for all $(w,r)=(v,\eta,r) \in \mathcal{M}$ the
Lagrange multiplier $\eta$ is uniformly bounded
by $||\eta|| \leq c$.}
\\ \\
To prove Step\,3 we estimate the functional
$\mathcal{A}=\mathcal{A}^\alpha-\widehat{\mathcal{A}}^\alpha$
along $w$. Note that we do not perturb $\mathcal{A}$
with the displacing Hamiltonian $F$. It is useful
to introduce further the functional
$$\mathcal{F}  \colon \mathcal{L}\times \mathbb{R} \to \mathbb{R},
\quad (v,\eta) \mapsto \int_0^1 F_t(v(t))dt$$
which actually only depends on the first variable. By
Proposition~\ref{mod} we have
$$\nabla_{\widehat{m}} \widehat{\mathcal{A}}^\alpha
=\nabla_m \mathcal{A}^\alpha.$$

Hence using (\ref{str}) we estimate similarly as in
Corollary~\ref{liap}
\begin{eqnarray*}
\frac{d}{ds} \mathcal{A}(w)&=&d\mathcal{A}^\alpha(w)\partial_s w
-d\widehat{\mathcal{A}}^\alpha(w)\partial_s w\\
&=&m\big(\nabla_m\mathcal{A}^\alpha(w),\partial_s w\big)-
\widehat{m}\big(\nabla_m \mathcal{A}^\alpha(w),\partial_s w\big)\\
&=&(m-\widehat{m})\big(\partial_s w,\partial_s w\big)+
\beta_r(m-\widehat{m})
\big(\nabla_m \mathcal{F}(w),\partial_s w\big)\\
&\geq&\beta_r(m-\widehat{m})
\big(\nabla_m \mathcal{F}(w),\partial_s w\big).
\end{eqnarray*}
Since $F$ has compact support there exists
a constant $c_0$ such that for all $w \in \mathcal{L}\times
\mathbb{R}$
$$\big|\big|(m-\widehat{m})\big(\nabla_m \mathcal{F}(w),
\cdot\big)\big|\big|_m
\leq c_0.$$
Hence we obtain for $\sigma \in \mathbb{R}$ using
Step\,1 and Step\,2
\begin{eqnarray*}
\mathcal{A}(w(\sigma))&=&\int_{-\infty}^\sigma
\frac{d}{ds}\mathcal{A}(w)ds\\
&\geq& \int_{-\infty}^\sigma \beta_r(m-\widehat{m})
\big(\nabla_m \mathcal{F}(w),\partial_s w)ds\\
&\geq&-c_0\int_{-r}^r||\partial_s w||_m ds\\
&\geq&-c_0\int_{-r}^r\big(||\partial_s w||_m^2+1\big)ds\\
&\geq&-c_0\big(2r+E(w)\big)\\
&\geq&-c_0\big(2r_0+||F||\big).
\end{eqnarray*}
Similarly, one gets
$$-\mathcal{A}(w(\sigma))=\int_\sigma^\infty
\frac{d}{ds}\mathcal{A}(w)ds
\leq c_0\big(2r_0+||F||\big).$$
Defining the constant
$$c_1=c_0\big(2r_0+||F||\big)$$
we obtain from the previous two estimates the
uniform $L^\infty$-bound
$$||\mathcal{A} \circ w|| \leq c_1.$$
Moreover, a closer look at the estimate in Step\,1 reveals
that
$$||\mathcal{A}^\alpha_r \circ w|| \leq ||F||.$$
Noting that for $s \notin (-r,r)$ we have
$\mathcal{A}^\alpha_r(\cdot,s)=\mathcal{A}^\alpha$ we
obtain from the previous two inequalities
\begin{equation}\label{haaa}
|\widehat{\mathcal{A}}^\alpha(w(s))| \leq
c_1+||F||, \quad s \in \mathbb{R}\setminus (-r,r).
\end{equation}
The proof for the bound of the
Lagrange multiplier now proceeds similarly
as in Proposition~\ref{abo}.
For $\sigma \in \mathbb{R}$ we set
$$\tau(\sigma)=
\inf\bigg\{\tau: \sigma+\tau \notin (-r,r),\,\,
||\nabla \mathcal{A}^\alpha(w(\sigma+\tau)|| \leq
\frac{2\gamma}{3}\bigg\}.$$
Again $\tau(\sigma)$ can be estimated in terms of the
energy by
\begin{equation}\label{ts}
\tau(\sigma) \leq \frac{9E(w)}{4\gamma^2}+2r
\leq \frac{9||F||}{4\gamma^2}+2r_0.
\end{equation}
Combining Lemma~\ref{abound} with (\ref{haaa}) and
(\ref{ts})
and using $\partial_s \eta=-\int_0^1 H(v)dt$ we
estimate
\begin{eqnarray*}
|\eta(\sigma)| &\leq&|\eta(\sigma+\tau(\sigma))|
+\int_{\sigma}^{\sigma+\tau(\sigma)}|\partial_s \eta|ds\\
&\leq&\frac{3c_1+3||F||+2\gamma}{3(1-\gamma)}+
(\gamma+\epsilon)\tau(\sigma)\\
&\leq&\frac{3c_1+3||F||+2\gamma}{3(1-\gamma)}+
\frac{9||F||(\gamma+\epsilon)}{4\gamma^2}+2r_0(\gamma+\epsilon).
\end{eqnarray*}
Since $\sigma$ was arbitrary we are done with Step\,3.
\\ \\
\textbf{Step\,4: } \emph{We prove the claim.}
\\ \\
For $\nu \in \mathbb{N}$ let $(w_\nu,r_\nu)=(v_\nu,\eta_\nu,r_\nu)$
be a sequence in $\mathcal{M}$. Since the homotopy parameter
$r_\nu$ is uniformly bounded by Step\,2 and
the Lagrange multiplier
$\eta_\nu$ is uniformly bounded by Step\,3 standard arguments in
Floer theory imply that $(w_\nu,r_\nu)$ has a
$C^\infty_{\mathrm{loc}}$-convergent subsequence.
Indeed, $v_\nu$ satisfies
a uniform $C^0$-bound by the assumption that $(V,\omega)$
is geometrically bounded and the derivatives of $v_\nu$ can be
controlled because there is no bubbling since $(V,\omega)$
is symplectically aspherical. Let $(w,r)$ be the limit of the
subsequence. $(w,r)$ obviously satisfies the first equation in
(\ref{str}). It remains to check that
$w$ satisfies the asymptotic conditions. Again by compactness
it follows that $w(s)$ converges to critical points
$w^\pm=(v^\pm,\eta^\pm)$ of
$\mathcal{A}^\alpha$ as $s$ goes to $\pm \infty$. On
the other hand it follows from Step\,1 that
$$\mathcal{A}^\alpha_r(w(s)) \in [-||F||,||F||], \quad
\forall\,\,s \in \mathbb{R}$$
and hence
$$\Omega(v^\pm)=\mathcal{A}^\alpha(w^\pm) \in [-||F||,||F||].$$
Therefore (\ref{wid}) implies that $v^\pm$ has to be constant
and hence
$$w^-=(x,0), \quad w^+ \in \Sigma \times \{0\}.$$
This finishes the prove of the claim.
\\ \\
Given the claim we are now in position to prove the theorem in
a last step.
\\ \\
\textbf{Step\,5: }\emph{We prove the theorem.}
\\ \\
Given the claim it remains to argue that the compactness
of the moduli space $\mathcal{M}$ is absurd in order to
show that (\ref{wid}) cannot be true. For $r=0$ there
is precisely one point $(w,0) \in \mathcal{M}$ namely
the constant gradient flow line $w=(x,0)$. The
constant gradient flow line is regular in the sense that
the linearization of the gradient flow equation at
it is surjective. Thinking of the moduli space
$\mathcal{M}$ as the zero set of a Fredholm section from
a Banach space into a Banach bundle and using
that it is compact we can perturb this
section slightly to make it transverse. The zero set of
the perturbed section is now a compact
manifold with one single boundary point $(x,0,0)$. However,
such manifolds do not exist. Therefore (\ref{wid}) had
to be wrong and we conclude that there exists $v \in X(\Sigma)$
such that
$$\Omega(v) \leq ||F||<e(\Sigma)+\epsilon_0.$$
Since $\epsilon_0>0$ was arbitrary the theorem follows.
\hfill $\square$

\subsection{Approvable perturbations}\label{ss:approvable}

Except in the case where $V$ is zero dimensional, the Rabinowitz
action functional is never Morse, since its critical set contains
the constant solutions and each nontrivial Reeb orbit comes in an
$S^1$-family coming from time-shift. The best situation we can hope
for, is that Rabinowitz action functional is Morse-Bott. However,
since the stability condition is not an open condition \cite{CFP}, a
slight perturbation of a stable hypersurface might not be stable
anymore.

In this subsection we study a class of perturbations of Rabinowitz
action functional. We first show that for a generic perturbation the
perturbed Rabinowitz action functional is Morse. We then explain how
in the weakly tame case for small perturbations the moduli space of
gradient flow lines in a fixed action interval can be written as the
disjoint union of two closed subspaces where one of them is compact.
We refer to the compact part as the essential part of the moduli
space of gradient flow lines. We finally explain how the essential
part of the moduli space of gradient flow lines can be used to
define a boundary operator for a fixed action interval.

The perturbations of Rabinowitz action functional
we consider are reminiscent of the ones we
considered in the previous subsection, however
they are more general, since we do not
require that the time support of the additional
perturbation Hamiltonian is disjoint from the
time support of the Hamiltonian $H_\alpha$.
Namely we choose a
compactly supported time-dependent Hamiltonian
$F \in C_c^\infty(V \times S^1)$ and define
$\mathcal{A}^\alpha_F \colon \mathcal{L} \times \mathbb{R}
\to \mathbb{R}$ as in the previous subsection by
$$\mathcal{A}^\alpha_F(v,\eta)=\mathcal{A}^\alpha(v,\eta)
-\int_0^1 F_t(v(t))dt.$$
Critical points of the action functional
$\mathcal{A}^\alpha_F$ are solutions of the problem
\begin{equation}\label{percr}
\left.\begin{array}{c}
\partial_t v=\eta X_H(v)+X_{F_t}(v)\\
\int_0^1 H(v(t))dt=0.
\end{array}\right\}
\end{equation}
We first show that for generic perturbations the perturbed
action functional is Morse.
\begin{proposition}\label{morse}
Given an admissible quadruple $\alpha$, there exists
a subset $\mathcal{U}(\alpha) \in
C^\infty_c(V \times S^1)$ of the second category
such that $\mathcal{A}^\alpha_F$ is Morse for every
$F \in \mathcal{U}(\alpha)$.
\end{proposition}
\begin{proof}
Consider the Hilbert manifold
$$\mathcal{H}=W^{1,2}(S^1,V)\times \mathbb{R}$$
where we define $W^{1,2}(S^1,V)$
by embedding $V$ into $\mathbb{R}^N$ for
$N$ large enough. Over the Hilbert manifold
$\mathcal{H}$ we introduce the Hilbert bundle
$$\pi \colon \mathcal{E} \to \mathcal{H}$$
whose fibre over $(v,\eta) \in \mathcal{H}$ is given by
$$\mathcal{E}_{(v,\eta)}=L^2(S^1,v^*TV)\times
\mathbb{R}.$$
Choose an $\omega$-compatible almost complex structure $J$
and denote by $\nabla$ the gradient with respect to the
metric $\omega(\cdot,J\cdot)$.
For $F \in C^\infty_c(V \times S^1)$ we define a section
$$s_F \colon \mathcal{H} \to \mathcal{E}$$
by
$$s_F(v,\eta)=
\left(\begin{array}{c}
J\partial_t v-\eta \nabla H(v)-\nabla F_t(v)\\
\int_0^1H(v(t))dt
\end{array}\right).$$
Note that the zero set $s_F^{-1}(0)$ coincides
with the solutions of (\ref{percr}). If
$w \in \mathcal{H} \subset \mathcal{E}$ then there is
a canonical splitting
$$T_w \mathcal{E}=\mathcal{E}_w \times T_w \mathcal{H}.$$
We denote by
$$\Pi_w \colon T_w \mathcal{E} \to \mathcal{E}_w$$
the projection along $T_w \mathcal{H}$. If
$(v,\eta) \in s_F^{-1}(0)$ we introduce the vertical
differential
$$Ds_F(v,\eta) \colon
T_{(v,\eta)}\mathcal{H}=W^{1,2}(S^1,v^*TV)\times
\mathbb{R} \to \mathcal{E}_{(v,\eta)}$$
by
$$Ds_F(v,\eta)=\Pi_{(v,\eta)} \circ ds_F(v,\eta).$$
The action functional $\mathcal{A}^\alpha_F$ is Morse
iff the vertical differential $Ds_F(w)$ is surjective
for every $w \in s_F^{-1}(0)$ meaning that
$s_F$ is transverse to the zero section
\begin{equation}\label{tra}
s_F\pitchfork 0.
\end{equation}
We prove in two steps that (\ref{tra}) holds for generic
$F \in C^\infty_c(V \times S^1)$. In the first step
we prove transversality for weaker differentiability
assumptions. The smooth case follows then by an argument
due to Taubes.
\\ \\
\textbf{Step\,1: }\emph{Assume that $2 \leq k <\infty$. Then there
exists $\mathcal{U}^k(\alpha) \subset C^k_c(V \times S^1)$ of the
second category, such that (\ref{tra}) holds for any $F \in
\mathcal{U}^k(\alpha)$.}
\\ \\
Consider the section
$$S \colon C^k_c(V \times S^1) \times \mathcal{H} \to
\mathcal{E}$$
which is defined by
$$S(F,w)=s_F(w), \quad F \in C^k_c(V \times S^1),\,\,
w \in \mathcal{H}.$$
If $(F,w) \in S^{-1}(0)$ with
$w=(v,\eta)$, then the vertical differential
$$DS(F,w) \colon T_{(F,w)}\big(C^k_c(V \times S^1)
\times \mathcal{H}\big)=C^k_c(V \times S^1)
\times T_w \mathcal{H} \to \mathcal{E}_w$$
is given for $\hat{F} \in C^k_c(V \times S^1)$ and
$\hat{w} \in T_w \mathcal{H}$ by
\begin{equation}\label{vedi}
DS(F,w)(\hat{F},\hat{w})=Ds_F(w) \hat{w}+
\left(\begin{array}{c}
-\nabla \hat{F}_t(v)\\
0
\end{array}\right).
\end{equation}
We first show the following claim.
\\ \\
\textbf{Claim: }\emph{For every $(F,w) \in S^{-1}(0)$ the operator
$DS(F,w)$ is surjective.}
\\ \\
Pick $(F,w) \in S^{-1}(0)$. Since $Ds_F(w)$ is Fredholm, the
image of $DS(F,w)$ is closed. Hence to show surjectivity, it
suffices to prove that the orthogonal complement of the image of
$DS(F,w)$ vanishes. To see that pick
$$x=(y,\zeta) \in \mathrm{im}DS(F,w)^\perp.$$
It follows from (\ref{vedi}) that
\begin{equation}\label{oco}
\left.\begin{array}{cc}
\langle Ds_F(w)\hat{w},x\rangle=0, & \forall
\,\,\hat{w} \in T_w \mathcal{H}\\
\langle \nabla\hat{F}(v),y\rangle=0, & \forall
\,\,\hat{F} \in C^k_c(V \times S^1).
\end{array}
\right\}
\end{equation}
The first equation in (\ref{oco}) implies that
$$x \in \mathrm{ker}(D s_F(w))^*$$
which forces $y$ to be of class $C^{k-1}$. Now we assume by
contradiction that there exists $t_0 \in S^1$ such that
$$y(t_0) \neq 0.$$
Choose $\hat{F}_{t_0} \in C^k_c(V)$ such that
$$\big\langle \nabla\hat{F}_{t_0}(v(t_0)),y(t_0) \big\rangle>0.$$
Since we have seen that $y$ is continuous, there exists
$\epsilon>0$ such that
$$\big\langle \nabla\hat{F}_{t_0}(v(t)),y(t) \big\rangle \geq 0,
\quad t \in (t_0-\epsilon,t_0+\epsilon).$$
Now choose a smooth cutoff function
$\beta \in C^\infty(S^1,[0,1])$ such that
$\beta(t_0)=1$ and $\beta(t)=0$ for
$t \notin (t_0-\epsilon,t_0+\epsilon)$ and set
$\hat{F}_t=\beta(t)\hat{F}_{t_0}$.
It follows that
$$\big\langle \nabla\hat{F},y\big\rangle
=\int_0^1 \beta(t)\big\langle \nabla\hat{F}_{t_0}(v(t)),y(t)
\big\rangle dt>0$$
contradicting the second equation in (\ref{oco}).
This proves that $y$ has to vanish identically.

It remains to show that $\zeta$ vanishes. To see this we write
$$\hat{w}=(\hat{v},\hat{\eta}) \in T_w \mathcal{H}=
W^{1,2}(S^1,v^*TV) \times \mathbb{R}.$$
Since $y$ vanishes identically the first equation in (\ref{oco})
becomes
$$\zeta \int_0^1 dH(v(t))\hat{v}(t)dt=0, \quad
\forall\,\,\hat{v} \in W^{1,2}(S^1,v^*TV).$$
Note that $(v,\eta)$ is a solution of (\ref{percr}). Hence, since
$0$ is a regular value of $H$, it follows from the second
equation in (\ref{percr}) that $dH(v)$ does not vanish identically
along $v$. Therefore, there exists $\hat{v} \in W^{1,2}(S^1,v^*TV)$
such that
$$\int_0^1 dH(v(t)) \hat{v}(t)dt \neq 0.$$
Consequently,
$$\zeta=0.$$
This finishes the proof of the claim.
\\ \\
Since $F \in C^k_c(V \times S^1)$ it follows that the section
$S$ is $C^{k-1}$. Hence by the claim the implicit function
theorem shows that $S^{-1}(0)$ is a $C^{k-1}$-manifold.
Consider the $C^{k-1}$-map
$$p \colon S^{-1}(0) \to C^k_c(V\times S^1), \quad
(F,w) \mapsto F.$$
It follows from the Sard-Smale theorem that the set
of regular values of the map $p$ is of the second category in
$C^k_c(V \times S^1)$. But $F$ is a regular value of
$p$, precisely if $s_F \pitchfork 0$. This finishes
the proof of Step 1.
\\ \\
\textbf{Step~2: }\emph{We prove the proposition.}
\\ \\
We explain the argument by Taubes in our set-up (cf. \cite[p.52]{MS1}).

Since $X_H$ and $X_{F_t}$ have compact support and $0$
is a regular value of $H$ it follows that there exists
a compact subset $V_0 \subset V$ such that
for every $(v,\eta) \in \mathrm{crit}(\mathcal{A}^\alpha_F)$
\begin{equation}\label{schr}
v(t) \in V_0, \quad t \in S^1.
\end{equation}
Choose $T>0$ and abbreviate
$$\mathrm{crit}_T(\mathcal{A}^\alpha_F)=
\big\{(v,\eta) \in \mathrm{crit}(\mathcal{A}^\alpha_F): |\eta| \leq
T\big\}.$$ It follows from (\ref{schr}), the first equation in
(\ref{percr}) and the Theorem of Arzela-Ascoli that
$\mathrm{crit}_T(\mathcal{A}^\alpha_F)$ is compact. For $2 \leq k
\leq \infty$ we abbreviate
$$\mathcal{U}^k_T(\alpha)=\big\{F \in C^k_c(V \times S^1):
Ds_F(w)\,\,\textrm{surjective},\,\,\forall\,\,
w \in \mathrm{crit}_T(\mathcal{A}^\alpha_F)\big\}.$$
Since $\mathrm{crit}_T(\mathcal{A}^\alpha_F)$ is compact
it follows that $\mathcal{U}^k_T(\alpha)$ is open in
$C^k_c(V \times S^1)$. Moreover, it follows from Step~1
that if $k<\infty$ it is also dense in $C^k_c(V \times S^1)$.
Since $C^\infty$ is dense in $C^k$ for every $k$, a
diagonal argument shows that
$\mathcal{U}^\infty_T(\alpha)$ is also dense in
$C^\infty_c(V \times S^1)$. It follows that
$$\mathcal{U}(\alpha)=\bigcap_{T \in \mathbb{N}}
\mathcal{U}^\infty_T(\alpha)$$
is of the second category in $C^\infty_c(V \times S^1)$.
This finishes the proof of the proposition.
\end{proof}

If $J \in \mathcal{J}(\alpha)$ is an $\omega$-compatible
almost complex structure and $m=m_J$ is the metric
induced from $J$
we denote by $\mathcal{M}(\mathcal{A}^\alpha_F,J)$
the moduli space of all finite energy gradient flow lines
of $\nabla_m \mathcal{A}^\alpha_F$.
For $a,b \in \mathbb{R}$ we abbreviate
$$\mathcal{M}_a^b(\mathcal{A}^\alpha_F,J)=
\Big\{w \in \mathcal{M}(\mathcal{A}^\alpha_F,J):
\mathcal{A}^\alpha_F(w(s)) \in [a,b],\,\,
\forall\,\,s \in \mathbb{R}\Big\}.$$
We cannot expect that gradient flow lines of the
perturbed action functional are still compact up to breaking.
However, we show that in the case of weakly tame
stable hypersurfaces for small perturbations
there is a decomposition
of $\mathcal{M}_a^b(\mathcal{A}^\alpha_F,J)$ into a disjoint
union of \emph{closed} subsets
$$\mathcal{M}_a^b(\mathcal{A}^\alpha_F,J)=
\underline{\mathcal{M}}_a^b(\mathcal{A}^\alpha_F,J)\sqcup
\overline{\mathcal{M}}_a^b(\mathcal{A}^\alpha_F,J)$$ where
$\underline{\mathcal{M}}_a^b(\mathcal{A}_F^\alpha,J)$ is compact up
to breaking of gradient flow lines. We refer to
$\underline{\mathcal{M}}_a^b(\mathcal{A}^\alpha_F,J)$ as the
\emph{essential part} of the moduli space
$\mathcal{M}_a^b(\mathcal{A}^\alpha_F,J)$ and to
$\overline{\mathcal{M}}_a^b(\mathcal{A}^\alpha_F,J;a,b)$ as the
\emph{unessential part}. The boundary operator then takes only
account of the essential part of the moduli space of gradient flow
lines.

If $\Sigma$ is weakly tame then by definition
for $a \leq b$ the set
$$\mathrm{crit}_a^b(\mathcal{A}^\alpha)=
\big\{w \in \mathrm{crit}(\mathcal{A}^\alpha):
a \leq \mathcal{A}^\alpha(w)\leq b\big\}$$
is compact. Hence we can define
$$\aleph(a,b;\alpha)=\min\Big\{a,\min\big\{\mathcal{A}(w):
w \in \mathrm{crit}_a^b(\mathcal{A}^\alpha)\big\}\Big\}$$
and
$$\beth(a,b;\alpha)=\max\Big\{b,\max\big\{\mathcal{A}^\alpha(w):
w \in \mathrm{crit}_a^b(\mathcal{A}^\alpha)\big\}\Big\}.$$
Note that $\aleph(a,b)$ and $\beth(a,b)$ actually only depend
on $\pi(\alpha)=\lambda \in \Lambda_s(\Sigma,\omega)$ so
that we can set
$$\aleph(a,b;\lambda)=\aleph(a,b;\alpha), \quad
\beth(a,b;\lambda)=\beth(a,b;\alpha).$$ We next introduce a subspace
of perturbations $F$ which have the property that the Lagrange
multiplier along gradient flow lines either becomes large or remains
small. We first introduce the number
$$\gimel(a,b;\lambda)=\bigg(\frac{5}{2}\bigg)^2
\Big(
\beth(a,b;\lambda)-\aleph(a,b;\lambda)\Big)$$
and define the interval
$$I(a,b;\lambda)=[\gimel(a,b;\lambda)+3,
\gimel(a,b;\lambda)+4].$$
We set
$$\widetilde{\mathcal{U}}_a^b(\alpha,J)=\Big\{
F \in C^\infty_c(V\times S^1):||\eta||_\infty \notin
I(a,b;\lambda),\,\,
\forall\,\,w=(v,\eta) \in \mathcal{M}^b_a(
\mathcal{A}^\alpha_F,J)\Big\}.$$
\begin{lemma}\label{app}
Assume that $\Sigma$ is a weakly tame stable hypersurface in a
symplectically aspherical, geometrically bounded, symplectic
manifold $(V,\omega)$ and $\alpha \in \mathfrak{A}(\Sigma,V)$ is a
$\big(\frac{3}{5},\epsilon\big)$-admissible quintuple with $\epsilon
\leq \frac{1}{\gimel(a,b;\lambda)}$, $J \in \mathcal{J}(\alpha)$ and
$a \leq b$. Then $\widetilde{\mathcal{U}}^b_a(\alpha,J)$ is open and
nonempty.
\end{lemma}
\begin{proof}
We first show that $\widetilde{\mathcal{U}}^b_a(\alpha,J)$ is open.
This is actually true for any admissible quintuple
$\alpha$. For openness we prove that the complement
is closed. Hence let
$$F_\nu \in C^\infty_c(V\times S^1)
\setminus \widetilde{\mathcal{U}}^b_a(\alpha,J), \quad
\nu \in \mathbb{N}$$ be a sequence
of perturbations in the complement of
$\widetilde{\mathcal{U}}^b_a(\alpha,J)$ such that
$$\lim_{\nu \to \infty} F_\nu =F \in C^\infty(V\times S^1).$$
It remains to show that $F \notin
\widetilde{\mathcal{U}}^b_a(\alpha,J)$. By definition
of $\widetilde{\mathcal{U}}^b_a(\alpha,J)$ there exists for each
$\nu \in \mathbb{N}$ a gradient flow line
$$w_\nu=(v_\nu,\eta_\nu) \in
\mathcal{M}^b_a(\mathcal{A}^\alpha_{F_\nu},J)$$
such that
\begin{equation}\label{ga}
||\eta_\nu||_\infty \in I(a,b;\lambda).
\end{equation}
Since the Lagrange multiplier is uniformly bounded the
usual compactness arguments in Floer homology
(boundedness at infinity and no bubbling because
of symplectic asphericity) imply that
$w_\nu$ has a convergent subsequence $w_{\nu_j}$ such that
$$\lim_{j \to \infty}w_{\nu_j}=w=(v,\eta)
\in \mathcal{M}^b_a(\mathcal{A}^\alpha_F,J).$$
Since $I(a,b;\lambda)$ is closed it follows from
(\ref{ga}) that
$$||\eta||_\infty \in I(a,b;\lambda).$$
Consequently
$$F \notin \widetilde{\mathcal{U}}^b_a(\alpha,J).$$
This finishes the proof of openness.

To see finally that $\widetilde{\mathcal{U}}^b_a(\alpha,J)$
is nonempty we observe that by our assumptions on
the admissible quintuple $\alpha$ it follows from
Corollary~\ref{obo} that
$$0 \in \widetilde{\mathcal{U}}^b_a(\alpha,J).$$
Hence we are done with the proof of the Lemma.
\end{proof}
\begin{definition}
For a fixed pair $(\alpha,J)$ consisting of an admissible quintuple
$\alpha$ and an $\omega$-compatible almost complex structure
$J \in \mathcal{J}(\alpha)$ and $a \leq b$ we say that a perturbation $F
\in C^\infty_{c}(V\times S^1)$ is \emph{$(a,b)$-approvable} if
$\mathcal{A}^\alpha_F$ is Morse and contained in
$\widetilde{\mathcal{U}}_a^b(\alpha,J)$. We set
$$\mathcal{U}_a^b(\alpha,J)=\{F \in \widetilde{\mathcal{U}}_a^b(\alpha,J):
\mathcal{A}^\alpha_F\,\,\textrm{Morse}\}$$
for the set of $(a,b)$-approvable perturbations.
\end{definition}
Combining Proposition~\ref{morse} with Lemma~\ref{app} we obtain
the following Corollary.
\begin{corollary}
Under the assumptions of Lemma~\ref{app} the set
$\mathcal{U}_a^b(\alpha,J)$ is nonempty.
\end{corollary}
We finally explain how to associate to an $(a,b)$-approvable
perturbation $F$ a homology group $HF_a^b(\mathcal{A}^\alpha_F,J)$.
We first define the set
\begin{equation}\label{tcr}
\mathcal{C}_a^b(\mathcal{A}^\alpha_F)=
\big\{w=(v,\eta) \in \mathrm{crit}_a^b(\mathcal{A}^\alpha_F):
|\eta| < \gimel(a,b;\lambda)+3\big\}.
\end{equation}
Because $\mathrm{crit}_a^b(\mathcal{A}^\alpha_F) \subset
\mathcal{M}_a^b(\mathcal{A}^\alpha_F,J)$ and $F$ is
$(a,b)$-approvable we infer
$$\mathcal{C}_a^b(\mathcal{A}^\alpha_F)=
\big\{w=(v,\eta) \in \mathrm{crit}_a^b(\mathcal{A}^\alpha_F): |\eta|
\leq \gimel(a,b;\lambda)+3\big\}.
$$
 Applying
the Theorem of Arzela-Ascoli to the critical point equation
(\ref{percr}) of $\mathcal{A}^\alpha_F$ we see that
$\mathcal{C}_a^b(\mathcal{A}^\alpha_F)$ is compact. Since
$\mathcal{A}^\alpha_F$ is Morse the set
$\mathcal{C}_a^b(\mathcal{A}^\alpha_F)$ is also discrete and hence
finite. Thus
$$CF_a^b(\mathcal{A}^\alpha_F)=\mathcal{C}^b_a(\mathcal{A}^\alpha_F)
\otimes \mathbb{Z}_2$$
is a finite dimensional $\mathbb{Z}_2$-vector space. We
further define the \emph{essential part} of the moduli space
$\mathcal{M}_a^b(\mathcal{A}^\alpha_F,J)$ to be
$$\underline{\mathcal{M}}_a^b(\mathcal{A}^\alpha_F,J)
=\Big\{w=(v,\eta) \in \mathcal{M}^b_a(\mathcal{A}^\alpha_F,J):
||\eta||_\infty<\gimel(a,b;\lambda)+3\Big\}.$$ Since $F$ is
$(a,b)$-approvable it follows that
$\underline{\mathcal{M}}_a^b(\mathcal{A}^\alpha_F,J)$ is a closed
subset of $\mathcal{M}_a^b(\mathcal{A}^\alpha_F,J)$. Moreover, since
the Lagrange multiplier is by definition uniformly bounded in
$\underline{\mathcal{M}}_a^b(\mathcal{A}^\alpha_F,J)$ standard
arguments in Floer theory using that $(V,\omega)$ is geometrically
bounded and symplectically aspherical imply that
$\underline{\mathcal{M}}_a^b(\mathcal{A}^\alpha_F,J)$ is
$C^\infty_{\mathrm{loc}}$-compact. Hence we can use
$\underline{\mathcal{M}}_a^b(\mathcal{A}^\alpha_F,J)$ to define a
linear map
$$\partial_a^b \colon
CF^b_a(\mathcal{A}^\alpha_F) \to CF^b_a(\mathcal{A}^\alpha_F)$$ by
counting gradient flow lines. Since the metric $m$ induced from $J$
does not necessarily fulfill the Morse-Smale condition one may need
to perturb the gradient flow equation to show that $\partial_a^b
\circ \partial_a^b=0$. It is quite likely that this can actually be
achieved by taking a generic family of $\omega$-compatible almost
complex structures $J_t$. However, we have not checked the details
since this is nowadays not needed anymore in view of the progress of
abstract perturbation theory. For example, if one compactifies the
moduli space of unparametrised trajectories
$\underline{\mathcal{M}}_a^b(\mathcal{A}^\alpha_F,J)/ \mathbb{R}$ by
adding broken gradient flow lines, this compactified space can be
interpreted as the zero-set of a Fredholm-section
$$\varsigma \colon \mathcal{P} \to \mathcal{E}$$
where $\mathcal{P}$ is
an $M$-polyfold and $\mathcal{E}$ is an $M$-polyfold bundle
over $\mathcal{P}$, see
\cite{HWZ1}. If one perturbs this section to make it transverse,
see \cite{HWZ2}, one can define a boundary operator by counting
the perturbed broken gradient trajectories between two critical points.
Indeed, if $\wp$ is a generic abstract perturbation and
$$\varsigma_\wp \colon \mathcal{P} \to \mathcal{E}$$
is the perturbed section, then its zero-set
$$\mathcal{N}(\wp)=\varsigma_\wp^{-1}(0)$$
is a compact manifold with corners. We denote by
$\mathcal{N}_0(\wp)$ its zero-dimensional part. For
a generic perturbation the section $\varsigma_\wp$ is also
transverse with respect to the boundary of the polyfold
$\mathcal{P}$ and hence elements in $\mathcal{N}_0(\wp)$
consist still of unparametrised trajectories
$[w] \in C^\infty(\mathbb{R},\mathcal{L}\times \mathbb{R})/\mathbb{R}$
which are unbroken and converge asymptotically to critical points
in $\mathcal{C}^b_a(\mathcal{A}^\alpha_F)$. Hence for $w^\pm \in
\mathcal{C}^b_a(\mathcal{A}^\alpha_F)$ we abbreviate
$$\mathcal{N}_0(\wp;w^-,w^+)=
\Big\{[w] \in \mathcal{N}_0(\wp): \lim_{s \to \pm \infty}
w(s)=w^\pm\Big\}$$
and introduce the $\mathbb{Z}_2$-number
$$n(w^-,w^+)=\#_2 \mathcal{N}_0(\wp;w^-,w^+)$$
where $\#_2$ denotes the count in $\mathbb{Z}_2$.
We define a linear map
$$\partial_a^b(\mathcal{A}^\alpha_F,J;\wp)
\colon CF^b_a(\mathcal{A}^\alpha_F) \to
CF^b_a(\mathcal{A}^\alpha_F)$$
which is defined on generators by
$$\partial^b_a(\mathcal{A}^\alpha_F,J;\wp)[w^+]=
\sum_{w^- \in \mathcal{C}^b_a(\mathcal{A}^\alpha_F)}
n(w^-,w^+)w^-.$$ Standard arguments show that
$\partial^b_a(\mathcal{A}^\alpha_F,J;\wp)$ is actually a boundary
operator, i.e.~its square is zero. The boundary operator might
depend indeed on the abstract perturbation but the homology
\begin{equation}\label{homol}
HF_a^b(\mathcal{A}^\alpha_F,J)=
\frac{\mathrm{ker}\partial_a^b(\mathcal{A}^\alpha_F,J;\wp)}
{\mathrm{im}\partial_a^b(\mathcal{A}^\alpha_F,J;\wp)}
\end{equation}
is independent of the abstract perturbation by standard
homotopy arguments.

\subsection{Definition of Rabinowitz Floer
  homology}\label{ss:RFH}

In this subsection we assume that $\Sigma$ is a weakly tame stable
hypersurface in a symplectically aspherical, geometrically bounded,
symplectic manifold $(V,\omega)$. In this situation we define
Rabinowitz Floer homology. We further compute it for two relevant
cases. One case is when there are no closed orbits in $\Sigma$
contractible in $V$, and the other is when $\Sigma$ is displaceable.

We define the $\Omega$-spectrum of $\Sigma$ to be
$$\mho(\Sigma)=\{\pm\Omega(v): v\in X(\Sigma)\} \cup \{0\}$$
where we recall that $X(\Sigma)$ denotes the set of closed
characteristics in $\Sigma$ which are contractible in $V$. Since
each $\Omega(v)$ corresponds to a critical value of Rabinowitz
action functional $\mathcal{A}^\alpha$ for an admissible quintuple
$\alpha \in \mathfrak{A}(\Sigma,V)$ the set $\mho(\Sigma)$ is a
meager subset of $\mathbb{R}$ by Sard's theorem.

Our first aim is to define Rabinowitz Floer homology groups
$RFH_a^b$ for $a,b \notin \mho(\Sigma)$. These groups basically
depend only on $\Sigma$ and $V$. However, there is a little subtlety
to note. We do not know if the space of all $\omega$-compatible
geometrically bounded almost complex structures $J$ on $V$
is connected. Therefore Rabinowitz Floer homology actually could
depend on the choice of the geometrically bounded compatible almost
complex structure. We therefore fix one such complex structure
$J_0$. For an admissible quintuple $\alpha$ we abbreviate by
$$\mathcal{J}(\alpha,J_0) \subset \mathcal{J}(\alpha)$$
the set of all $J \in \mathcal{J}(\alpha)$ which outside of a
compact set coincide with $J_0$. We observe that the space
$\mathcal{J}(\alpha,J_0)$ is connected. It is possible that
Rabinowitz Floer homology depends on the choice of $J_0$, although
this seems unlikely. We refer in the following to the triple
$(V,\omega,J_0)$ as the \emph{geometrically bounded symplectic
manifold}. We often skip the reference to $\omega$ and $J_0$ and
just mention $V$.

If $a,b \notin \mho(\Sigma)$, $\alpha \in \mathfrak{A}(\Sigma,V)$
meets the assumptions of Lemma~\ref{app},
and $J \in \mathcal{J}(\alpha,J_0)$ we introduce the subset
of perturbations
$$\widetilde{\mathcal{V}}_a^b(\alpha,J)=
\Big\{F \in \widetilde{\mathcal{U}}_a^b(\alpha,J):
\mathcal{A}^\alpha_F(w) \notin \{a,b\},\,\,\forall\,\, w \in
\mathcal{C}^b_a(\mathcal{A}_F^\alpha,J)\Big\}$$ where we refer to
(\ref{tcr}) for the definition of
$\mathcal{C}_a^b(\mathcal{A}^\alpha_F,J)$. By the Theorem of
Arzela-Ascoli $\widetilde{\mathcal{V}}^b_a(\alpha,J)$ is an open
subset of $\widetilde{\mathcal{U}}^b_a(\alpha,J)$. Moreover, since
$a,b \notin \mho(\Sigma)$ the zero perturbation is contained in
$\widetilde{\mathcal{V}}^b_a(\alpha,J)$. Hence we abbreviate by
$$\widehat{\mathcal{V}}_a^b(\alpha,J)
\subset \widetilde{\mathcal{V}}_a^b(\alpha,J)$$ the connected
component of $\widetilde{\mathcal{V}}(\alpha,J)$ containing $0$. We
set
$$\mathcal{V}_a^b(\alpha,J)=\{F \in
\widehat{\mathcal{V}}_a^b(\alpha,J): \mathcal{A}^\alpha_F\,\,
\textrm{Morse}\}.$$
By Proposition~\ref{morse} and Lemma~\ref{app} the set
$\mathcal{V}_a^b(\alpha,J)$ is non-empty. Hence we
pick $F \in \mathcal{V}_a^b(\alpha,J)$ and define
 $$RFH_a^b(\Sigma,V)=HF_a^b(\mathcal{A}^\alpha_F,J)$$
where the right-hand side was defined in (\ref{homol}). It is
straightforward to check that this definition is independent of the
choices of $F$, $J$ and $\alpha$. For this we actually use that the
space $\mathcal{J}(\alpha,J_0)$ is connected.

There are canonical homomorphisms between Rabinowitz Floer
homology groups
$$\pi^b_{a_2,a_1} \colon RFH^b_{a_1} \to RFH^b_{a_2},\quad a_1
\leq a_2 \leq b,$$
and
$$\iota^{b_2,b_1}_a \colon RFH^{b_1}_a \to RFH^{b_2}_a,\quad a
\leq b_1 \leq b_2.$$
These maps satisfy
\begin{equation}\label{bidi}
\left.\begin{array}{cc}
\pi^b_{a_3,a_2} \circ \pi^b_{a_2,a_1}=\pi^b_{a_3,a_1},&
 a_1 \leq a_2 \leq a_3 \leq b,\\
 & \\
\iota^{b_3,b_2}_a \circ \iota^{b_2,b_1}_a=
\iota^{b_3,b_1}_a,& a\leq b_1 \leq b_2 \leq b_3,\\
& \\
\iota^{b_2,b_1}_{a_2} \circ \pi^{b_1}_{a_2,a_1}=
\pi^{b_2}_{a_2,a_1} \circ \iota^{b_2,b_1}_{a_1}, &
a_1 \leq a_2 \leq b_1 \leq b_2.
\end{array}\right\}
\end{equation}
In particular, the Rabinowitz Floer homology groups
together with these maps have the
structure of a bidirect system of $\mathbb{Z}_2$-vector
spaces.

We next describe the construction of the
map $\pi^b_{a_2,a_1}$.
Assume that
$$a_1 \leq a_2 \leq b\quad
a_1,a_2,b \notin \mho(\Sigma).$$
We first pick a small stabilizing 1-form
$\lambda \in \Lambda_s(\Sigma,V)$.
We note that
\begin{equation}\label{gim}
\gimel(a_1,b;\lambda) \geq \gimel(a_2,b;\lambda).
\end{equation}
We pick a $\big(\frac{3}{5},\epsilon\big)$-admissible quintuple
$\alpha$ with $\pi(\alpha)=\lambda$ and $\epsilon \leq
\frac{1}{\gimel(a_1,b;\lambda)}$. Note that because of (\ref{gim})
this quintuple meets the assumptions of Lemma~\ref{app} for
$(a_1,b)$ as well as for $(a_2,b)$. We introduce the closed interval
$$I(a_1,a_2,b;\lambda)=[\gimel(a_2,b;\lambda)+3,
\gimel(a_1,b;\lambda)+4] \supset
I(a_2,b;\lambda).$$
We set
$$\widetilde{\mathcal{U}}_{a_2,a_1}^b(\alpha,J)=
\Big\{F \in \widetilde{\mathcal{U}}^b_{a_1}(\alpha,J):
||\eta||_\infty \notin I(a_1,a_2,b;\lambda),\,\,\forall\,\,
(v,\eta) \in \mathcal{M}^b_{a_2}(\mathcal{A}^\alpha_F,J)\Big\}.$$
A similar reasoning as in Lemma~\ref{app} shows that
$\widetilde{\mathcal{U}}^b_{a_2,a_1}(\alpha,J)$ is an
open set of $C^\infty_c(V \times S^1)$
containing the zero perturbation.
We define
$$\widetilde{\mathcal{V}}^b_{a_2,a_1}(\alpha,J)
=\Big\{F \in \widetilde{\mathcal{U}}^b_{a_2,a_1}(\alpha,J):
\mathcal{A}^\alpha_F(w) \notin \{a_1,a_2,b\},\,\,
\forall\,\,w \in \mathcal{C}^b_{a_1}(\mathcal{A}^\alpha_F,J)\Big\}.$$
Again this is an open subset of $C^\infty_c(V \times S^1)$
containing the zero perturbation. Let
$$\widehat{\mathcal{V}}^b_{a_2,a_1}(\alpha,J)\subset
\widetilde{\mathcal{V}}^b_{a_2,a_1}(\alpha,J)$$
be the connected component containing $0$ and set
$$\mathcal{V}^b_{a_2,a_1}(\alpha,J)=
\{F \in \widehat{\mathcal{V}}^b_{a_2,a_1}(\alpha,J):
\mathcal{A}^\alpha_F\,\,\textrm{Morse}\}.$$
Note that $\mathcal{V}^b_{a_2,a_1}(\alpha,J)$ is non-empty
and satisfies
$$\mathcal{V}^b_{a_2,a_1}(\alpha,J)\subset
\mathcal{V}^b_{a_1}(\alpha,J)\cap \mathcal{V}^b_{a_2}(\alpha,J).$$
We pick
$$F \in \mathcal{V}^b_{a_2,a_1}(\alpha,J).$$
Abbreviate
$$\mathcal{C}^{a_2}_{a_1}(\mathcal{A}^\alpha_F,b)=\big\{w
=(v,\eta) \in \mathrm{crit}_{a_1}^{a_2}(\mathcal{A}^\alpha_F):
|\eta|<\gimel(a_1,b;\lambda)+3\big\}.$$
Since $a_2$ is not a critical value of $\mathcal{A}^\alpha_F$
and there is no critical point
$w=(v,\eta) \in \mathrm{crit}_{a_2}^b(\mathcal{A}^\alpha_F)
\subset \mathcal{M}_{a_2}^b(\mathcal{A}^\alpha_F,J)$
with $\eta \in I(a_1,a_2,b;\lambda)$,
we have a disjoint union
$$\mathcal{C}^b_{a_1}(\mathcal{A}^\alpha_F)
=\mathcal{C}^b_{a_2}(\mathcal{A}^\alpha_F) \sqcup
\mathcal{C}^{a_2}_{a_1}(\mathcal{A}^\alpha_F,b).$$
This leads to a direct sum of $\mathbb{Z}_2$-vector spaces
$$CF_{a_1}^b(\mathcal{A}^\alpha_F)=
CF_{a_2}^b(\mathcal{A}^\alpha_F) \oplus
CF_{a_1}^{a_2}(\mathcal{A}^\alpha_F,b)$$
with
$$CF_{a_1}^{a_2}(\mathcal{A}^\alpha_F,b)=
\mathcal{C}^{a_2}_{a_1}(\mathcal{A}^\alpha_F,b)\otimes \mathbb{Z}_2.
$$
Let
$$p^b_{a_2,a_1}
\colon CF^b_{a_1}(\mathcal{A}^\alpha_F)\to
CF^b_{a_2}(\mathcal{A}^\alpha_F)$$ be the projection along
$CF^{a_2}_{a_1}(\mathcal{A}^\alpha_F,b)$. Using the fact that the
action is increasing along gradient flow lines we observe that the
projection $p^b_{a_2,a_1}$ commutes with the boundary operators and
hence induces a homomorphism
$$\pi^b_{a_2,a_1} \colon HF^b_{a_1}(\mathcal{A}^\alpha_F,J)
\to HF^b_{a_2}(\mathcal{A}^\alpha_F,J).$$
A usual homotopy argument shows that this homomorphism
is independent of $F$, $\alpha$, and $J$ and hence
can be interpreted as homomorphism
$$\pi^b_{a_2,a_1} \colon RFH^b_{a_1} \to RFH^b_{a_2}.$$
The construction of the homomorphism
$$\iota^{b_2,b_1}_a \colon RFH^{b_1}_a \to RFH^{b_1}_a$$
is similar and will not be carried out here. We also omit
the proof of (\ref{bidi}).

Given the bidirect system of $\mathbb{Z}_2$-vector spaces
$(RFH,\pi,\iota)$ we can extract out of it two Rabinowitz Floer
homology groups
$$\overline{RFH}=\lim_{\substack{\longrightarrow\\ b \to
\infty}}\lim_{\substack{\longleftarrow\\
a \to -\infty}}RFH_a^b$$
and
$$\underline{RFH}=\lim_{\substack{\longleftarrow\\ a \to
-\infty}}\lim_{\substack{\longrightarrow\\
b \to \infty}}RFH_a^b.$$ Since inverse and direct limits do not
necessarily commute it is an open problem if the two Rabinowitz
Floer homology groups coincide. However, it is well known (see
\cite[page 215]{ML}) that there is a canonical homomorphism which
takes account of the interchange of inverse and direct limits
$$\kappa \colon \overline{RFH} \to \underline{RFH}.$$
In general the canonical homomorphism does not need to be an
isomorphism, although we have no example where $\kappa$ fails to be
an isomorphism in Rabinowitz Floer homology. We finish this
subsection by computing the Rabinowitz Floer homology groups in two
easy examples.
\begin{theorem} Suppose that $\Sigma$ is
a weakly tame, stable hypersurface in a
symplectically aspherical, geometrically bounded, symplectic
manifold $(V,\omega)$.
\begin{description}
\item[(i)] Assume that there are no closed characteristics
in $\Sigma$ which are contractible in $V$, then
$$\overline{RFH}(\Sigma,V)=\underline{RFH}(\Sigma,V)
=H(\Sigma;\mathbb{Z}_2)$$
where $H$ denotes (ungraded) singular homology.
\item[(ii)] Assume that $\Sigma$ is displaceable in $V$, then
$$\overline{RFH}(\Sigma,V)=\underline{RFH}(\Sigma,V)=\{0\}.$$
\end{description}
\label{thm:rfh=0}
\end{theorem}
\begin{proof}
We first prove assertion (i). If there are no closed characteristics
on $\Sigma$ which are contractible in $V$, the unperturbed
Rabinowitz action functional $\mathcal{A}^\alpha$ is Morse-Bott with
critical manifold given by $\{(x,0)\}$ where $x \in \Sigma$ is
thought of as a constant loop. Since $\mathcal{A}^\alpha$ is
constant on the critical manifold its homology is just the homology
of the critical manifold, hence the homology of the hypersurface
$\Sigma$.

We next show assertion (ii). A reasoning similar as in
the proof of Theorem~\ref{felix} shows that if
$d \geq e(\Sigma)$, the displacement energy of $\Sigma$, and
$a \leq b$, then
$$\pi^{b+d}_{a+d,a} \circ \iota^{b+d,b}_a=0
\colon RFH_a^b(\Sigma,V) \to RFH^{b+d}_{a+d}(\Sigma,V).$$
This implies the vanishing of the two Rabinowitz Floer
homology groups.
\end{proof}

\subsection{Invariance}\label{ss:invariance}

In this subsection we show that Rabinowitz Floer homology
is invariant under stable tame  homotopies.
\begin{theorem}\label{inva}
Assume that $\mathcal{S}=(\Sigma_\zeta,\lambda_\zeta)$ for
$\zeta \in [0,1]$ is a
stable tame homotopy. Then there are isomorphisms
$$\overline{\Phi}=
\overline{\Phi}_{\mathcal{S}}
\colon \overline{RFH}(\Sigma_0,V) \to
\overline{RFH}(\Sigma_1,V)$$
and
$$\underline{\Phi}=
\underline{\Phi}^{\mathcal{S}}
\colon \underline{RFH}(\Sigma_0,V) \to
\underline{RFH}(\Sigma_1,V)$$
such that the following diagram commutes
\begin{equation}\label{dia}
\begin{xy}
 \xymatrix{
  \overline{RFH}(\Sigma_0,V) \ar[r]^{\overline{\Phi}}
  \ar[d]_\kappa &
  \overline{RFH}(\Sigma_1,V) \ar[d]^\kappa \\
  \underline{RFH}(\Sigma_0,V) \ar[r]^{\underline{\Phi}}         &
  \underline{RFH}(\Sigma_1,V).
 }
\end{xy}
\end{equation}
\end{theorem}

\emph{Proof. }
Given a stable homotopy $\mathcal{S}=(\Sigma_\zeta,\lambda_\zeta)$
there exists
a smooth family of positive real numbers $\sigma_\zeta>0$,
such that $\sigma_\zeta \lambda_\zeta$ is small for every
$\zeta \in [0,1]$ and there exists a smooth path
$\tau_\zeta
\in \mathcal{T}_\ell(\Sigma_\zeta,\sigma_\zeta\lambda_\zeta)$
of large tubular neighbourhoods and a smooth path
$I_\zeta \in \mathcal{I}_s(\Sigma_\zeta,\lambda_\zeta)$
of small $\omega_{\Sigma_\zeta}$-compatible complex structures
such that
$$\mathfrak{P}=(\Sigma_\zeta,\sigma_\zeta\lambda_\zeta,\tau_\zeta,
I_\zeta)$$
is a smooth path of stable quadruples. We further remark that
the stable homotopy $(\Sigma_\zeta,\sigma_\zeta\lambda_\zeta)$
is still tame.

Since each path of stable quadruples can be obtained by
concatenating short paths it actually suffices to prove the theorem
only for  short paths of stable quadruples. We need the following
Lemma. We
refer the reader to (\ref{mah}) for the definition
of $\mathcal{H}(\alpha^-,\alpha^+)$ and to Theorem~\ref{tdc}
for the definition of $\Delta\big(\frac{1}{2}\big)$.
\begin{lemma}\label{acti}
Assume that
$\mathfrak{P}=(\Sigma_\zeta,\lambda_\zeta,\tau_\zeta,I_\zeta)$
for $\zeta \in [0,1]$ is a short path of stable quadruples
such that $(\Sigma_\zeta,\lambda_\zeta)$ is tame with
taming constant $c$ and $\mathfrak{P}$ is short in the sense that
\begin{equation}\label{sh1}
\Delta(\mathfrak{P})\leq \min\bigg\{\frac{1}{144\cdot\max\{1,c\}},
\Delta\bigg(\frac{1}{2}\bigg)\bigg\}.
\end{equation}
Then there exist admissible
quintuples $\alpha^-$ for $\Sigma_0$ and
$\alpha^+$ for $\Sigma_1$, a time dependent Hamiltonian
$H \in \mathcal{H}(\alpha^-,\alpha^+)$, and
a time-dependent metric $m=\{m_s\}_{s \in \mathbb{R}}$
on $\mathcal{L}\times \mathbb{R}$ which is constant
for $|s|$ large such that the following holds true.
Suppose that $w=(v,\eta) \in C^\infty(\mathbb{R},\mathcal{L}\times
\mathbb{R})$ is a gradient flow line of the time
dependent gradient $\nabla_{m_s}\mathcal{A}^{H_s}$
which converges asymptotically
$\lim_{s\to \pm \infty}w(s)=w^\pm$ to critical points
of $\mathcal{A}^{\alpha^\pm}$, such that
$a=\mathcal{A}^{\alpha^-}(w^-)$ and
$b=\mathcal{A}^{\alpha^+}(w^+)$. Then the following holds
\begin{description}
 \item[(i)] If $a \geq \frac{1}{9}$, then $b \geq \frac{a}{2}$.
 \item[(ii)] If $b \leq -\frac{1}{9}$, then $a \leq \frac{b}{2}$.
\end{description}
\end{lemma}
\begin{proof}
We choose $H$ as in Theorem~\ref{tdc}. Combining tameness
with Theorem~\ref{tdc} for $\epsilon=\frac{1}{2}$ and using
(\ref{sh1}) we obtain
\begin{equation}\label{sh2}
||\eta||_\infty \leq 9\max\{c,1\}|b-a|+2.
\end{equation}
Moreover, by letting $\sigma$ tend to $\infty$ in Step~1 of
Theorem~\ref{tdc} we obtain
\begin{equation}\label{sh3}
b \geq a-2\Delta(\mathfrak{P})||\eta||_\infty.
\end{equation}
$\,$From (\ref{sh1})-(\ref{sh3}) we extract
\begin{equation}\label{sh4}
b \geq a-\frac{|b-a|}{8}-\frac{1}{36}.
\end{equation}
To prove assertion (i) we first consider the case
$$|b| \leq a, \quad a \geq \frac{1}{9}.$$
In this case we estimate
$$b \geq a-\frac{|a|}{4}-\frac{1}{36}=\frac{3a}{4}-\frac{1}{36}
\geq \frac{a}{2}.$$
Hence to prove assertion (i) it suffices to exclude the case
\begin{equation}\label{sh5}
-b \geq a \geq \frac{1}{9}.
\end{equation}
But in this case (\ref{sh4}) leads to a contradiction in
the following way
$$b \geq \frac{1}{9}-\frac{|b-a|}{8}-\frac{1}{36}
\geq -\frac{|b-a|}{8} \geq \frac{b}{4}$$
implying that $\frac{3b}{4} \geq 0$ and hence $b \geq 0$
contradicting (\ref{sh5}). This proves the first assertion.

To prove assertion (ii) we set
$$b'=-a, \quad a'=-b.$$
We note that if (\ref{sh4}) holds for $a$ and $b$, it also
holds for $b'$ and $a'$. Hence we get from assertion (i)
the implication
$$-b \geq \frac{1}{9} \quad \Longrightarrow \quad
-a \geq -\frac{b}{2}$$
which is equivalent to assertion (ii). This finishes the
proof of the Lemma.
\end{proof}
\textbf{Proof of Theorem~\ref{inva} continued: }In view
of Lemma~\ref{acti} we obtain for $a \leq -\frac{1}{9}$ and
$b \geq \frac{1}{9}$ homomorphisms
$$\Phi^b_a \colon RFH_a^{\frac{b}{2}}(\Sigma_0,V)
\to RFH_{\frac{a}{2}}^b(\Sigma_1,V)$$ defined by counting gradient
flow lines of the time dependent Rabinowitz action functional. Again
to count gradient flow lines we have to choose small perturbations
which make Rabinowitz action functional Morse and then take the
count (possibly after a further abstract perturbation) of the
essential part of the moduli space of gradient flow lines of the
perturbed time dependent functional. These homomorphisms interchange
the maps $\pi$ and $\iota$ and hence induce homomorphisms
$$\overline{\Phi} \colon \overline{RFH}(\Sigma_0,V)
\to \overline{RFH}(\Sigma_1,V)$$
and
$$\underline{\Phi} \colon \underline{RFH}(\Sigma_0,V)
\to \underline{RFH}(\Sigma_1,V)$$
such that (\ref{dia}) commutes.

It remains to show that $\overline{\Phi}$ and
$\underline{\Phi}$ are isomorphisms.
By using the homotopy backwards
we get homomorphisms
$$\Psi^b_a \colon RFH_a^{\frac{b}{2}}(\Sigma_1,V)
\to RFH_{\frac{a}{2}}^b(\Sigma_0,V).$$
A standard homotopy of homotopy argument shows that
for $a \leq -\frac{2}{9}$ and $b \geq \frac{2}{9}$ the
composition of $\Phi$ and $\Psi$ is given by
$$\Psi_{\frac{a}{2}}^b \circ \Phi_a^{\frac{b}{2}}=
\pi^b_{\frac{a}{4},a} \circ \iota^{b,\frac{b}{4}}_a
\colon RFH_a^{\frac{b}{4}}(\Sigma_0,V) \to
RFH_{\frac{a}{4}}^b(\Sigma_0,V)$$
and analoguously
$$\Phi_{\frac{a}{2}}^b \circ \Psi_a^{\frac{b}{2}}=
\pi^b_{\frac{a}{4},a} \circ \iota^{b,\frac{b}{4}}_a
\colon RFH_a^{\frac{b}{4}}(\Sigma_1,V) \to
RFH_{\frac{a}{4}}^b(\Sigma_1,V).$$
Hence we conclude that
$$\overline{\Psi} \circ \overline{\Phi}=
\mathrm{id}|_{\overline{RFH}(\Sigma_0,V)}, \quad
\overline{\Phi} \circ \overline{\Psi}=
\mathrm{id}|_{\overline{RFH}(\Sigma_1,V)}$$
and
$$\underline{\Psi} \circ \underline{\Phi}=
\mathrm{id}|_{\underline{RFH}(\Sigma_0,V)}, \quad
\underline{\Phi} \circ \underline{\Psi}=
\mathrm{id}|_{\underline{RFH}(\Sigma_1,V)}.$$
This implies that $\overline{\Phi}$ and
$\underline{\Phi}$ are isomorphisms with inverses
$$\overline{\Phi}^{-1}=\overline{\Psi},\quad
\underline{\Phi}^{-1}=\underline{\Psi}.$$
This finishes the proof of the theorem. \hfill $\square$

\section{Ma\~n\'e's critical values}\label{sec:mane}

In this section we summarize the main properties associated
with the various Ma\~n\'e's critical values. We also show that
there are open sets of hypersurfaces with high energy which
are not stable; however these are virtually contact, showing the
need to consider the latter notion.

\subsection{Definition and basic properties}
As in the Introduction we consider the cotangent bundle $\tau\co T^*M\to M$
 of a closed manifold $M$ and the autonomous Hamiltonian system
 defined by a convex Hamiltonian
$$
   H(q,p) = \frac{1}{2}|p|^2 + U(q)
$$
and a twisted symplectic form
$$
   \om_{\sigma}=\om = dp\wedge dq + \tau^*\sigma.
$$
Here $dp\wedge dq$ is the canonical symplectic form in canonical
coordinates $(q,p)$ on $T^*M$, $|p|$ denotes
the dual norm of a Riemannian metric $g$ on $M$, $U\co M\to\R$ is a
smooth potential, and $\sigma$ is a closed 2-form on $M$. This
Hamiltonian system describes the motion of a particle on $M$ subject
to the conservative force $-\nabla U$ and the magnetic field
$\sigma$. In local coordinates $q_1,\dots,q_n$ on $M$ and dual
coordinates $p_1,\dots,p_n$ the Hamiltonian system is given by
\begin{align*}
   \dot q_i &= \frac{\p H}{\p p_i}, \cr
   \dot p_i &= -\frac{\p H}{\p q_i} +
   \sum_{j=1}^n\sigma_{ij}(q)\frac{\p H}{\p p_j},
\end{align*}
where
$$
   \sigma = \frac{1}{2}\sum_{i,j=1}^n\sigma_{ij}(q)dq_i\wedge
   dq_j,\qquad \sigma_{ij}=-\sigma_{ji}.
$$
In particular, the $q$-components of the Hamiltonian vector field
$X_H$ are independent of $\sigma$,
$$
   X_H = (H_p,*),\qquad H_p=\frac{\p H}{\p p}.
$$
Let $\Pi\co \widehat{M}\to M$ be a cover of $M$ and suppose that
$\Pi^*\sigma$ is {\it exact}. The Hamiltonian $H$ lifts to Hamiltonian
$\widehat{H}$ and we define the
{\em Ma\~n\'e critical value} of the cover as:

$$
   c(\widehat{H}) := \inf_{\theta}\sup_{q\in \widehat{M}} \widehat{H}(q,\theta_q),
$$
where the infimum is taken over all 1-forms $\theta$ on $\widehat M$
with $d\theta=\Pi^*\sigma$.

 If $\overline{M}$ is a cover of $\widehat{M}$,
then clearly $c(\overline{H})\leq c(\widehat{H})$ and equality
holds if $\overline{M}$ is a {\it finite} cover of $\widehat{M}$.

The critical value may also be defined in Lagrangian terms.
Consider the Lagrangian on $T\widehat{M}$ given by
\[\widehat{L}(q,v)=\frac{1}{2}|v|^{2}-U(q)+\theta_{q}(v),\]
where $\theta$ is any primitive of $\Pi^*\sigma$.
It is well known that the extremals of $\widehat{L}$, i.e., the solutions of
the Euler-Lagrange equations of $\widehat{L}$,
 $$
 \frac{d}{dt}\frac{\partial \widehat{L}}{\partial v}(q,v) = \frac{\partial \widehat{L}}{\partial q}(q,v)
$$
coincide with the projection to $\widehat{M}$ of the orbits of the Hamiltonian
$\widehat{H}$.
The action of the Lagrangian $\widehat{L}$ on
an absolutely continuous curve $\gamma\co [a,b]\rightarrow \widehat M$ is defined by
\[A_{\widehat{L}}(\gamma)=\int_{a}^{b}\widehat{L}(\gamma(t),\dot{\gamma}(t))\,dt.\]
We define the Ma\~n\'e critical value of the Lagrangian $\widehat{L}$
as (this was Ma\~n\'e's original definition):
\[c(\widehat{L})=\inf\{k\in\R:\;A_{\widehat{L}+k}(\gamma)\geq 0\;\mbox{\rm
for any absolutely continuous closed curve $\gamma$}\]
\[\mbox{\rm defined on any closed interval $[0,T]$ }\}.\]
Note that $c(\widehat{L})$ may depend on the primitive $\theta$
chosen, but it will remain unchanged for all primitives of the form
$\theta+df$, thus it could only change if we consider another primitive
$\theta'$ such that $\theta-\theta'$ determines a non-zero class in
$H^{1}(\widehat{M},\R)$.
The relationship between the Lagrangian and Hamiltonian critical values
is given by (cf. \cite{BurPa,CIPP}):
\[c(\widehat{H})=\inf_{[\varpi]\in H^{1}(\widehat{M},\R)}c(\widehat{L}-\varpi).\]

There are two covers which are distinguished: the {\em universal
  cover} $\pi\co \widetilde{M}\to M$ and the {\em abelian cover}
  $\pi_{0}\co M_{0}\to M$.
The latter is defined as the cover of $M$ whose fundamental group
is the kernel of the Hurewicz homomorphism $\pi_{1}(M)\mapsto
H_{1}(M,\R)$ (we could also take coefficients in $\Z$; this will not alter
the discussion below since critical values are unchanged by finite covers).
We denote the lifts of $H$ to $\widetilde M$ resp.~$M_0$
by $\widetilde H$ resp.~$H_0$ and the corresponding Ma\~n\'e
critical values (when defined) by
$$
   c := c(\widetilde H) = c(\widetilde L),\qquad c_0 := c(H_0) = c(L_0).
$$
The critical value
$c=c(\widetilde{H})=c(\widetilde{L})$ is the one given
in the Introduction. We note here some of its properties:
\begin{enumerate}
\item $c<\infty$ if and only if $\pi^*\sigma$ has a {\it bounded} primitive
(with respect to the lifted Riemannian metric);
\item if $M$ admits a metric of negative curvature, any closed 2-form
$\sigma$ has bounded primitives in $\widetilde{M}$ \cite{Gro};
\item if $[\sigma]\in H^{2}(M,\R)$ is non-zero and $\pi_{1}(M)$
is amenable
\footnote{A group $G$ is amenable e.g.~if it is abelian, solvable or
  nilpotent. If $G$ contains a free subgroup on two generators (which
  is the case e.g.~for the fundamental group of a closed surface of
  genus at least 2) it is not amenable.},
$c=\infty$ (\cite{Gro}, see also \cite[Corollary 5.4]{P}).
\end{enumerate}

\begin{lemma}
For any $k>c$, the hypersurface $\Sigma_k$ is virtually
contact.
\label{virtc}
\end{lemma}

\begin{proof}If $k>c$ we may choose $\varepsilon>0$ and a primitive $\theta$ of
$\pi^*\sigma$ such that
\begin{equation}
\varepsilon+|\theta_{q}|\leq \sqrt{2(k-\widetilde{U}(q))}
\label{fromc}
\end{equation}
for all $q\in \widetilde{M}$.

Let $\lambda$ be the Liouville form on $\widetilde{M}$. Then we may write
$\widetilde{\omega}=d(\lambda+\widetilde{\tau}^*\theta).$
Since $X_{\widetilde{H}}=(\widetilde{H}_{p},*)$, on $\widetilde{\Sigma}_{k}$
we have
\begin{equation}
(\lambda+\widetilde{\tau}^*\theta)(X_{\widetilde{H}})=
|p|^2+\theta_{q}(\widetilde{H}_{p})=2(k-\widetilde{U}(q))+\theta_{q}(\widetilde{H}_{p}).
\label{primi}
\end{equation}
But $\theta_{q}(\widetilde{H}_{p})\geq -|\theta_{q}|\sqrt{2(k-\widetilde{U}(q))}$ for all $q\in\widetilde{M}$. It follows from (\ref{fromc}) and (\ref{primi})
that
\begin{align*}
(\lambda+\widetilde{\tau}^*\theta)(X_{\widetilde{H}})&\geq
2(k-\widetilde{U}(q))-|\theta_{q}|\sqrt{2(k-\widetilde{U}(q))}\\
&\geq \sqrt{2(k-\widetilde{U}(q))}\left(\sqrt{2(k-\widetilde{U}(q))}-|\theta_{q}|\right)\\
&\geq \varepsilon^2.
\end{align*}
On the other hand (\ref{fromc}) also implies that there is a constant
$C>0$ such that
$$\sup_{(q,p)\in\widetilde{\Sigma_{k}}}|(\lambda+\widetilde{\tau}^*\theta)_{(q,p)}|<C$$
and thus $\Sigma_k$ is virtually contact.
\end{proof}

\subsection{Hyperbolic spaces}\label{ss:hyp}
The results in this subsection are well known; we include them here for
the reader's convenience (cf. \cite{G1}).
Let $\Gamma$ be a cocompact lattice of $PSL(2,\R)$.
The standard horocycle flow $h_t$ is given by the right action of the
 one-parameter subgroup
\[\left(\begin{array}{cc}
1&t\\
0&1\\
\end{array}\right)   \]
on $\Gamma\backslash PSL(2,\R)$.
The horocycle flow is known to be uniquely ergodic \cite{Fu}.
Morever, it has zero topological entropy
since
\[\phi^0_t\circ h_s=h_{se^{-t}}\circ \phi^0_t\]
for all $s,t\in \R$, where $\phi^0$ is the geodesic flow given by the
one-parameter subgroup
\[\left(\begin{array}{cc}
e^{t/2}&0\\
0&e^{-t/2}\\
\end{array}\right).   \]
(The relation $\phi^0_t\circ h_s=h_{se^{-t}}\circ \phi^0_t$ implies
that for any $t$, the flows $s\mapsto h_{s}$ and $s\mapsto h_{se^{-t}}$
are conjugate and thus they must have the same entropy.
But for $t\neq 0$, this forces the entropy of $h_s$ to be zero.)

In fact, $h_t$ parametrizes the strong stable manifolds of $\phi^0$.
A matrix $X$ in $\mathfrak{sl}(2,\R)$ gives rise to a flow $\phi$ on $\Gamma\backslash PSL(2,\R)$
by setting
\[\phi_{t}(\Gamma g)=\Gamma g e^{tX}.\]
The geodesic and horocycle flows are just particular cases of these algebraic flows.
Consider the following path of matrices in $\mathfrak{sl}(2,\R)$:
\[\R\ni s\mapsto X_{s}:=\left(\begin{array}{cc}
1/2&0\\
0&-1/2\\
\end{array}\right)+s\,\left(\begin{array}{cc}
0&1/2\\
-1/2&0\\
\end{array}\right).  \]
The flows $\phi^{s}$ on $\Gamma\backslash PSL(2,\R)$ associated with the matrices $X_{s}$ can be interpreted as magnetic flows.
Since $PSL(2,\R)$ acts by isometries on the hyperbolic
plane ${\mathbb H}^{2}$, $M:=\Gamma\backslash {\mathbb H}^{2}$ is a compact
hyperbolic surface provided $\Gamma$ acts without fixed points, and the unit sphere bundle $SM$ of $M$ can be identified
with $\Gamma\backslash PSL(2,\R)$.
A simple calculation shows that $\phi^{s}$ is the
Hamiltonian flow of the Hamiltonian
$H(x,v)=\frac{1}{2}|v|^{2}_{x}$ with respect to the twisted symplectic
form on $TM$ given by
\[\omega_{s}=d\alpha-s\,\pi^{*}\sigma,\]
where $\sigma$ is the area form of $M$, $\pi\co TM\to M$ is
the canonical projection and $\alpha$ is the Liouville 1-form that generates
 the geodesic flow of $M$.
For $s=0$, $\phi^{0}$ is the geodesic flow and for $s=1$, $\phi^{1}$ is the flow
induced by the one-parameter subgroup with matrix on $\mathfrak{sl}(2,\R)$ given by
\[X_1=\left(\begin{array}{cc}
1/2&1/2\\
-1/2&-1/2\\
\end{array}\right). \]
Observe that there exists an element $c\in PSL(2,\R)$ such that
\[ c^{-1}X_1 c=\left(\begin{array}{cc}
0&1\\
0& 0\\
\end{array}\right). \]
Explicitly
\[c=\frac{1}{\sqrt{2}}\left(\begin{array}{cc}
1&1\\
-1&1\\
\end{array}\right). \]

Then the map $f\co \Gamma\backslash PSL(2,\R)\to \Gamma\backslash PSL(2,\R)$
given by $f(\Gamma g)=\Gamma gc$ conjugates
$\phi^1$ and $h$, i.e. $f\circ \phi^{1}_{t}=h_{t}\circ f$. (In fact,
any matrix
in $\mathfrak{sl}(2,\R)$ with determinant zero will give rise to a
flow which is conjugate to $h_t$ or $h_{-t}$. So, up to orientation,
there is just one algebraic horocycle flow.)
Note that $\det X_{s}=-\frac{1}{4}(1-s^2)$, so for $|s|<1$,
the flow $\phi^{s}$ is conjugate to the geodesic flow $\phi^0$, up to
a constant time scaling by $\sqrt{1-s^2}$.
Hence the magnetic flows $\phi^s$ are just geodesic
flows, but with entropy $\sqrt{1-s^2}$. Similarly, for $|s|>1$, $\phi^s$
is up to a constant time scaling conjugate to the flow generated by
\[V:=\left(\begin{array}{cc}
0&1/2\\
-1/2&0\\
\end{array}\right),  \]
which is actually the fibrewise circle flow.

In the discussion above we have kept the hypersurface $SM$ ($k=1/2$) fixed
 and changed $\sigma$ by $s\,\sigma$. This was just done for convenience, since
the flow $\phi^s$ on $SM$ is equivalent to the Hamiltonian flow
on $\Sigma_k$ with respect to $\sigma$, where $s^2=1/2k$.
Note that $\omega_{s}$ restricted to $SM$ has a primitive given by
$\psi_{s}:=\alpha-s\gamma$, where $\gamma$ is the unique left-invariant 1-form
which
takes the value one on $V$ and zero on
\[ \left(\begin{array}{cc}
1/2&0\\
0&-1/2\\
\end{array}\right),\;\;\;\;\left(\begin{array}{cc}
0&-1/2\\
-1/2&0\\
\end{array}\right).\]
The 1-form $\gamma$ is also the connection 1-form determined by the
Levi-Civita connection of the hyperbolic metric.
Also note that $\psi_{s}(X_{s})=1-s^2$ and thus $\Sigma_{k}$ is of contact
type for all $k\neq 1/2$ (with opposite orientation for $k<1/2$).

The fact that the Ma\~n\'e critical value is $c=1/2$ can be seen as
follows. The form
$y^{-1}dx$ is a primitive of the area form $y^{-2} dx\wedge dy$
and has norm $1$. This shows that $c\leq 1/2$.
To see that in fact $c=1/2$ one performs a calculation with geodesic
circles exactly as the one in the proof of Lemma \ref{lemma:cc_0}
below.

Summarizing, we have obtained the picture described in the
Introduction:

\begin{itemize}
\item For $k>1/2$, the dynamics is Anosov and conjugate (after rescaling)
to the underlying geodesic flow. The energy levels are contact.
\item At the Ma\~n\'e critical value $k=1/2$ we hit the horocycle
flow. There are no closed orbits and the level is unstable.
\item For $k<1/2$ all orbits are closed and contractible. Energy levels
are contact but with opposite orientation.
\end{itemize}

A very similar picture is obtained if we take compact quotients of complex
hyperbolic space $\mathbb H_{\bf C}^n$ with its K\"ahler form, see
\cite[Appendix]{DP}.
If we normalize the K\"ahler structure $(g,\sigma)$
to have holomorphic sectional curvature $-1$ (so the sectional
curvature $K$ satisfies $-1\leq K\leq -1/4$), then the Ma\~n\'e
critical value is $c=1/2$, $\Sigma_k$ is Anosov
for $k>1/2$, for $k=1/2$ one obtains a unipotent ergodic flow
without closed orbits, and for $k<1/2$ all orbits are closed and contractible.
However, there is an essential difference between the case $n=1$
and $n\geq 2$. As we saw above, for $n=1$ the restriction of the symplectic
form to $\Sigma_k$ is exact. This is no longer the case for $n\geq 2$.
Nevertheless it is easy to see that for $k>1/2$, $\Sigma_k$ is stable.
Indeed, since the flow is algebraic, the strong stable and unstable bundles
are real analytic and together they span a hyperplane bundle $\xi$ in $\Sigma_k$
invariant under the magnetic flow. If we define a 1-form $\lambda$
such that $\lambda(X_{H})=1$ and $\xi=\mbox{\rm ker}\,\lambda$, then
$\lambda$ is a stabilizing 1-form. Observe that in this case
$d\lambda$ and $\omega$ span a 2-dimensional space of flow invariant
2-forms. For $n=1$ this space is just one dimensional.

\subsection{The exact case}\label{ss:exact}

Now we focus on the case that $\sigma$ is exact. We begin with the
case $\sigma=0$.

\begin{lemma}
Suppose $\sigma=0$. Then $c=\max\, U$ and any regular level set
$\Sigma_k$ is of contact type.
\label{lemma:contact}
\end{lemma}

\begin{proof}
If $\sigma=0$ the infimum in the definition of $c=c(\hat H)$ is
attained for $\theta=0$, hence $c=\max U$ and $k=c$ is a singular
value of $H$.
If $k>c=\max\,U$, then $\Sigma_k$ encloses the Lagrangian zero section
and is thus of contact type (with the Liouville form $\lambda$ as
contact form). Suppose then $k<c$ is a regular value.
A quick glance at the proof of the previous lemma shows that
$\lambda(X_{H})=|p|^2\geq 0$. Note that the set
$\{p=0\}\cap\Sigma_{k}$ does not contain any set invariant under the
flow of $X_{H}$ since $X_{H}(q,0)=(0,*)\neq 0$. Therefore, for
any invariant Borel probability measure $\mu$ we have
\[\int_{\Sigma_{k}}\lambda(X_{H})\,d\mu>0\]
and $\Sigma_k$ is of contact type by Theorem~\ref{thm:mcduff}.
Alternatively, the condition $X_{H}(q,0)=(0,*)\neq 0$ allows us to
pick a function $f\co T^*M\to\R$ with $df(X_H)>0$ along $\{p=0\}\cap
\Sigma_k$, hence (for $f$ small) $(\lambda+df)(X_H)>0$ and
$\lambda+df$ is a contact form on $\Sigma_k$.
\end{proof}

More generally, we have

\begin{lemma}\label{lem:exact}
A closed 2-form $\sigma$ on $M$ is exact if and only if
$\pi_0^*\sigma$ has a bounded primitive. The Ma\~n\'e critical value
$c_0$ associated to the abelian cover $\pi_0\co M_0\to M$ is given by the
formula in the Introduction
$$
   c_0 = c(H) = \inf_\theta\sup_{q\in M} H(q,\theta_q),
$$
where the infimum is taken over all 1-forms $\theta$ on $M$ with
$d\theta=\sigma$ (and $c_0=\infty$ if $\sigma$ is not exact). If
$k>c_0$ the hypersurface $\Sigma_k$ is of contact type.
\end{lemma}

\begin{proof}
The first two statements follow from the fact that the deck
transformation group of the cover $\pi_0\co M_0\to M$ is abelian, hence
amenable, so a bounded primitive of $\pi_0^*\sigma$ can be averaged to
a primitive of $\sigma$ on $M$, see~\cite{Gro} and \cite[Corollary 5.4]{P}.
The proof of the last statement is analogous to the proof of
Lemma~\ref{virtc} (but simpler).
\end{proof}

\begin{remark} Note that in the last two lemmas we actually obtain
 {\it restricted} contact type.
\end{remark}

{\em Relation to Mather's $\alpha$-function.}
Suppose now that $\sigma$ is exact and fix a primitive $\theta$ in $M$
and consequently a Lagrangian $L$.
As Ma\~n\'e pointed out \cite{Ma,CDI}, there is a close relationship between
the critical values and Mather's minimizing measures \cite{M}.
Mather's $\alpha$ function is a convex superlinear function
$\alpha\co H^{1}(M,\R)\to \R$
given by
\[\alpha([\varpi])=-\min_{\mu} \int_{TM}(L-\varpi)\,d\mu\]
where $\mu$ runs over all Borel probability measures on $TM$ invariant
under the Euler-Lagrange flow of $L$. It turns out that
\[\alpha([\varpi])=c(L-\varpi)\]
and therefore
\[
   c_0 = c(H)=\min_{[\varpi]\in H^{1}(M,\R)}\alpha([\varpi])
\]
(see~\cite{PP,CIPP}).
The value $c_0=c(H)$ is also called the {\it strict critical
  value}. We now summarize some of the main properties of $c_0$:

\begin{enumerate}
\item $c_{0}=c(L_{0})=c(H_0)$, where $L_0,H_0$ are the lifts
of $L,H$ to the abelian cover $M_0$ \cite{PP};
\item if $M\neq \T^2$ and $c<k\leq c_0$, then $\Sigma_k$ is not
of contact type \cite[Theorem B.1]{Co};
\item there is a characterization of $c_0$ in terms of Symplectic
Topology \cite{PPS}: it is the infimum of the values of $k$ such
that the region bounded by $\Sigma_k$ contains a Lagrangian
submanifold Hamiltonian isotopic to ${\rm gr}(-\theta)$.
(in fact the whole of Mather's $\alpha$ function can be characterized
in a similar way just by considering Lagrangian submanifolds
with a fixed Liouville class);
\item if $\pi_1(M)$ is amenable, $c=c_0$ \cite{FM}.
\end{enumerate}

{\em The gap between $c$ and $c_0$. }
The inequality $c\leq c_0$ could be strict. Examples of this phenomenon
were given for the first time in \cite{PP}:

\begin{theorem}[\cite{PP}]\label{thm:PP}
On a closed oriented surface $M$ of genus $\geq 2$ there exists a
metric of negative curvature and an exact 2-form $\sigma$ such that
$c<c_0$. Moreover, there exists an open interval $I\subset (c,c_0)$
such that all level sets $\Sigma_k$ with $k\in I$ are Anosov. In
particular, these levels are virtually contact but not stable.
\end{theorem}

\begin{proof}
All statements except the last one are proven in~\cite{PP}. The levels
$k\in I$ are virtually contact by Lemma \ref{virtc}. On the other
hand, suppose a level $k\in I$ is stable and let
$\lambda$ be a stabilizing 1-form. The transitivity of the flow and the
fact that $\dim M=2$ easily imply the existence of a constant
$a$ such that $d\lambda=a\omega$. If $a\neq 0$ we obtain a contradiction
with the fact that these energy levels cannot be contact by property
(ii) above. If $a=0$, we obtain a contradiction with
Corollary~\ref{cor:liouville}.
\end{proof}

Next we describe a large class of manifolds exhibiting a large gap between
$c$ and $c_0$.
Consider a closed manifold $M$ with a closed 2-form $\Omega$ such that
$0\neq [\Omega]\in H^2(M,\Z)$. Suppose moreover that $\Omega$  has a bounded primitive
in $\widetilde{M}$. Consider a circle bundle $p\co P\to M$ with Euler
class $[\Omega]$ and $\psi$
a connection 1-form with $p^*\Omega=-d\psi$. Fix any Riemannian metric
$h$ on $M$ and consider for $\eps>0$
the metric $g_{\varepsilon}$ on $P$ given by
\[g_{\varepsilon}(u,v)=\varepsilon^{-1}h(dp(u),dp(v))+\psi(u)\psi(v).\]

\begin{theorem}For $H=\frac{1}{2}|p|_\eps^2$ on $T^*P$ and
  $\sigma=-d\psi$ the strict Ma\~n\'e critical value satisfies
  $c_{0}(g_{\varepsilon},\sigma)=1/2$ for all $\varepsilon>0$,
but $c(g_{\varepsilon},\sigma)\to 0$ as $\varepsilon\to 0$.
\label{thm:gap}
\end{theorem}

\begin{proof} Since $[\Omega]\neq 0$, the Gysin sequence of the circle
  bundle $p\co P\to M$ shows that
$p^*\co H^{1}(M,\R)\to H^{1}(P,\R)$ is an isomorphism. Thus the orbits of
the circle action are all null-homologous in $H_1(P;\R)$.
Let $V$ be the vector field dual to $\psi$ and note that $g_{\varepsilon}(V,V)=1$.
Clearly this implies $c_{0}(g_{\varepsilon},\sigma)\leq 1/2$. To show
that equality holds we consider the Lagrangian on $P$ given by
\[L(q,v)=\frac{1}{2}|v|_{\varepsilon}^{2}-\psi_{q}(v),\]
and we note that
the action $A_{L+k}$ of the orbits of $V$ equals $2\pi(k-1/2)$. Since
the orbits of $V$ are null-homologous in $H_1(P;\R)$ and
this action is negative if $k<1/2$ we must have
$c_{0}(g_{\varepsilon},\sigma)=1/2$.

Let $\pi\co \widetilde{M}\to M$ be the universal cover and
$\hat{p}\co \pi^*P\to \widetilde{M}$, the pull-back bundle. Let
$\hat{\pi}\co \pi^*P\to P$ be the obvious map such that
$p\circ\hat{\pi}=\pi\circ \hat{p}$.
Clearly $\hat{\pi}$ is a covering map of $P$.
By hypothesis, we may write $\pi^*\Omega=d\theta$, where
$|\theta|_{\infty}=\sup_{q\in \widetilde{M}}|\theta_{q}|<\infty$.
Note that
\[d\hat{p}^*\theta=-d\hat{\pi}^*\psi=\hat{\pi}^*\sigma\]
and that
\[|\theta_{q}|_{\varepsilon}=\min_{|v|_{\varepsilon}=1}|\theta_{q}(v)|
=\varepsilon\,\min_{|v|=1}|\theta_{q}(v)|=\varepsilon\,|\theta_{q}|.\]
Thus
\[c(g_{\varepsilon},\sigma)\leq
\frac{\varepsilon|\theta|^2_{\infty}}{2}\longrightarrow 0\;\mbox{\rm as}\;
\varepsilon\to 0.\]
\end{proof}

Later on in Section \ref{sec:ex} we shall see another very explicit
example which together with the preceding result suggests that the gap
between $c$ and $c_0$ is rather frequent on manifolds
with non-amenable fundamental groups.

\subsection{Instability for large energies when $[\sigma]\neq 0$}

The purpose of this subsection is to show that when $[\sigma]\neq 0$, there
can be hypersurfaces $\Sigma_k$ which are not stable for any sufficiently
large $k$. More precisely, we will prove:

\begin{theorem} Let $(M,g)$ be a closed Riemannian manifold of even dimension
different from two whose sectional curvature satisfies the pinching
condition $-4\leq K < -1$.
Let $\sigma$ be a closed 2-form with
$[\sigma]\neq 0$. Then for any $k$ sufficiently large, the hypersurface
$\Sigma_k$ is not stable.
\label{insthe}
\end{theorem}

\begin{remark}
Note that these levels are actually virtually contact by Lemma \ref{virtc}
and the fact that $\sigma$ has bounded primitives on the universal cover
of $M$.
\end{remark}

The ideas for the proof of this theorem come from \cite{FK1,Fe,Ha,Ka}.

We start with a preliminary discussion on Anosov Hamiltonian structures.

Recall that the flow $\phi_t$ of a vector field $F$ on a manifold
$\Sigma$ is {\em Anosov} if there is a splitting $T\Sigma=\R
F\oplus E^s\oplus E^u$ and positive constants $\lambda$ and $C$ such that for
all $x\in \Sigma$
\[|d_{x}\phi_{t}(v)|\leq
C e^{-\lambda t}|v|\;\;\mbox{\rm for}\;v\in E^s\;\;\mbox{\rm and}\;t\geq 0,\]
\[|d_{x}\phi_{-t}(v)|\leq
C e^{-\lambda t}|v|\;\;\mbox{\rm for}\;v\in E^u\;\;\mbox{\rm and}\;t\geq 0.\]
If an Anosov vector field $F$ is rescaled by a positive function its
flow remains Anosov~\cite{AS,Pa}. It will be useful
for us to know how the bundles $E^s$ and $E^u$ change when
we rescale $F$ by a smooth positive function $r\co \Sigma\to\R_{+}$.
Let $\tilde\phi$ be the flow of $rF$ and $\tilde{E}^s$ its stable bundle.
Then (cf. \cite{Pa})
\begin{equation}
\tilde{E}^{s}(x)=\{v+z(x,v)F(x):\;\;\;v\in E^{s}(x)\},
\label{eq:stc}
\end{equation}
where $z(x,v)$ is a continuous 1-form (i.e. linear in $v$ and continuous
in $x$). Moreover, if we let $l=l(t,x)$ be (for fixed $x$) the inverse
of the diffeomorphism
\[t\mapsto\int_{0}^{t}r(\phi_{s}(x))^{-1}\,ds\]
then
\begin{equation}
d\tilde\phi_{t}(v+z(x,v)F(x))=
d\phi_{l}(v)+z(\phi_{l}(v),d\phi_{l}(v))F(\phi_{l}(x)).
\label{eq:tch}
\end{equation}
There is a similar expression for $\tilde{E}^{u}$.
It is clear from the discussion above that the weak bundles $\R F\oplus E^s$
 and
$\R F\oplus E^u$ do not change under rescaling of $F$ (the
strong bundles $E^{s,u}$ are indeed affected by rescaling as we have just seen).

Let $(\Sigma,\omega)$ be a Hamiltonian structure. We say that the structure is
{\em Anosov} if the flow of any vector field $F$ spanning $\ker\om$ is Anosov.

We say that an Anosov Hamiltonian structure satisfies the {\it
  $1/2$-pinching condition} or that it is {\it 1-bunched} \cite{H,H2} if
for any vector field $F$ spanning $\ker\om$ with flow $\phi_t$
there are functions $\mu_{f}, \mu_{s}\co \Sigma\times \R_{+}\to \R_{+}$
such that
\begin{itemize}
\item $\lim_{t\to\infty}\sup_{x\in\Sigma}\frac{\mu_{s}(x,t)^{2}}{\mu_{f}(x,t)}=0$;
\item $\mu_{f}(x,t)|v|\leq |d\phi_{t}(v)|\leq \mu_{s}(x,t)|v|$ for all
$x\in\Sigma$, $t>0$ and $v\in E^{s}(x)$, and $\mu_{f}(x,t)|v|\leq |d\phi_{-t}(v)|\leq \mu_{s}(x,t)|v|$ for all
$x\in\Sigma$, $t>0$ and $v\in E^{u}(\phi_t x)$.
\end{itemize}
We remark that the 1/2-pinching condition is invariant under rescaling.
Indeed, consider the flow $\tilde\phi_{t}$ of $rF$. It is clear from
(\ref{eq:stc}) and (\ref{eq:tch}) that there is a positive constant $\kappa$
such that
\[\frac{1}{\kappa}\mu_{f}(x,l(t,x))|\tilde v|\leq
|d\tilde\phi_{t}(\tilde v)|\leq \kappa \mu_{s}(x,l(t,x))|\tilde v|\]
for $t>0$ and $\tilde v\in \tilde{E}^{s}$ (with a similar expression
for $\tilde{E}^{u}$). We know that given $\varepsilon>0$, there exists
$T>0$ such that
for all $x\in\Sigma$ and all $t>T$ we have
\[\frac{\mu_{s}(x,t)^{2}}{\mu_{f}(x,t)}<\varepsilon.\]
On the other hand, there exists $a>0$ such that $l(t,x)\geq at$ for
all $x\in\Sigma$ and $t>0$. Hence for all $t>T/a$ we have
\[\frac{\mu_{s}(x,l(t,x))^{2}}{\mu_{f}(x,l(t,x))}<\varepsilon\]
for all $x\in\Sigma$.
Therefore
$$\lim_{t\to\infty}\sup_{x\in\Sigma}\frac{\mu_{s}(x,l(t,x))^{2}}{\mu_{f}(x,l(t,x))}=0$$
and thus $\tilde\phi_{t}$ is also $1/2$-pinched.

Hence the Anosov
property as well as the $1/2$-pinching condition are invariant under
rescaling and thus intrinsic properties of the Hamiltonian structure.
One of the main consequences of the $1/2$-pinching condition is that
the weak bundles $\R F\oplus E^s$ and $\R F\oplus E^u$ are of class
$C^1$ \cite[Theorem 5]{H2} (see also \cite{HPS}).

Suppose now $(\Sigma,\om)$ is a {\it stable} Anosov Hamiltonian structure
satisfying the $1/2$-pinching condition and let $\lambda$ be the
stabilizing 1-form and $R$ the Reeb vector field. Invariance under the
flow implies that $\om$ and $\lambda$ both vanish on $E^s$ and $E^u$.
Since the flow $\phi_t$ of $R$ is Anosov and
$E^s\oplus E^u=\ker\lambda$ which is $C^\infty$, it is clear that
$E^s$ and $E^u$ must be $C^1$ since $E^{s,u}=(\R F\oplus E^{s,u})\cap\ker\lambda$. Under these conditions we can introduce
the {\it Kanai connection} \cite{Ka} which is defined as follows.

Let $I$ be the $(1,1)$-tensor on $\Sigma$ given by $I(v)=-v$ for $v\in E^s$,
$I(v)=v$ for $v\in E^u$ and $I(R)=0$. Consider the symmetric non-degenerate
 bilinear form given by
\[h(X,Y):=\om(X,IY)+\lambda\otimes\lambda(X,Y).\]
The pseudo-Riemannian metric $h$ is of class $C^1$ and thus there exists
a unique $C^0$ affine connection $\nabla$ such that:
\begin{enumerate}
\item $h$ is parallel with respect to $\nabla$;
\item $\nabla$ has torsion $\om\otimes R$.
\end{enumerate}
This connection has the following desirable properties \cite{Fe,Ka}:
\begin{itemize}
\item The connection is invariant under $\phi_t$;
\item The Anosov splitting is invariant under $\nabla$: if $X$ is any section
of $E^{s,u}$, $\nabla_{v}X\in E^{s,u}$ for any $v$;
\item the restriction of $\nabla$ to the weak stable and unstable
  manifolds (i.e.~leaves of the weak stable and unstable foliations)
  is flat;
\item parallel transport along curves on weak stable/unstable manifolds
coincide with the holonomy transport determined by the stable/unstable
foliations.
\end{itemize}

The other good consequence of the $1/2$-pinching condition, besides
$C^1$ smoothness of the bundles, is the following lemma (cf. \cite[Lemma 3.2]{Ka}).

\begin{lemma} $\nabla(d\lambda)=0$.
\label{parallel}
\end{lemma}

\begin{proof}
Suppose $\tau$ is any invariant $(0,3)$-tensor annihilated by $R$.
We claim that $\tau$ must vanish. To see this, consider for example
a triple of vectors $(v_1,v_2,v_3)$ where $v_1,v_2\in E^{s}$
but $v_{3}\in E^{u}$. Then there is a constant $C>0$ such that
\begin{align*}
|\tau_{x}(v_{1},v_{2},v_{3})|&=|\tau_{\phi_{t}x}(d\phi_{t}(v_{1}),d\phi_{t}(v_{2}),d\phi_{t}(v_{3}))|\\
&\leq C {\mu_{s}(x,t)}^{2}\mu_{f}(x,t)^{-1}|v_{1}||v_{2}||v_{3}|.
\end{align*}
By the $1/2$-pinching condition the last expression tends to zero
as $t\to\infty$ and therefore $\tau_{x}(v_{1},v_{2},v_{3})=0$.
The same will happen for other possible triples $(v_1,v_2,v_3)$
when we let $t\to\pm\infty$.

Since $d\lambda$ and $\nabla$ are $\phi_t$-invariant, so is
 $\nabla(d\lambda)$. Since $i_{R}d\lambda=0$, $\nabla(d\lambda)$ is
also annihilated by $R$ (to see that $\nabla_{R}(d\lambda)=0$ use
that $d\lambda$ is $\phi_t$-invariant and that $\nabla_{R}=L_{R}$).
Hence by the previous argument applied to $\tau=\nabla(d\lambda)$
we conclude that $\nabla(d\lambda)=0$ as desired.
\end{proof}

Since $\om$ is non-degenerate, there exists a smooth bundle map
$L\co E^s\oplus E^u\to E^s\oplus E^u$ such that for sections $X,Y$
of $E^s\oplus E^u$
\[d\lambda(X,Y)=\om(LX,Y)=\om(X,LY).\]
The map $L$ is invariant under $\phi_t$ and preserves the
 decomposition $E^s\oplus E^u$, i.e.
 $L=L^s+L^u$, where $L^s\co E^s\to E^s$ and $L^u\co E^u\to E^u$. In
 particular, $L$ commutes with $I$.

Suppose now $\dim\Sigma=2n-1$, where $n$ is an even integer.
Since $\dim E^s=n-1$ is odd, $L_{x}^s$ admits a real eigenvalue
$\rho$ (note that by transitivity of $\phi_t$, the characteristic polynomial
of $L_{x}^s$ is independent of $x\in\Sigma$).
Let
\[H(x):=\{v\in E^s(x):\;L_{x}^{s} v=\rho v\}.\]
Since $\nabla(d\lambda)=0$ (Lemma \ref{parallel}), $H(x)$ is invariant
under the parallel transport of $\nabla$ and thus $x\mapsto H(x)$
is a $C^1$ subbundle of $E^s$.

Let $W^s(x)$ be the strong stable manifold through $x$. We note that
the restriction of $H$ to $W^s(x)$ is integrable. Indeed, let $X$ and $Y$
be parallel sections of $E^s$ over $W^s(x)$ (such sections must be $C^1$)
 and observe that since $\nabla$ has zero torsion on $E^s$ we have
$[X,Y]=\nabla_{X}Y-\nabla_{Y}X=0$.

The maximal integral submanifolds of $H$ on $W^s(x)$ define a foliation
of class $C^1$ on $W^s(x)$
and thus a foliation $\mathcal F$ of class $C^1$ on $\Sigma$.

\begin{proof}[Proof of Theorem~\ref{insthe}]
Suppose the sectional curvature $K$ of a Riemannian metric
satisfies $-A^2\leq K\leq -a^2$ for some positive constants
$a$ and $A$. Then, comparison theorems show that
\cite[Theorem 3.2.17]{K} (see also \cite[Proposition 3.2]{Kn})
there is a constant
$C>0$ such that
\begin{equation}
\frac{1}{C}|v|e^{-At}\leq|d_{x}\phi_{t}(v)|\leq C|v|e^{-at}\;\;\mbox{\rm for}\;v\in E^s\;\;\mbox{\rm and}\;t\geq 0,
\label{eq:a1}
\end{equation}
\begin{equation}
\frac{1}{C}|v|e^{-At}\leq|d_{x}\phi_{-t}(v)|\leq C|v|e^{-at}\;\;\mbox{\rm for}\;v\in E^u\;\;\mbox{\rm and}\;t\geq 0,
\label{eq:a2}
\end{equation}
where $\phi_t$ is the geodesic flow of the Riemannian metric.
If we let $\mu_{s}=Ce^{-at}$ and $\mu_{f}= \frac{1}{C}e^{-At}$ we see
that $\phi_t$ is $1/2$-pinched as long as $A<2a$.
Therefore the geodesic flow of a metric whose sectional curvature satisfies
$-4\leq K< -1$ is $1/2$-pinched.
For sufficiently large $k$, $(\Sigma_{k},\om_{\sigma})$
is an Anosov Hamiltonian structure since it can be seen as a perturbation
of the geodesic flow. In fact it is also $1/2$-pinched. This can be seen
as follows. An equivalent claim is that $(\Sigma_{1/2},\om_{s\sigma})$
is $1/2$-pinched for small $s$. An inspection of the proof of (\ref{eq:a1})
and (\ref{eq:a2}) in \cite[Theorem 3.2.17]{K} shows that if we do the same
analysis for the magnetic Jacobi (or Riccati) equation we obtain numbers
$A(s)$ and $a(s)$ for which (\ref{eq:a1}) and (\ref{eq:a2}) hold.
These numbers will be as close as we wish to $A(0)=2$ and
$a(0)=\sqrt{-\max\, K}>1$ if $s$ is small enough and the
$1/2$-pinching condition follows.

We now make some remarks concerning the space of leaves of the weak
foliations.
The weak unstable foliation ${\cal W}^{u}$ of $(\Sigma_k,\om_{\sigma})$ is
{\it transverse} to the fibres of the fibration by $(n-1)$-spheres given by
\[\tau\co \Sigma_k\to M.\]
(This simply follows from the fact that this property is true for the
geodesic flow.)
Let $\widetilde{M}$ denote the universal cover of $M$ and
let $\widetilde{\Sigma}_k$ denote the preimage of $\Sigma_k$ in
$T^*\widetilde{M}$.
We also have a fibration by $(n-1)$-spheres
\[\widetilde{\tau}|_{\widetilde{\Sigma}_k}\co \widetilde{\Sigma}_k\to \widetilde{M}.\]
Let $\widetilde{{\cal W}}^{u}$ be the lifted foliation and note that
the foliation $\widetilde{{\cal W}}^{u}$ is also transverse to the fibration
$\widetilde{\tau}|_{\widetilde{\Sigma}_k}\co \widetilde{\Sigma}_k\to \widetilde{M}$.
Since the fibres are compact a standard result in foliations
\cite[p.91]{CN} implies that for every $p\in \widetilde{\Sigma}_k$ the map
\[\widetilde{\tau}|_{\widetilde{{\cal W}}^{u}(p)}\co \widetilde{{\cal
    W}}^{u}(p)\to \widetilde{M},\]
is a covering map. Since $\widetilde{M}$ is simply connected,
$\widetilde{\tau}|_{\widetilde{{\cal W}}^{u}(p)}$ is in fact a diffeomorphism
and $\widetilde{{\cal W}}^{u}(p)$ is simply connected.
Consequently, $\widetilde{{\cal W}}^{u}(p)$ intersects
each fibre of the fibration
$\widetilde{\tau}|_{\widetilde{\Sigma}_k}\co \widetilde{\Sigma}_k\to \widetilde{M}$
at just one point and therefore the space of leaves
 $B^u:=\widetilde{\Sigma}_k/\widetilde{{\cal W}}^{u}$
 of the weak unstable
foliation can be identified topologically with the $(n-1)$-sphere.
Similarly the space of leaves of the weak stable
foliation is also an $(n-1)$-sphere.
Note that $\pi_{1}(M)$ acts on
$B^{u}$.
Since the characteristic foliation of $(\Sigma_k,\om_{\sigma})$ is topologically
conjugate to that of the geodesic flow, the action of $\pi_1(M)$ on the space
of leaves is topologically like in the geodesic flow case: Every element
in $\pi_1(M)$ acts on $B^u$ as a {\em North-South dynamics}, i.e.~a
homeomorphism of $S^{n-1}$ with exactly two fixed points $P_\pm$ such
that every other point converges to $P_\pm$ under forward
resp.~backward iteration.

Now suppose that $\Sigma_k$ is stable. Our previous discussion produces
a foliation $\mathcal F$ of class $C^1$ on $\Sigma_k$.
This foliation can be lifted to $\widetilde{\Sigma}_k$ and then
projected to $B^u$ ($\mathcal F$ is invariant under holonomy maps)
to produce a $C^1$ foliation of positive dimension
on $B^u$ which is $\pi_1(M)$-invariant. By the result in \cite{Fo},
such a foliation must be trivial, i.e.~consist of just one leaf. This
implies that $H=E^s$, and hence
$$
   d\lambda=\rho\om_{\sigma}
$$
for some constant $\rho$. By Corollary~\ref{cor:liouville},
the constant $\rho$ cannot be zero. It follows then that
$\om_{\sigma}$ must be exact on $\Sigma_k$.
However, the Gysin sequence of the sphere bundle shows
that this is impossible since $[\sigma]\neq 0$ and $n\neq 2$.

Thus $\Sigma_k$ cannot be stable for high energies, which proves
Theorem~\ref{insthe}.
\end{proof}

\begin{remark} The proof above can be considerably improved
following the ideas in \cite{Ha} to show the following result: suppose
$[\sigma]\neq 0$ and $n\geq 3$ (not necessarily even).
If $(\Sigma_k,\om_{\sigma})$ is a $1/2$-pinched Anosov Hamiltonian structure,
then it cannot be stable \cite{MP}.
If we drop the $1/2$-pinching condition, the result is no longer
true since compact quotients of complex hyperbolic space with
the K\"ahler form are stable for high energies (cf. Subsection \ref{ss:hyp})
It is tempting to conjecture that these are the only stable Anosov
Hamiltonian structures with $[\sigma]\neq 0$ and $n\geq 3$.
\end{remark}

\section{Homogeneous examples}\label{sec:ex}

In this section we study magnetic flows for left-invariant metrics on
some compact homogeneous spaces $\Gamma\setminus G$ and verify the
paradigms in the introduction. Here $G$ is always a Lie group and
$\Gamma\subset G$ a cocompact lattice, i.e.~a discrete subgroup such
that the left quotient $\Gamma\setminus G$ is compact. The magnetic field
$\sigma$ will be always a left-invariant 2-form.

Using left translations we identify $T^*G$ with
$G\times \g^*$ so left-invariant smooth functions $f\co T^*G\to\R$ are just
elements in $C^{\infty}(\g^*)$.
As before we set $\omega_{\sigma}=dp\wedge dq+\tau^*\sigma$, where
$dp\wedge dq$ is the canonical symplectic form in $T^*G$.
Let $\{\;,\;\}_{\sigma}$ be the Poisson bracket of $\omega_{\sigma}$
defined as $\{H,F\}_{\sigma}=\omega_{\sigma}(X_{H},X_{F})$.
Since all the objects involved are left-invariant, in the next lemma
we just focus on the identity $e\in G$.

\begin{lemma}Let $\mu\in \g^*$, $(v,\xi),\;(w,\eta)\in
  T_{(e,\mu)}(G\times\g^*)=\g\times\g^*$. We have:
\begin{enumerate}
\item $\omega_{\sigma}(e,\mu)((v,\xi),(w,\eta))=
\xi(w)-\eta(v)-\mu([v,w])+\sigma_{e}(v,w)$;
\item Given $f,g\in C^{\infty}(\g^*)$,
\[\{f,g\}_{\sigma}(\mu)=\mu([d_{\mu}f,d_{\mu}g])-\sigma_{e}(d_{\mu}f,d_{\mu}g)\]where we canonically identify $(\g^{*})^*$ with $\g$;
\item Given $f\in C^{\infty}(\g^*)$, let $X_{f}$ be the Hamiltonian vector
field of $f$ with respect to $\omega_{\sigma}$. Then
\[X_{f}(e,\mu)=(d_{\mu}f,E_{f}(\mu)),\]
where $E_{f}(\mu)\in \g^*$ is given by
\[E_{f}(\mu)(w)=\mu([d_{\mu}f,w])-\sigma_{e}(d_{\mu}f,w).\]
We will call $E_{f}$ the {\em Euler vector field} of $f$.
\end{enumerate}
\label{evf}
\end{lemma}

\begin{proof}
The proof of the lemma for $\sigma=0$ can be found in \cite[Section 4.4]{AM}.
When $\sigma\neq 0$ the lemma follows right away from the definition
of $\omega_\sigma$; we leave the details to the reader.
\end{proof}

Suppose $f\in C^{\infty}(\g^*)$ has a compact level set
$\bS_{k}:=f^{-1}(k)\subset \g^*$. Clearly $E_{f}$ is tangent to
$\bS_k$; let $\psi_t$ be its flow. Let $\phi_t$ be the flow
of $X_{f}$ in $\Sigma_k=(\Gamma\setminus G)\times \bS_k$. Let
$\pi_2\co (\Gamma\setminus G)\times \bS_k\to\bS_k$ be the
second factor projection. Clearly $\pi_2\circ\phi_{t}=\psi_{t}\circ\pi_2$.
Thus, if $E_f$ is geodesible on $\bS_k$, then $\Sigma_k$ is a stable
hypersurface. This observation is nothing but a rephrasing of Lemma
\ref{lem:proj-stable} in this context.

\subsection{Tori}
\label{sub:tori}
Consider the torus $\T^n$ ($n\geq 2$) and let $\sigma$ be a
non-zero constant 2-form. Let $\omega_{\sigma}=dp\wedge dq+\tau^*\sigma$ be the
twisted symplectic form on $T^*\T^n$. We shall consider
on $\T^n$ the usual flat metric and we let $\Sigma_{k}:=H^{-1}(k)$,
where $H(q,p)=\frac{1}{2}|p|^2$.
Below we will make use of the following elementary lemma.

\begin{lemma}Consider $\R^k$ with its usual inner product and let
$A\co \R^k\to\R^k$ be an antisymmetric linear map.
Consider the 1-form $\alpha$ in $\R^k$
given by $\alpha_{p}(\xi)=\frac{1}{2}\langle p,A\xi\rangle$.
Then
\[d\alpha_{p}(\xi_{1},\xi_{2})=\langle \xi_1,A\xi_2\rangle.\]
\label{elemd}
\end{lemma}

\begin{proposition}If $\sigma$ is non-zero, then the hypersurface
$\Sigma_k$ is stable, tame and displaceable for any $k>0$.
\label{tcase}
\end{proposition}

\begin{proof} Write $T^*\T^n=\T^n\times\R^n$ and let $\pi\co \T^n\times \R^n\to
\R^n$ be the second factor projection. The equations of motion
of any Hamiltonian $H(q,p)$ with respect to $\omega_{\sigma}$ are
given by
\begin{align*}
\dot{q}&=\nabla_pH,\\
\dot{p}&=J\nabla_pH-\nabla_qH,
\end{align*}
where $J\co \R^n\to \R^n$ is the antisymmetric linear map determined
by $\sigma(\cdot,\cdot)=\langle \cdot,J\cdot\rangle$.
In particular,
for the Hamiltonian $H(q,p)=\frac{1}{2}|p|^2$ we obtain
\begin{align*}
\dot{q}&=p,\\
\dot{p}&=Jp
\end{align*}
Split $\R^n$ orthogonally as $\R^n=\mbox{\rm Ker}J\oplus\mbox{\rm Im}J$ and
let $P_{1}\co \R^n\to \mbox{\rm Ker}J$ and $P_{2}\co \R^n\to \mbox{\rm Im}J$ be the
corresponding orthogonal projections. The restriction of $J$ to
$\mbox{\rm Im}J$ is invertible and we let $A:=\left (J|_{\mbox{\rm Im}J}\right)^{-1}$.
With this choice of $A$ we obtain a 1-form $\alpha$ on $\mbox{\rm Im}J$
as in Lemma \ref{elemd}.

Let $f\co \T^n\times\R^n\to \T^n\times\R^n$ be the map $f(q,p)=(q,P_{1}(p))$
and let $\nu$ be the Liouville 1-form in $T^*\T^n$. We claim that
\[\lambda:=f^*\nu+(P_{2}\circ\pi)^{*}\alpha\]
is a stabilizing 1-form on $\Sigma_k$ for any $k>0$.
Note that $df_{(q,p)}(X_{H})=(p,0)$ and
$d\pi_{(q,p)}(X_{H})=Jp.$ Thus
\[\lambda_{(q,p)}(X_{H})=\nu_{(q,P_{1}(p))}(p,0)+\alpha_{P_{2}(p)}(Jp)=|P_{1}(p)|^2+\frac{1}{2}|P_{2}(p)|^2\]
which is always positive on $\Sigma_k$.
On the other hand using Lemma \ref{elemd} we obtain
\[i_{X_{H}}d\lambda_{(q,p)}(x,y)=d\nu_{(q,P_{1}(p))}((p,0),(x,P_{1}(y)))
+\langle Jp,AP_{2}(y)\rangle,\]
where $(x,y)$ denotes a vector in $T_{(q,p)}(\T^n\times\R^n)$.
Using the fact that $d\nu$ is the standard symplectic form on
$T^*\T^n$ we obtain
\[d\nu_{(q,P_{1}(p))}((p,0),(x,P_{1}(y)))=-dH_{(q,P_{1}(p))}(x,P_{1}(y))
=-\langle p,P_{1}(y)\rangle.\]
Using the definition of $A$ the other term is
\[
   \langle Jp,AP_{2}(y)\rangle = -\langle p,JAP_{2}(y)\rangle =
   -\langle p,P_{2}(y)\rangle
\]
and hence
\[i_{X_{H}}d\lambda_{(q,p)}(x,y)=-\langle p,y\rangle=0\]
for $(x,y)\in T_{(q,p)}\Sigma_k$. This proves
$(i_{X_H}d\lambda)|_{\Sigma_k}=0$, so $\lambda$ is a stabilizing
1-form for $\Sigma_k$.

Displaceability of $\Sigma_{k}$ is easy to see (cf.~\cite[Theorem
B]{BP} and the proof of Theorem 3.1 in~\cite{GK}): Pick any $a\in\R^n$
with $Ja\neq 0$ (this exists since $\sigma\neq 0$) and consider the
Hamiltonian $h(q,p)=\la a,p\ra$. The corresponding Hamiltonian flow
defined by
$$
   \dot q=a,\qquad \dot p=Ja
$$
contains a translation in direction $Ja$ and hence displaces every
compact subset of $T^{*}\T^{n}$.

In order to show that $\Sigma_k$ is tame we need to locate the contractible
closed orbits. We can readily find the flow $\phi_t$ of $X_{H}$.
Since $J$ is antisymmetric, the solution to $\dot{p}=Jp$ is
a 1-parameter subgroup $t\mapsto e^{tJ}\in SO(n)$.
We have
\begin{equation}
\phi_{t}(q,p)=\left(\int_{0}^{t}e^{sJ}p\,ds+q,\;e^{tJ}p\right).\
\label{lflow}
\end{equation}
Write $p=p_1+p_2$, where $p_1\in \mbox{\rm Ker}J$
and $p_2\in \mbox{\rm Im}J$. Then $e^{sJ}p=p_1+e^{sJ}p_{2}$ and thus
\begin{equation}
\int_{0}^{t}e^{sJ}p\,ds=p_{1}t+\int_{0}^{t}e^{sJ}p_{2}\,ds.
\label{lflow2}
\end{equation}
Let $\hat{J}:=J|_{\mbox{\rm Im}J}$. Then
$e^{s\hat{J}}=e^{sJ}|_{\mbox{\rm Im}J}$.
But $\hat{J}$ is invertible, so we may write
\[\int_{0}^{t}e^{sJ}p_{2}\,ds=
\int_{0}^{t}(\hat{J})^{-1}\frac{d}{ds}\left(e^{s\hat{J}}p_{2}\right)\,ds=
(\hat{J})^{-1}(e^{t\hat{J}}-I)p_2.\]
Combining this with (\ref{lflow}) and (\ref{lflow2}) we see that
$(q,p)\in \Sigma_{k}$ gives rise to a closed contractible orbit of period $T$ iff
$p_{1}=0$, $e^{TJ}p_{2}=p_{2}$ and $|p_{2}|^2=2k$.

Now, a primitive for $dp\wedge dq+\tau^{*}\sigma$ in $T^*\R^n$ is given by
$pdq+\tau^*\beta$ where $\beta_{q}(a)=\frac{1}{2}\langle q,Ja\rangle$
(we use again Lemma \ref{elemd}). Thus, if $v$ is a closed contractible
orbit in $\Sigma_k$, we have
\[\Omega(v)=\int_{v}(pdq+\tau^*\beta)=2kT+\int_{\tau(v)}\beta.\]
We compute
\[\int_{\tau(v)}\beta=\frac{1}{2}\int_{0}^{T}\langle
(\hat{J})^{-1}(e^{t\hat{J}}-I)p_2,Je^{t\hat{J}}p_{2}\rangle\,dt=-\frac{1}{2}|p_{2}|^2\,T=-kT,\]
and therefore
\[\Omega(v)=kT\]
which clearly shows that $\Sigma_k$ is tame.
\end{proof}

\begin{remark} It is instructive to see what the form $\lambda$ looks
like in the following special cases:
\begin{itemize}
\item when $n=2$ and $\sigma=dq_1\wedge dq_2$, then
\[\lambda=-\frac{1}{2}(p_{1}dp_{2}-p_{2}dp_{1}),\]
\item when $n=3$ and $\sigma=dq_1\wedge dq_2$, then
\[\lambda=p_{3}dq_{3}-\frac{1}{2}(p_{1}dp_{2}-p_{2}dp_{1}).\]
\end{itemize}
\end{remark}

Note that the Ma\~n\'e critical value is infinite because a non-zero
$\sigma$ has no bounded primitives in $\R^n$.

\subsection{The Heisenberg group}
\label{sub:heisenberg}
Let $G$ be the 3-dimensional Heisenberg group of matrices
\[\left(\begin{array}{ccc}

1&x&z\\
0&1&y\\
0&0&1\\

\end{array}\right),\]
where $x,y,z\in \R$. If we identify $G$ with $\R^3$, then the product is
\[(x,y,z)\star(x',y',z')=(x+x',y+y',z+z'+xy').\]
For $\Gamma$ we take the lattice of those matrices with
$x,y,z\in\Z$. However, the following discussion will hold for {\em
  any} cocompact lattice $\Gamma$. In fact, all lattices are isomorphic
to the semidirect product $\mathbb Z^2\ltimes_{A}\mathbb{Z}$ where
\[A=\left(\begin{array}{cc}

1&k\\
0&1\\

\end{array}\right),\]
for some positive integer $k$ \cite{S}.

The 1-forms $\alpha:=dx$, $\beta:=dy$ and $\gamma:=dz-xdy$ are left-invariant
and provide a trivialization of $T^{*}G$ as $G\times\g^{*}$.
We let $(x,y,z,\pa,\pb,\pc)$ be the coordinates induced by this trivialization.
Note that $(\pa,\pb,\pc)$ descend to coordinates on
$T^{*}(\Gamma\setminus G)=(\Gamma\setminus G)\x\g^*$.
The coordinates $(\pa,\pb,\pc)$ are related to the coordinates
$(p_{x},p_{y},p_{z})$ by $\pa=p_{x}$, $\pb=p_{y}+xp_{z}$ and $\pc=p_{z}$.

We consider the following left-invariant Hamiltonian $H$ (dual to a
suitable left-invariant metric on $G$):
\[2H=\pa^2+\pb^2+\pc^2=p_{x}^2+(p_{y}+xp_{z})^2+p_{z}^2.\]
The magnetic field is given by the left-invariant 2-form
\[\sigma:=-dx\wedge dy.\]
Note that $\sigma$ descends to an {\it exact} 2-form on $\Gamma\setminus G$
since $\sigma=d\gamma$. Also observe that $2H(\gamma)\equiv 1$, which
implies that Ma\~n\'e's critical value $c_0$ is $\leq 1/2$. Later on we will
see that $c=c_0=1/2$.

The twisted symplectic form $\omega$ is:
\begin{align*}
   \omega &= dp_{x}\wedge dx+dp_{y}\wedge dy+dp_{z}\wedge dz-dx\wedge
   dy \cr
   &= dp_{\alpha}\wedge dx+dp_{\beta}\wedge dy+dp_{\gamma}\wedge dz -
   x\,dp_\gamma\wedge dy - (1+p_\gamma)dx\wedge dy.
\end{align*}
The Hamiltonian vector field of a function $H$ with respect to
$\omega$ is given by (where $H_x$ etc.~denote partial derivatives)
\begin{equation} \label{eq:XH-gen}
X_H = \left\{ \begin{array}{lclclcl}
\dot{x}   &=& H_\pa,          & \hspace{10mm} & \dot{p}_{\alpha}   &=& -H_x-(1+\pc)H_\pb,\\
\dot{y} &=& H_\pb, &            & \dot{p}_{\beta} &=& -H_y + xH_z + (1+\pc)H_\pa ,\\
\dot{z} &=& x\,H_\pb+H_\pc, &           & \dot{p}_{\gamma} &=&  -H_z.
\end{array}
\right.
\end{equation}
In particular, for the left-invariant Hamiltonian above we obtain
\begin{equation} \label{eq:XH}
X_H = \left\{ \begin{array}{lclclcl}
\dot{x}   &=& \pa,          & \hspace{10mm} & \dot{p}_{\alpha}   &=& -\pb \pc-\pb,\\
\dot{y} &=& \pb, &            & \dot{p}_{\beta} &=& \pa \pc+\pa ,\\
\dot{z} &=& x\pb+\pc, &           & \dot{p}_{\gamma} &=&  0.
\end{array}
\right.
\end{equation}
The non-zero Poisson brackets of the coordinate functions are
($X_f(\bullet) = \{f, \bullet\}$)
\[\{\pa,x\}=\{\pb,y\}=\{\pc,z\}=1,\;\;\{\pb,z\}=x,\;\;\{\pa,\pb\}=1+\pc.\]
Let the Liouville 1-form be $\nu=p_{x}dx+p_{y}dy+p_{z}dz$ and let
\[\psi:=\nu+\tau^{*}\gamma.\]
Clearly $d\psi=\omega$. We compute
\begin{align}
\psi(X_{H})&=\pa^2+\pb(\pb-x\pc)+\pc(x\pb+\pc)+\pc\nonumber\\
&=\pa^2+\pb^2+\pc(\pc+1)=2H+\pc. \label{co}
\end{align}
Let $\Sigma_{k}:=H^{-1}(k)$.

\begin{lemma} The hypersurface $\Sigma_{k}$ is of contact type
if and only if $k>1/2$.
\label{conhei}
\end{lemma}

\begin{proof}
Equation (\ref{co}) tells us right away that if
$k>1/2$, then $\Sigma_{k}$ is of contact type with contact form $\psi$
since in this case
$$
   \psi(X_H) = 2k+\pc\geq 2k-\sqrt{2k}>0
$$
for every point in $\Sigma_{k}$.

An inspection of (\ref{eq:XH}) reveals that
$\pa=\pb=y=0$, $\pc=-\sqrt{2k}$, $x=x_0$, $z(t)=z_0-\sqrt{2k}t$ give
orbits of $X_{H}$ with energy $k$ which project to closed orbits
in $T^*(\Gamma\setminus G)$. These closed orbits are in fact
null-homologous since $\alpha$ and $\beta$ integrate to zero
along them and these two 1-forms span $H^{1}(\Gamma\setminus G,\R)$.
(This is because by a theorem of K. Nomizu \cite{N} the de Rham cohomology
ring of $\Gamma\setminus G$ is isomorphic to the Lie algebra cohomology
of $\g$ and the isomorphism is induced by the natural inclusion of
left-invariant forms.)
Let $\delta$ be one of these closed orbits. Using (\ref{co}) we see that
for $k\leq 1/2$
\begin{equation}
\int_{\delta}\psi= (2k-\sqrt{2k})\tau\leq 0,\label{neg}
\end{equation}
where $\tau$ is the period of $\delta$.
Now let $\mu$ be the Liouville measure on $\Sigma_k$. Again, using (\ref{co})
we obtain
\begin{equation}
\int_{\Sigma_{k}}\psi(X_{H})\,d\mu=\int_{\Sigma_{k}}(2k+\pc)\,d\mu=2k>0.\label{pos}
\end{equation}
By Lemma~\ref{lem:liouville} the Liouville measure is exact as a
current, so (\ref{neg}), (\ref{pos}) and Corollary~\ref{cor:mcduff}
ensure that $\Sigma_{k}$ cannot be of contact type.
Alternatively, instead of the Liouville measure one may also take the distinguished null-homologous closed orbits with $\pa=\pb=y=0$, $\pc=\sqrt{2k}$, $x=x_0$ and $z(t)=z_0+\sqrt{2k}t$.
The integral of $\psi$ along them has value $(2k+\sqrt{2k})\tau>0$.
\end{proof}

\begin{lemma}The hypersurface $\Sigma_{k}$ is stable for $k<1/2$.
\end{lemma}

\begin{proof} Consider the following form on $\Sigma_k$
\[\varphi:=\frac{\pa d\pb-\pb d\pa}{\pa^2+\pb^2}.\]
This form is smooth away from $\pc=\pm\sqrt{2k}$ and $d\varphi=0$.
We look for stabilizing 1-forms $\lambda$ of the form
\[\lambda=f(\pc)\psi+g(\pc)\varphi,\]
where $f$ and $g$ are smooth functions in $[-\sqrt{2k},\sqrt{2k}]$
and $g$ vanishes in small neighbourhoods of the end points. Using
$\psi(X_H)=2k+\pc$ and $\varphi(X_H)=1+\pc$, the condition
$\lambda(X_{H})>0$ is equivalent to
\begin{equation}\label{con1}
(2k+\pc)f(\pc)+(1+\pc)g(\pc)>0,
\end{equation}
and the condition $i_{X_{H}}d\lambda=0$ is equivalent to
\begin{equation}\label{con2}
(2k+\pc)f'(\pc)+(1+\pc)g'(\pc)=0.
\end{equation}
To see that we can always choose $f$ and $g$ satisfying (\ref{con1})
and (\ref{con2}) provided that $2k<1$, take a non-negative
smooth function
$r(\pc)$ which vanishes near the end points of $[-\sqrt{2k},\sqrt{2k}]$,
$r(-2k)>0$ (note that $-\sqrt{2k}<-2k$) and
\[\int_{-\sqrt{2k}}^{\sqrt{2k}}\frac{2k+\pc}{1+\pc}r(\pc)\,d\pc=0.\]
Note that this is always possible. Indeed as $r$ we could take
$(1+\pc)q(2k+\pc)$ where $q(t)$ is a smooth non-negative even bump function
with small support containing the origin.
Now let
\[f(\pc):=\int_{-2k}^{\pc}r(t)\,dt\]
and
\[g(\pc):=-\int_{-\sqrt{2k}}^{\pc}\frac{2k+t}{1+t}r(t)\,dt.\]
Clearly (\ref{con2}) holds.
With these choices, $f$ vanishes only at $-2k$, is negative at the left
of $-2k$ and positive at the right of $-2k$. Also $g\geq 0$, it vanishes
near the end points and $g(-2k)>0$, thus (\ref{con1}) also holds.
\end{proof}

\begin{lemma} The hypersurface $\Sigma_{k}$ is displaceable for $k<1/2$.
\label{disp}
\end{lemma}

\begin{proof} Consider the Hamiltonian in $T^{*}(\Gamma\setminus G)$
given by $f=\pa$. Clearly $\pc$ does not change along
the Hamiltonian flow of $f$ and $\pb$ changes according to
\[\dot{\pb}=\{\pa,\pb\}=1+\pc.\]
Thus if $k<1/2$ and we flow a point in $\Sigma_{k}$ along the
Hamiltonian flow of $f$, we have $\dot{\pb}\geq 1-\sqrt{2k}>0$.
Thus the flow of $f$ will displace $\Sigma_k$ as desired.
\end{proof}

\begin{lemma} The hypersurface $\Sigma_{k}$ has closed contractible
orbits if and only if $k<1/2$. If $k<1/2$ and $v\in X(\Sigma_{k})$
then $\Omega(v)$ is a positive integer multiple of $2\pi(1-\sqrt{1-2k})>0$
and $\Sigma_{k}$ is tame.

\label{cco}
\end{lemma}

\begin{proof} The equations in (\ref{eq:XH}) show that $\pc$ is a first integral
(in fact it is the Casimir of $\g^*$) and that if $\pc+1\neq 0$, then
$\pa$ and $\pb$ are trigonometric functions with period $T=2\pi/(1+\pc)$
and $\pa^2+\pb^2=2k-\pc^2$.
Since $\pa$ and $\pb$ have zero integral along the period,
$x$ and $y$ are also $T$-periodic functions.
Hence we get a closed contractible orbit iff $z$ is periodic of
period $T$ (or an integer multiple). To analyze this condition,
abbreviate $\mu:=1+\pc\neq 0$ and
$a:=\sqrt{2k-p_\gamma^2}\geq 0$. After time shift, the solutions of
Hamilton's equations~\eqref{eq:XH} are given by
$$
   \pa(t)=a\cos\mu t,\qquad \pb(t)=a\sin\mu t,\qquad \pc={\rm const},
   \qquad
   x(t) = x_0+\frac{a\sin\mu t}{\mu}.
$$
Again looking at (\ref{eq:XH}), we see that periodicity of $z$ is
equivalent to
$$
  \int_{0}^{2\pi/(1+\pc)}(x(t)\pb(t)+\pc)\,dt = 0.
$$
Evaluating the integral, we obtain the condition
\begin{equation}\label{period}
   2\pc+\frac{2k-\pc^2}{1+\pc}=0.
\end{equation}
If $k<1/2$, then $1+\pc>0$ and there is a unique solution
to (\ref{period}) given by $\pc=\sqrt{1-2k}-1\in(-\sqrt{2k},0)$, so we
have contractible closed orbits.

If $k\geq 1/2$ it is easy to see that there are no such solutions.
Indeed, this is clear if $1+\pc\neq 0$ because equation~\eqref{period}
has no real solutions for $k>1/2$ and only the solution $\pc=-1$ for
$k=1/2$. It remains to analyse the case $\pc=-1$. It gives rise to a
circle $\pa^2+\pb^2=2k-1$ of critical points for the Euler vector
field (which degenerates to a point when $k=1/2$). Thus along
solutions $\pa,\pb,\pc$ are constant, and from (\ref{eq:XH}) we see
that
$$
   x(t)=x_0+\pa t,\qquad y(t)=y_0+\pb t
$$
are periodic iff $\pa=\pb=0$ (so in particular $k=1/2$). But then
$z(t)=z_0+\pc t$ is not periodic since $\pc\neq 0$, so there are no
contractible closed orbits for $k\geq 1/2$.

Since $d\psi=\omega$, to compute $\Omega(v)$ we just need to integrate
$\psi$ along one of the closed orbits of $X_{H}$ described above.
Since $\psi(X_{H})=2k+\pc$ we deduce
\[\Omega(v)=T(2k+\pc)=2\pi(1-\sqrt{1-2k})>0.\]
This also shows that $\Sigma_k$ is tame.
\end{proof}

\begin{remark} It is not hard to show that for $k<1/2$, the closed contractible
orbits form a Morse-Bott 4-dimensional submanifold of $\Sigma_{k}$
diffeomorphic to $(\Gamma\setminus G)\times S^{1}$.
\end{remark}

\begin{lemma}
The Ma\~n\'e critical values equal $c=c_0=1/2$.
\end{lemma}

\begin{proof}
As noted above, $2H(\gamma)\equiv 1$ implies $c_0\leq 1/2$. On the
other hand, Lemma~\ref{conhei} and Lemma~\ref{lem:exact} imply
$c_0\geq 1/2$, hence $c_0=1/2$. Since $\Gamma$ is nilpotent, it is
amenable and $c=c_0$ by property (iv) in Section~\ref{ss:exact}.
\end{proof}

\begin{lemma}
The hypersurface $\Sigma_{1/2}$ is not stable.
\end{lemma}

\begin{proof}
By Lemma \ref{cco}, $\Sigma_{1/2}$ has no closed contractible
orbits. By Lemma \ref{disp}, $\Sigma_{k}$ is displaceable for any $k<1/2$,
hence Theorem \ref{thm:crit} implies that $\Sigma_{1/2}$ is not stable.
\end{proof}

\subsection{$PSL(2,\R)$}\label{ss:PSL}
The magnetic flow discussed here also
appears in \cite{Bu}.
It will be convenient for us to identify $PSL(2,\R)$ with
$S\BH^2$, the unit sphere bundle of the upper half plane
with its usual metric of curvature $-1$
\[\frac{dx^2+dy^2}{y^{2}}.\]
We recall that the standard identification $\phi\co S\BH^2\to PSL(2,\R)$ is
given as follows: $\phi(x,y,v)$ is the unique M\"obius transformation $T$ (with real coefficients)
such that $T(i)=x+iy$ and whose derivative at $i$ takes the tangent vector
$(0,1)$ at $i$ to the tangent vector $v$ at $(x,y)$.
In $S\BH^2$ we consider coordinates $(x,y,\theta)$ and the
following vector fields
\begin{align*}
X&=y\cos\theta\,\frac{\partial}{\partial x}+y\sin\theta\,\frac{\partial}{\partial y}-\cos\theta\,\frac{\partial}{\partial\theta},\\
Y&=-y\sin\theta\,\frac{\partial}{\partial x}+y\cos\theta\,\frac{\partial}{\partial y}+\sin\theta\,\frac{\partial}{\partial\theta},\\
V&=\frac{\partial}{\partial\theta}.
\end{align*}
It is easy to check that these vector fields satisfy the bracket
relations:
\[[X,Y]=-V,\;\;\;[V,X]=Y,\;\;\;[V,Y]=-X,\]
so they form the basis of a Lie algebra isomorphic
to $\mathfrak{sl}(2,\R)$. In fact under the map $\phi$ the vector fields
$X$, $Y$ and $V$ correspond to the left invariant vector fields
on $PSL(2,\R)$ whose values at the Lie algebra $\mathfrak{sl}(2,\R)$ are:
\[X=\left(\begin{array}{cc}
1/2&0\\
0&-1/2\\
\end{array}\right),\;\;Y=\left(\begin{array}{cc}
0&-1/2\\
-1/2&0\\
\end{array}\right),\;\;V=\left(\begin{array}{cc}
0&1/2\\
-1/2&0\\
\end{array}\right).
\]
The dual coframe of left-invariant 1-forms $\{\alpha,\beta,\gamma\}$
is given by
\begin{align*}
\alpha&=\frac{\cos\theta\,dx+\sin\theta\,dy}{y},\\
\beta&=\frac{-\sin\theta\,dx+\cos\theta\,dy}{y},\\
\gamma&=\frac{dx}{y}+d\theta.
\end{align*}
Using this coframe we identify $T^{*}S\BH^2$ with $S\BH^2\times \mathfrak{sl}(2,\R)^{*}$ and we obtain coordinates $(x,y,\theta,\pa,\pb,\pc)$.
The left-invariant coordinates $(\pa,\pb,\pc)$ are related to $(p_{x},p_{y},p_{\theta})$ by the formulas
\begin{align*}
\pa&=(yp_{x}-p_{\theta})\cos\theta+yp_{y}\sin\theta,\\
\pb&=-(yp_{x}-p_{\theta})\sin\theta+yp_{y}\cos\theta,\\
\pc&=p_{\theta}.
\end{align*}

We will endow $S\BH^2$ with its Sasaki metric. In this case this means
the unique metric that makes the basis $\{X,Y,V\}$ an orthonormal frame
at every point. Explicitly
\[ds^{2}=\frac{1}{y^{2}}\left[\dot{x}^2+\dot{y}^2+(\dot{\theta}y+\dot{x})^{2}\right].\]
The Hamiltonian is
\[2H=(yp_{x}-p_{\theta})^2+(yp_{y})^2+p_{\theta}^2=\pa^2+\pb^2+\pc^2.\]
The magnetic field is given by the left-invariant exact 2-form
\[\sigma:=d\gamma=\frac{dx\wedge dy}{y^{2}}.\]
The Hamiltonian vector field of $H$ with respect to the twisted symplectic form $\omega$ is given by
\begin{equation} \label{eq:XHsl1}
X_H = \left\{ \begin{array}{lclclcl}
\dot{x}   &=& y(yp_{x}-p_{\theta}),          & \hspace{10mm} & \dot{p}_{x}   &=& p_{y},\\
\dot{y} &=& y^{2}p_{y}, &            & \dot{p}_{y} &=&(-yp_{x}+p_{\theta})(p_{x}+1/y)-yp_{y}^2 ,\\
\dot{\theta} &=& 2p_{\theta}-yp_{x}, &           & \dot{p}_{\theta} &=&  0.
\end{array}
\right.
\end{equation}
In terms of the left-invariant coordinates we get
\begin{equation} \label{eq:XHsl2}
X_H = \left\{ \begin{array}{lclclcl}
\dot{x}   &=& y(\pa\cos\theta-\pb\sin\theta),          & \hspace{10mm} & \dot{p}_{\alpha}   &=& 2\pb \pc+\pb,\\
\dot{y} &=& y(\pa\sin\theta+\pb\cos\theta), &            & \dot{p}_{\beta} &=&-2\pa\pc-\pa ,\\
\dot{\theta} &=&\pc-\pa\cos\theta+\pb\sin\theta, &           & \dot{p}_{\gamma} &=&  0.
\end{array}
\right.
\end{equation}
The Poisson brackets of the left-invariant coordinate functions are
\[\{\pa,\pb\}=-\pc-1,\;\;\;\{\pa,\pc\}=-\pb,\;\;\;\{\pb,\pc\}=\pa.\]
Let the Liouville 1-form be $\nu=p_{x}dx+p_{y}dy+p_{\theta}d\theta$ and let
\[\psi:=\nu+\tau^{*}\gamma.\]
Clearly $d\psi=\omega$. We compute
\begin{align}
\psi(X_{H})&=p_{x}y(yp_{x}-p_{\theta})+y^{2}p_{y}^{2}+(2p_{\theta}-yp_{x})(p_{\theta}+1)+(yp_{x}-p_{\theta})\nonumber\\
&=2H+p_{\theta}=2H+\pc. \label{cosl}
\end{align}

Let $\Gamma\subset PSL(2,\R)$ be a discrete cocompact lattice which
acts without fixed points on $\BH^2$. We denote the quantities on the
quotient $T^*(\Gamma\setminus PSL(2,\R))$ by the same letters
$H,\dots$. Let $\Sigma_{k}:=H^{-1}(k)$ ($k>0$).

\begin{lemma} The hypersurface $\Sigma_{k}$ is of contact type
if and only if $k>1/2$.
\label{consl}
\end{lemma}

\begin{proof}
Equation (\ref{cosl}) tells us right away that if
$k>1/2$, then $\Sigma_{k}$ is of contact type with contact form $\psi$
since in this case
$$
   \psi(X_H) = 2k+\pc\geq 2k-\sqrt{2k}>0
$$
for every point in $\Sigma_{k}$.
An inspection of (\ref{eq:XHsl2}) reveals that
$\pa=\pb=0$, $\pc=-\sqrt{2k}$, $x=x_0$, $y=y_{0}$, $\theta(t)=\theta_0-\sqrt{2k}t$ give
orbits of $X_{H}$ with energy $k$ which project to closed orbits
in $T^*(\Gamma\setminus G)$. These closed orbits are in fact
null-homologous. To see this note that
$\Gamma\setminus PSL(2,\R)$ can be identified with the unit circle bundle
of the closed Riemann surface $\Gamma\setminus \BH^2$. The closed orbits
of $X_{H}$ under consideration, when projected to $\Gamma\setminus PSL(2,\R)$,
are precisely the unit circles in $S(\Gamma\setminus \BH^2)$.
But these circles are null-homologous. Indeed, it is well known (and it follows from the Gysin sequence of the circle bundle), that the foot-point projection
map $S(\Gamma\setminus\BH^2)\mapsto \Gamma\setminus\BH^2$ induces an isomorphism
$H^{1}(\Gamma\setminus\BH^2,\R)\cong H^{1}(S(\Gamma\setminus\BH^2),\R)$.

Let $\delta$ be one of these distinguished null-homologous closed orbits of $X_{H}$ and let $\tau$ be its period. Using (\ref{cosl}) we see that
for $k\leq 1/2$
\begin{equation}
\int_{\delta}\psi= (2k-\sqrt{2k})\tau\leq 0.\label{negl}
\end{equation}
Now let $\mu$ be the Liouville measure on $\Sigma_k$. Again, using (\ref{cosl})
we obtain
\begin{equation}
\int_{\Sigma_{k}}\psi(X_{H})\,d\mu=\int_{\Sigma_{k}}(2k+\pc)\,d\mu=2k>0.\label{posl}
\end{equation}
By Lemma~\ref{lem:liouville} the Liouville measure is exact as a
current, so (\ref{negl}), (\ref{posl}) and Corollary~\ref{cor:mcduff}
ensure that $\Sigma_{k}$ cannot be of contact type.
Alternatively, instead of the Liouville measure one may also take the distinguished null-homologous closed orbits with $\pa=\pb=0$, $\pc=\sqrt{2k}$, $x=x_0$, $y=y_0$ and $\theta(t)=\theta_0+\sqrt{2k}t$.
The integral of $\psi$ along them has value $(2k+\sqrt{2k})\tau>0$.
\end{proof}

We now prove that in this example we have a gap between $c$ and $c_0$
of size $1/4$.

\begin{lemma} The strict Ma\~n\'e critical value equals $c_0=1/2$, but
the Ma\~n\'e critical value equals $c=1/4$.
\label{lemma:cc_0}
\end{lemma}

\begin{proof} Note that $2H(\gamma)=1$ and thus $c_0\leq 1/2$.
If $c_0<1/2$, then by Lemma~\ref{lem:exact} we would have energy
levels $\Sigma_k$ with $k<1/2$ which are of contact type, which is
impossible by Lemma \ref{consl}. Thus $c_0=1/2$.

Now consider the 1-form
$$\delta:=\frac{dx}{y}+\frac{d\theta}{2}.$$
Clearly $d\delta=d\gamma$ in $S\BH^2$ and $2H(\delta)=1/2$.
(Note that $\delta$ is {\it not} left-invariant.)
Thus $c\leq 1/4$. The fact that $c=1/4$ is most easily seen as follows.
Consider the Lagrangian $L$ in $S\BH^2$ determined by
$\delta$, namely
\[L=\frac{1}{2y^{2}}\left[\dot{x}^2+\dot{y}^2+(\dot{\theta}y+\dot{x})^{2}\right]
+\frac{\dot{x}}{y}+\frac{\dot{\theta}}{2}.\]
Let $C_r$ be the curve in $\BH^2$ given by the boundary of a geodesic disk
$D_r$ of radius $r$. We parametrize $C_r$ so that it has hyperbolic speed
equal to $\sqrt{2k}$ and it is oriented clockwise. Let
$\ell_r=2\pi\sinh r$ be its hyperbolic length.
If we let $(x(t),y(t))$ be the chosen parametrization for $C_r$,
let $B_r$ be the closed curve given by $[0,\ell_r/\sqrt{2k}]\ni t\mapsto (x(t)/\sqrt{2},y(t))$.
By taking $\theta$ constant, we may see $B_r$ as a closed contractible curve
in $S\BH^2$. We will compute the action $A_{L+k}(B_{r})$ and show that
for any $k<1/4$, there is $r$ such that $A_{L+k}(B_{r})<0$. This implies
$c=1/4$ (cf. Section \ref{sec:mane}). Indeed we have
\begin{align*}
A_{L+k}(B_{r})&=\sqrt{2k}\ell_r+\frac{1}{\sqrt{2}}\int_{0}^{\ell_r/\sqrt{2k}}\frac{\dot{x}}{y}\,dt\\
&=\sqrt{2k}\ell_r-\frac{1}{\sqrt{2}}\int_{D_r}\frac{dx\wedge dy}{y^2}\\
&=2\pi\left(\sqrt{2k}\sinh r-\frac{\cosh r-1}{\sqrt{2}}\right).
\end{align*}
The last expression shows that if $k<1/4$, then
$A_{L+k}(B_{r})\to -\infty$ as $r\to\infty$.

\end{proof}

\begin{lemma}The hypersurface $\Sigma_{k}$ is stable for $k<1/4$ and $k\in (1/4,1/2)$.
\end{lemma}

\begin{proof} Consider the following form on $\Sigma_k$
\[\varphi:=\frac{-\pa d\pb+\pb d\pa}{\pa^2+\pb^2}.\]
This form is smooth away from $\pc=\pm\sqrt{2k}$ and $d\varphi=0$.
As in the case of the Heisenberg group, we look for stabilizing 1-forms
$\lambda$ of the form
\[\lambda=f(\pc)\psi+g(\pc)\varphi,\]
where $f$ and $g$ are smooth functions in $[-\sqrt{2k},\sqrt{2k}]$
and $g$ vanishes in neighbourhoods of the end points. Using
$\psi(X_H)=2k+\pc$ and $\varphi(X_H)=1+2\pc$, the condition
$\lambda(X_{H})>0$ is equivalent to
\begin{equation}\label{con1sl}
(2k+\pc)f(\pc)+(1+2\pc)g(\pc)>0,
\end{equation}
and the condition $i_{X_{H}}d\lambda=0$ is equivalent to
\begin{equation}\label{con2sl}
(2k+\pc)f'(\pc)+(1+2\pc)g'(\pc)=0.
\end{equation}
To see that we can always choose $f$ and $g$ satisfying (\ref{con1sl})
and (\ref{con2sl}) provided that $k<1/4$ or $k\in (1/4,1/2)$ we proceed as in the case of the
Heisenberg group, but we will need some minor adjustments.

{\em Case $k<1/4$.} Here $-\sqrt{2k}<-1/2<-2k$.

Take a non-negative
smooth function
$r(\pc)$ which vanishes outside a small neighbourhood of $-2k$ which
{\it excludes} $-1/2$ such that $r(-2k)>0$ and
\[\int_{-\sqrt{2k}}^{\sqrt{2k}}\frac{2k+\pc}{1+2\pc}r(\pc)\,d\pc=0.\]
Note that this is always possible. Indeed as $r$ we could take
$(1+2\pc)q(2k+\pc)$ where $q(t)$ is a smooth non-negative even bump function
with small support containing the origin.
Now let
\[f(\pc):=\int_{-2k}^{\pc}r(t)\,dt\]
and
\[g(\pc):=-\int_{-\sqrt{2k}}^{\pc}\frac{2k+t}{1+2t}r(t)\,dt.\]
Clearly (\ref{con2sl}) holds.
With these choices, $f$ vanishes only at $-2k$, is negative at the left
of $-2k$ and positive at the right of $-2k$. Also $g\geq 0$, it vanishes
outside a small neighbourhood of $-2k$ and $g(-2k)>0$, thus (\ref{con1sl}) also holds.

{\em Case $k\in (1/4,1/2)$.} Here $-\sqrt{2k}<-2k<-1/2$. As in the previous case, take a
non-negative smooth function
$r(\pc)$ which vanishes outside a small neighbourhood of $-2k$ which
{\it excludes} $-1/2$ such that $r(-2k)>0$ and
\[\int_{-\sqrt{2k}}^{\sqrt{2k}}\frac{2k+\pc}{1+2\pc}r(\pc)\,d\pc=0.\]
Also let
\[f(\pc):=\int_{-2k}^{\pc}r(t)\,dt\]
and
\[g(\pc):=-\int_{-\sqrt{2k}}^{\pc}\frac{2k+t}{1+2t}r(t)\,dt.\]
Clearly (\ref{con2sl}) holds.
With these choices, $f$ vanishes only at $-2k$, is negative at the left
of $-2k$ and positive at the right of $-2k$. But now $g\leq 0$, it vanishes
outside a small neighbourhood of $-2k$ and $g(-2k)<0$, thus (\ref{con1sl}) also holds
since now $1+2\pc$ is negative in the small neighbourhood of $-2k$.

\end{proof}

\begin{lemma} The hypersurface $\Sigma_{k}$ is displaceable for $k<1/4$.
\label{dispsl}
\end{lemma}

\begin{proof}
 Consider the Hamiltonian in $T^{*}(\Gamma\setminus G)$
given by $f=\pb$. Clearly $\pb$ does not change along
the Hamiltonian flow of $f$ and $\pa$ and $\pc$ change according to
\begin{align*}
\dot{\pa}&=\{\pb,\pa\}=\pc+1,\\
\dot{\pc}&=\{\pb,\pc\}=\pa.
\end{align*}
These equations are easy to solve and one finds that
\[2\pa(t)=(\pa(0)+\pc(0)+1)e^{t}+(\pa(0)-\pc(0)-1)e^{-t}.\]
But if $k<1/4$ and $[\pa(0)]^2+[\pb(0)]^2+[\pc(0)]^2=2k$,
then
$$\pa(0)+\pc(0)+1\geq 1-2\sqrt{k}>0,$$
and the flow of $f$ will displace $\Sigma_k$ as desired.

\end{proof}

{\em Closed contractible orbits and entropy.} We now will study
dynamical properties of the magnetic flow. For this purpose
it is better to revert to $PSL(2,\R)$ and use the group
structure.

Equation (\ref{eq:XHsl2}) gives the Euler equations in $\mathfrak{sl}(2,\R)^*$.
These are easily solved and after a time shift one obtains:
\[\pa=A\cos\mu t,\;\;\;\;\;\pb=A\sin\mu t,\;\;\;\;\;\pc=C,\]
where $\mu:=-1-2C$, $A^2+C^2=2k$ and $k$ is the value of the energy.
This curve $b(t)$ in $\mathfrak{sl}(2,\R)^*$ gives rise to a curve $a(t)$ in
$\mathfrak{sl}(2,\R)$ via the Legendre transform of $H$:
\begin{align*}
a(t)&=A\cos\mu t X+A\sin\mu t Y+C\,V\\
&=A\cos\mu t\left(\begin{array}{cc}
1/2&0\\
0&-1/2\\
\end{array}\right)+A\sin\mu t\left(\begin{array}{cc}
0&-1/2\\
-1/2&0\\
\end{array}\right)+C\left(\begin{array}{cc}
0&1/2\\
-1/2&0\\
\end{array}\right)\\
&=\frac{A}{2}\left(\begin{array}{cc}
\cos\mu t&-\sin\mu t\\
-\sin\mu t&-\cos\mu t\\
\end{array}\right)+\frac{C}{2}\left(\begin{array}{cc}
0&1\\
-1&0\\
\end{array}\right).
\end{align*}
Let us consider the loop of matrices in $SO(2)$ given by
\[Q(t):=\left(\begin{array}{cc}
\cos(\mu t/2)&\sin(\mu t/2)\\
-\sin(\mu t/2)&\cos(\mu t/2)\\
\end{array}\right).\]
Clearly $Q(t)=e^{tR}$ where
\[R=\left(\begin{array}{cc}
0&\mu/2\\
-\mu/2&0\\
\end{array}\right).\]
Then we may write $a(t)$ as
\begin{equation}
a(t)=Q(t)\left(\begin{array}{cc}
A/2&0\\
0&-A/2\\
\end{array}\right)Q^{*}(t)+\frac{C}{2}\left(\begin{array}{cc}
0&1\\
-1&0\\
\end{array}\right),
\label{eqna}
\end{equation}
where $Q^*(t)$ denotes the transpose of $Q(t)$. Since our system is left-invariant,
 in order to find the solution of $X_{H}$ through a point $(q,p)\in PSL(2,\R)\times \mathfrak{sl}(2,\R)^*$ with $H(p)=k$ and $p=(A,0,C)$ we may proceed as follows (cf. \cite[Appendix 2]{Ar} or \cite[Section 4.4]{AM}). We compute $a(t)$ and then we
solve the matrix differential equation in $PSL(2,\R)$ given by
\[\dot{g}(t)=g(t)a(t),\;\;\;\;g(0)=I.\]
The desired solution curve is $(q\,g(t),b(t))$. To solve the matrix
differential equation we make the change of variables
\[h(t):=g(t)Q(t).\]
A calculation using (\ref{eqna}) shows that $h$ satisfies
\[\dot{h}(t)=h(t)\,d,\]
where $d$ is the constant matrix
\[d:=\frac{1}{2}\left(\begin{array}{cc}
A&C+\mu\\
-(C+\mu)&-A\\
\end{array}\right)=\frac{1}{2}\left(\begin{array}{cc}
A&-1-C\\
1+C&-A\\
\end{array}\right).\]
Therefore we have obtained the following explicit formula
for our magnetic flow $\phi_t$:
\begin{equation}
\phi_{t}(q,p)=(q\,e^{td}Q^*(t),b(t)).
\label{eqf}
\end{equation}
$\,$From this equation we shall derive various dynamical consequences.
We will say that the matrix $d\in \mathfrak{sl}(2,\R)$ is elliptic
if $(1+C)^2>A^2$, parabolic if
$(1+C)^2=A^2$ and hyperbolic if $(1+C)^2<A^2$.

Recall from Section~\ref{sec:tame} the definition of tameness and
the $\om$-energy
$$
   \Om\co X(\Sigma_k)\to\R,\qquad v\mapsto\int\bar v^*\om.
$$

\begin{lemma} The hypersurface $\Sigma_k$ has closed contractible
orbits if and only if $k<1/4$.
If $k<1/4$ and $v\in X(\Sigma_{k})$
then $\Omega(v)$ is a positive integer multiple of $\pi(1-\sqrt{1-4k})>0$
and $\Sigma_{k}$ is tame.
\label{ccosl}
\end{lemma}

\begin{proof} Before embarking into the proof we make some preliminary remarks
about $PSL(2,\R)$. A matrix $d\in \mathfrak{sl}(2,\R)$ satisfies
$e^{d}=\pm I$ if and only if $\det d>0$ (i.e. $d$ is elliptic)
and $\sqrt{\det d}/\pi$ is a positive integer. Let us call
$\tau:=\pi/\sqrt{\det d}$ the {\it period} of $d$. It has the property
of being the smallest positive real number $t$ for which
$e^{td}=\pm I$ (and in fact equal to $-I$).
The fundamental group of $PSL(2,\R)=S\mathbb H^2$ is of course $\Z$
and is generated by the loop $[0,\tau]\ni t\mapsto e^{td}$ of any elliptic
element $d$. Any two such loops determined by elements $d_1$ and $d_2$
will be homotopic iff $d_1$ and $d_2$ belong to the same side
(i.e. connected component) of the solid cone $\det d>0$ in the Lie algebra
 $\mathfrak{sl}(2,\R)$. If they belong to opposite sides of the cone,
 they will have opposite homotopy classes.

Suppose now that we look for a closed contractible orbit
 on $\Sigma_k$
with $\mu\neq 0$.
For $b(t)$ to be periodic it must have period $T:=2\pi\,l/|\mu|$ where
 $l$ is a positive integer.
Using (\ref{eqf}) we see that
\[\phi_{T}(q,p)=(q\,e^{Td},p)=(q,p)\]
and this happens if and only if the matrix $d$ is elliptic
and $e^{Td}=\pm I$ (recall that we are working in
$PSL(2,\R)$). The last equality forces $T=m\tau$, where $m$ is a positive
integer and $\tau$ is the period of $d$ as defined above.
In order for this closed orbit to be contractible in $\Sigma_k$
we need the loop
\[[0,T]\ni t\mapsto e^{td}e^{-tR}\]
to be contractible in $PSL(2,\R)$.
But the homotopy class of this loop is $m[d]-l[R]$, where
$[d]$ and $[R]$ denote the homotopy classes associated to
$d$ and $R$ respectively. Since $m$ and $l$ are positive
and $[d]=\pm [R]$, the loop will be contractible iff
$m=l$ and $[d]=[R]$ (i.e. $d$ and $R$ must belong to the same
side of the cone $\det d>0$).
Hence $2 |\mu|=1/\sqrt{\det d}$ which translates into
$2C^2+2C=-2k$. Since $d$ is elliptic we must have
$(1+C)^2>A^2=2k-C^2$ which implies $k<1/4$ as desired.

The quadratic equation $2C^2+2C=-2k$ gives the following values
for $C$ and $\mu$:
\[2C=-1\pm\sqrt{1-4k},\;\;\;\;\;\mu=\mp\sqrt{1-4k}.\]
However, the values $2C=-1-\sqrt{1-4k}$, $\mu=\sqrt{1-4k}$ give rise
to elements $d$ and $R$ such that $[d]=-[R]$. Hence the unique
values of $C$ and $\mu$ for which we obtain closed contractible orbits
are:
\[2C=-1+\sqrt{1-4k},\;\;\;\;\;\mu=-\sqrt{1-4k}.\]


Finally, note that if we are at a $p$ with $\mu=0$, then the flow
is just $\phi_{t}(q,p)=(qe^{td},p)$ and even though the orbits are
all closed, they are clearly {\it not} contractible.

Since $d\psi=\omega$, to compute $\Omega(v)$ for a contractible closed
orbit $v$ on $\Sigma_k$ we just need to integrate
$\psi$ along one of the closed orbits of $\phi_t$ described above.
Since $\psi(X_{H})=2k+C$ we deduce
\[\Omega(v)=T(2k+C)=\pi l(1-\sqrt{1-4k})>0.\]
This also shows that $\Sigma_k$ is tame.
\end{proof}

\begin{lemma}The topological entropy of $\phi_t$ on $\Sigma_k$
is zero if and only if $k\leq 1/4$.
\end{lemma}

\begin{proof} Recall that by the variational principle for topological entropy
and Ruelle's inequality, the topological entropy vanishes if
 all the Lyapunov exponents are zero. Recall also that by Pesin's formula, the topological entropy is at least the integral of the sum of the positive Lyapunov exponents with respect to the Liouville measure \cite{KH}.
Inspection of (\ref{eqf}) shows that all the Lyapunov exponents are zero if
if all possible $d's$ are elliptic or parabolic. It also shows that as soon
as there are $d's$ which are hyperbolic, there will be a positive measure
set (with respect to the Liouville measure) with positive Lyapunov exponents.
Hence the topological entropy of the flow is zero if and only
if all possible $d's$ are elliptic or parabolic.
This will happen
if and only if $2C^2+2C+1\geq 2k$ for all $C\in [-\sqrt{2k},\sqrt{2k}]$.
It is easy to see that this is equivalent to saying that $k\leq 1/4$.
\end{proof}

\begin{lemma}\label{lem:unstable}
The hypersurfaces $\Sigma_{1/2}$ and $\Sigma_{1/4}$ are not stable.
\end{lemma}

\begin{proof}
By Lemma \ref{ccosl}, $\Sigma_{1/4}$ has no closed contractible
orbits. By Lemma \ref{dispsl}, $\Sigma_{k}$ is displaceable for any $k<1/4$,
hence Theorem \ref{thm:crit} implies that $\Sigma_{1/4}$ is not stable.
Alternatively, we could say that $\Sigma_{1/4}$ contains a copy of the horocycle flow
(arising when $\mu=0$) and thus it cannot be stable.

The proof that $\Sigma_{1/2}$ is not stable is a bit more involved and we need
some preliminary observations.

\begin{lemma} Let $N$ be a closed manifold with a non-zero geodesible
vector field $F$.
Suppose $G$ is a compact Lie group acting on $N$ and leaving $F$ invariant.
Then there exists a $G$-invariant 1-form which stabilizes $F$.
\label{average}
\end{lemma}

\begin{proof} Let $\mu$ be the Haar probability measure and $\varphi_g$
the diffeomorphism on $N$ determined by $g\in G$.
Let $\lambda$ be a stabilizing 1-form for $F$ and set
\[\tilde{\lambda}:=\int_{G}\varphi_{g}^{*}\lambda\,d\mu(g).\]
It is straightforward to check that $\tilde{\lambda}$ is $G$-invariant
and that it stabilizes $F$.
\end{proof}

Consider the flow of $V$ on $\Gamma\setminus PSL(2,\R)$; it is simply given by
the right action of $e^{tV}$. Its lift to
$T^*(\Gamma\setminus PSL(2,\R))=\Gamma\setminus PSL(2,\R)\times \mathfrak{sl}(2,\R)^*$ is given by
\begin{equation}
(q,p)\mapsto (qe^{tV},\mbox{\rm Ad}^{*}_{e^{tV}}(p)).
\label{formV}
\end{equation}
This gives an action of $S^1$ that leaves $H$, $\tau^*\gamma$ and $d\psi$
invariant, so
$S^1$ acts on $(T^*(\Gamma\setminus PSL(2,\R)),\omega)$
by Hamiltonian transformations and it has moment map $p_{\gamma}$.
We denote the vector field of this $S^1$ action by $V^*$.

Let $\Sigma_{1/2,C}\subset \Sigma_{1/2}$ be given by $p_{\gamma}=C$
and $C$ is constant (energy-momentum reduction).
Clearly $S^1$ leaves $\Sigma_{1/2,C}$ invariant and, $X_H$ and $V^*$
are linearly independent except when $C=\pm 1$.

Suppose $\Sigma_{1/2}$ is stable. By Lemma \ref{average} we may take
an $S^1$-invariant stabilizing 1-form $\lambda$.
Consider the restriction of both $\lambda$ and $\psi$ to the
4-manifold $\Sigma_{1/2,C}$ for $C>-1$ close to $-1$.
The forms $d\lambda$ and $d\psi=\om$ are both annihilated by
$X_H$. The form $d\psi$ is annihilated by $V^*$ and we claim that
$d\lambda$ is also annihilated by $V^*$. To see this, it suffices
to show that $\lambda(V^*)$ is a constant function since $\lambda$
is $S^1$-invariant. Observe that the function $\lambda(V^*)$ is invariant
under both $S^1$ and the flow of $X_H$ so it
descends to a function on $\Sigma_{1/2,C}/S^{1}=\Gamma\setminus PSL(2,\R)$
which is invariant under the reduced flow of $X_H$. However, this reduced
flow is given by the right action of $e^{td}$. This follows from
(\ref{eqf}) and (\ref{formV}) (note that $b(t)=\mbox{\rm Ad}^{*}_{e^{-t\mu V}}(p)$). Since $d$ is hyperbolic for $-1<C<0$,
the right action of $e^{td}$ is Anosov and thus $\lambda(V^*)$ must be
 a constant.

Hence, the forms $d\lambda$ and $d\psi=\om$ are both annhilated by
$X_H$ and $V^*$.
Since $d\psi$ is non-degenerate in a transverse
plane to the span of $X_H$ and $V^*$ we deduce that there exists
a smooth function $f\co \Sigma_{1/2,C}\to\R$ such that on $\Sigma_{1/2,C}$ we have
\[d\lambda=f\,d\psi.\]
The function $f$ is invariant under both $S^1$ and the flow of $X_H$ so it
also descends to a function on $\Sigma_{1/2,C}/S^{1}=\Gamma\setminus PSL(2,\R)$
which is invariant under the reduced flow of $X_H$ and as above, we deduce that $f$ must be
 a constant. If we now let $C$ vary close to $-1$ we get $f=f(C)$ and
$d\lambda=f(C)d\psi$. We wish to derive a contradiction from this.

We proceed as follows. On $\Sigma_{1/2}=\Gamma\setminus PSL(2,\R)\times S^2$
we let $\tilde{X}:=(X,0)$ and $\tilde{Y}:=(Y,0)$. These vectors are of course
tangent to $\Sigma_{1/2,C}$ for every $C$.
A simple calculation shows that
$d\psi(\tilde{X},\tilde{Y})=d\nu(\tilde{X},\tilde{Y}) +
\tau^{*}d\gamma(\tilde{X},\tilde{Y})=0+1=1$
and thus
$f(C)=d\lambda(\tilde{X},\tilde{Y})$. We now let $C\to -1$ and we may
suppose that $f(C)\to a$ for some $a\in\R$. On
$\Sigma_{1/2,-1}=\Gamma\setminus PSL(2,\R)\times \{(0,0,-1)\}$ we have
\[d\lambda=a\,d\psi.\]
The 1-form $\lambda-a\,\psi$ is closed and
$(\lambda-a\,\psi)(X_{H})=\lambda(X_{H})>0$ since $\psi(X_{H})=0$ on
$\Sigma_{1/2,-1}$.
But this is impossible since the Hamiltonian flow on $\Sigma_{1/2,-1}$
is given by the circle action
$$
   \phi_t(q,p)=(qQ^*(t),p),
$$
which is completely periodic with all closed orbits homologous to
zero. Thus $\Sigma_{1/2}$ cannot be stable and
Lemma~\ref{lem:unstable} is proved.
\end{proof}

\subsection{${\bf Sol}$-manifolds}\label{ss:sol}
Let $G={\bf Sol}$ be the semidirect product of $\R^2$ with $\R$, with
coordinates $(y_0,y_1,u)$ and multiplication
\begin{equation*} \label{eq:sm}
(y_0,y_1,u)\star(y_0',y_1',u')=(y_0+e^{u}y_0', y_1+e^{-u}y_1',u+u').
\end{equation*}
The map $(y_0,y_1,u) \mapsto u$ is the epimorphism ${\bf Sol} \to \R$
whose kernel is the normal subgroup $\R^2$.
The group ${\bf Sol}$ is isomorphic to the matrix group
\[\left(\begin{array}{ccc}

e^u&0&y_0\\
0&e^{-u}&y_1\\
0&0&1\\

\end{array}\right).\]

It is not difficult to see that ${\bf Sol}$ admits cocompact lattices.
Let $A\in SL(2,\mathbb Z)$ be such that there is $P\in GL(2,\mathbb R)$ with
\[PAP^{-1}=\left(\begin{array}{cc}
\lambda &0\\
0&1/\lambda\\
\end{array}\right)\]
and $\lambda>1$.
There is an injective homomorphism
$$\mathbb Z^2\ltimes_{A}\mathbb{Z}\hookrightarrow \mathbf{Sol}$$
given by $(m,n,l)\mapsto (P(m,n),(\log\lambda)\,l)$ which defines
a cocompact lattice $\Gamma$ in $\mathbf{Sol}$.
The closed 3-manifold $\Gamma\setminus {\bf Sol}$ is a 2-torus
bundle over the circle with hyperbolic gluing map $A$.

If we denote by $p_{u}$, $p_{y_0}$ and $p_{y_1}$ the momenta that
are canonically conjugate to $u$, $y_0$ and $y_1$ respectively,
then the functions
\begin{equation*} \label{eq:mom}
\begin{array}{lcl}
\alpha_0 &=& e^{u} p_{y_0}, \\
\alpha_1 &=& e^{-u} p_{y_1}, \\
\nu &=& p_{u}
\end{array}
\end{equation*}
are left-invariant functions on $T^*{\bf Sol}$. The closed $2$-form
\begin{equation} \label{eq:magnetic}
\sigma = -d y_0 \wedge d y_1
\end{equation}
is also left-invariant, and we consider the twisted symplectic form

\[\omega_{\sigma} = d p_{u} \wedge du + d p_{y_0} \wedge dy_0 +
d p_{y_1} \wedge dy_1 -dy_0 \wedge dy_1 .\]

The Casimir of $\g^*$ is $f=\nu+\alpha_{0}\alpha_{1}$
and the left-invariant Hamiltonian is given by
$2H:=\alpha_{0}^2+\alpha_{1}^2+\nu^2$.
An easy calculation shows that the vector field of the magnetic flow is:
\[X_{H}=(e^u\alpha_{0},\,e^{-u}\alpha_{1},\,\nu,\,-\alpha_1+\nu\alpha_0,\,
\alpha_0-\nu\alpha_1,\,\alpha_{1}^2-\alpha_{0}^2).\]

The dynamics of the magnetic flow is very interesting and was investigated
in \cite{BP}. It turns out that all energy levels have positive
Liouville entropy and they all carry closed contractible orbits.
The magnetic field is non-exact and is a generator of $H^{2}(\Gamma\setminus {\bf Sol},\R)$. Since the lattices are solvable there are no bounded
primitives for $\sigma$ in ${\bf Sol}$ and the Ma\~n\'e critical value
is $\infty$.

\begin{proposition} The hypersurface $\Sigma_{k}$ is stable and displaceable
for any $k>0$.
\end{proposition}

\begin{proof}

Let
\[\lambda:=fdu+\frac{1}{2}(\alpha_{0}d\alpha_{1}-\alpha_{1}d\alpha_{0}).\]
Since all objects involved in the definition of $\lambda$ are left-invariant, $\lambda$ descends to compact quotients of ${\bf Sol}$.
A calculation using that $df(X_{H})=0$ shows that
$i_{X_{H}}d\lambda=0$. One also sees that
\[\lambda(X_{H})=k+\frac{\nu^2}{2}>0,\]
and thus $\lambda$ stabilizes $\Sigma_{k}$. The proof that
$\Sigma_k$ is displaceable can be found in \cite{BP} and is similar
to the proof of Lemma \ref{disp}.

\end{proof}

\subsection{The case of $\sigma$ symplectic}\label{ss:symp}
In this subsection we study the special case in which the closed left-invariant
2-form $\sigma$ is symplectic and we show that low energy levels are always stable.

\begin{proposition}\label{prop:symp}
 Assume that $\sigma$ is symplectic and
let $H(\mu)=|\mu|^2/2$ for some fixed positive definite inner product in $\g^*$.
Then there exists $k_0>0$ such that for all $k<k_0$ the
hypersurface $\Sigma_k$ is stable.
\end{proposition}

\begin{proof}
We will show that there exists $k_0>0$ such that for all $k<k_0$
the Euler vector field $E$ is geodesible in $\bS_k$. By
Lemma~\ref{lem:proj-stable} this implies the stability of $\Sigma_k$.

Since $\sigma$ is symplectic, inspection of the
bracket given in (ii) of Lemma \ref{evf} shows that
for all $k$ sufficiently small, the Poisson bracket
$\{\;,\;\}_{\sigma}$ is non-degenerate on the set
$H^{-1}([0,k])\subset\g^*$.
In other words, for all $\mu\in H^{-1}([0,k])$ the linear map
$B_{\mu}\co \g\to \g^*$ given by
\[B_{\mu}(X)(w):=\mu([X,w])-\sigma_{e}(X,w)\]
is invertible.
Let $A\co \g\to \g^*$ be $A(X)(w):=\sigma_{e}(X,w)$.
We may write
\[E(\mu)=E^0(\mu)-J\mu\]
where $J\co \g^*\to\g^*$ is defined by $J\mu(w):=\sigma(d_\mu H,w)$ and
$E^0$ is the Euler vector field for $\sigma=0$. A
short computation shows that $J$ is uniquely determined by
$$
   \hat{\sigma}(\xi,\eta) := \sigma(A^{-1}\xi,A^{-1}\eta)=
   \langle \xi,J^{-1}\eta\rangle.
$$
For $k$ small the Poisson structure $\{\;,\;\}_{\sigma}$ is induced by
the symplectic form $\varpi$ in $H^{-1}([0,k])$ given by
\[\varpi(\mu)(B_{\mu}(v),B_{\mu}(w)) := \mu([v,w])-\sigma(v,w).\]
Inverting $B_{\mu}$ we get
\[\varpi=-\hat{\sigma}+d\tau\]
where $\tau$ is a 1-form such that
$|\tau(\mu)|=O(k)$ for $\mu\in H^{-1}([0,k])$. (The 2-form
$\varpi+\hat{\sigma}$ is
$O(k^{1/2})$, so we can choose a primitive which is $O(k)$.)
Let $\alpha$ be the 1-form in $\g^*$ given by
$\alpha(\mu)(\eta):=\langle\mu,J^{-1}\eta\rangle/2$. Then
\[\varpi=d(-\alpha+\tau).\]
On $\bS_k$ we have $i_{E}\varpi=0$ and
\begin{align*}
(-\alpha+\tau)(\mu)(E(\mu))&=\alpha(\mu)(J\mu)-\alpha(\mu)(E^0(\mu))+\tau(\mu)(E^0(\mu))
-\tau(\mu)(J\mu)\\
&=k+O(k^{3/2}),
\end{align*}
because $\alpha(J\mu)=|\mu|^2/2=k$ and $E^0(\mu)$ is quadratic in
$\mu$. Thus for $k$ small $(-\alpha+\tau)(E)>0$ on $\bS_k$
and $E$ is geodesible (in fact contact).
\end{proof}

Suppose now that $G$ admits a cocompact lattice $\Gamma\subset G$.
Then we have a symplectic manifold $(\Gamma\setminus G, \sigma)$.
Any left invariant metric will give rise to a magnetic flow with
stable low energy levels. We now look in detail at a concrete example.

\begin{example}
Consider the 4-dimensional nilpotent Lie algebra
$\g$ with basis $\{X_1,X_2,X_3,X_4\}$ whose only non-zero bracket is
$[X_1,X_2]=X_3$. If $\{e_1,e_2,e_3,e_4\}$ is the dual basis of $\g^*$, then
$de_1=de_2=de_4=0$ and $de_3=-e_1\wedge e_2$. The symplectic 2-form is
\[\sigma=-e_1\wedge e_3-e_2\wedge e_4.\]
If $x_i$ denote the coordinates of $\mu\in\g^*$ in the given basis we
easily find that the non-zero Poisson brackets of $\{\;,\;\}_{\sigma}$ are:
\[\{x_1,x_2\}=x_3,\;\;\;\;\{x_1,x_3\}=\{x_2,x_4\}=1.\]
We see right away that $\{\;,\;\}$ is non-degenerate
at every point $\mu\in\g^*$.
In this case it is easy to invert the operator $B_\mu$ of the previous
proposition and we find $\varpi$ to be
\begin{align*}
\varpi&=-\sigma-x_{3}\,dx_{3}\wedge dx_{4}\\
&=dx_1\wedge dx_3+dx_2\wedge dx_4-x_{3}\,dx_{3}\wedge dx_{4}.
\end{align*}
If we let
\[\theta:=\frac{x_1 dx_3-x_3 dx_1+x_2 dx_4-x_4dx_2}{2}
-\frac{x^{2}_{3}dx_4}{2}\]
then
\[\varpi=d\theta.\]

Now take $H(\mu)=|\mu|^2/2$, where the inner product is defined so that
$\{e_1,e_2,e_3,e_4\}$ is an orthonormal basis. The Euler vector
field $E_{H}$ is easily computed using Lemma \ref{evf}. One finds
\begin{align*}
E_{H}(\mu)&=(-x_{2}x_{3}-x_{3},x_{1}x_{3}-x_4,x_{1},x_{2})\\
&=E_{H}^0(\mu)-J\mu\\
&=(-x_{2}x_{3},x_{1}x_{3},0,0)-
(x_{3},x_4,-x_{1},-x_{2}).
\end{align*}
Next we compute $\theta(E_{H})$ and we find:
\[\theta(E_{H})=\frac{x_1^2+x_2^2+x_3^2+x_4^2}{2}-\frac{x_1x_3x_4}{2}.\]Using that
\[\frac{x_1x_3x_4}{2}=-\frac{d(x_2x_4)(E_H)}{2}+\frac{x_4^2-x_2^2}{2}\]
we obtain
\[\left(\theta+\frac{d(x_2x_4)}{2}\right)(E_H)=\frac{x_1^2+2x_2^2+x_3^2}{2}.\]
Then
\[\varphi:=\theta+\frac{d(x_2x_4)}{2}\]
is a primitive of $\varpi$ such that on $\bS_k$:
\[\varphi(E_{H})\geq 0\]
with equality if and only if $x_1=x_2=x_3=0$ and $x_4=\pm \sqrt{2k}$.
Since this set is not invariant under the flow of $E_{H}$ we conclude
that for any invariant Borel probability measure $\nu$ on $\bS_k$ we have
\[\int_{\bS_k}\varphi(E_{H})\,d\nu> 0\]
and thus $E_{H}$ is in fact of contact type by Theorem \ref{thm:mcduff}.

Summarizing, we have shown that $E_H$ is geodesible on $\bS_k$ for any $k>0$.
Let $G$ be the simply connected Lie group with Lie algebra $\g$.
This group certainly admits cocompact lattices $\Gamma$.
The manifold $(\Gamma\setminus G,\sigma)$ is symplectic (with first Betti
number 3 in fact) and with the left invariant metric considered above
we obtain that $\Sigma_k$ is stable for any $k>0$. Every compact set here
is displaceable by \cite[Theorem B]{BP} and the Ma\~n\'e critical value is
$\infty$ since $\Gamma$ is nilpotent.

A realisation of $G$ is $\mathcal H\times \R$ where $\mathcal H$ is the
3-dimensional Heisenberg group. The metric is just the product metric.
It is easy to see that $\mathcal H$ has no contractible closed geodesics, hence
the same is true for $G$.
\end{example}

\section{Proofs of the theorems}\label{sec:proofs}

\subsection{Proof of Theorem \ref{thm:RFH}}
For the case of tame stable hypersurfaces and tame stable homotopies
the theorem is a consequence of Theorems \ref{thm:rfh=0} and \ref{inva}.
Note that we could take as definition of $RFH(\Sigma)$ either
$\overline{RFH}(\Sigma,V)$ or $\underline{RFH}(\Sigma,V)$.
The claims for the virtually contact case are straightforward extensions
of the contact case treated in \cite{CF}.
Indeed, the virtually contact condition implies that
the period-action inequality for almost Reeb orbits of Proposition 3.2
in \cite{CF} continues to hold with constants now depending on the constants
appearing in the virtually contact condition. Here we use that
$\pi_1(\Sigma)$ injects into $\pi_1(M)$.
Similarly, for the
time-dependent case Proposition 3.4 in \cite{CF}
continues to hold for a virtually contact homotopy. Having established these
two Propositions the forthcoming arguments can be repeated just word by
word.

\begin{remark} In the virtually contact case, one can replace
the condition that $\pi_1(\Sigma)$ injects into $\pi_1(V)$ by a weaker
condition which we now explain.

Consider a Hamiltonian structure $\omega$ on a closed odd dimensional
manifold $\Sigma$. We call the pair $(\Sigma,\omega)$
\emph{virtually contact}, if on a cover $\pi \colon
\widetilde{\Sigma}\to \Sigma$ there exists a primitive $\lambda \in
\Omega^1(\widetilde{\Sigma})$ of $\pi^*\omega$ with the property
that for one and hence every Riemannian metric $g$ on $\Sigma$ there
exists a constant $c=c(g)>0$ such that
$$||\lambda||_{\pi^* g} \leq c, \qquad ||\xi||_{\pi^* g}
\leq c|\lambda(\xi)|,\quad \forall\,\, \xi \in \mathrm{ker}(\pi^*
\omega).$$ We refer to the one-form $\lambda$ as a \emph{bounded
primitive}.
A closed hypersurface $\Sigma$ in a symplectic manifold
$(V,\omega)$ is called \emph{of virtual restricted contact type}, if
there exists a cover $\pi \colon \widetilde{V} \to V$ and a
primitive $\lambda$ of $\pi^* \omega$ such that
$\lambda|_{\pi^{-1}(\Sigma)}$ is a bounded primitive for
$(\Sigma, \omega|_{\Sigma})$.

\begin{lemma}
Assume that $(V,\omega)$ is symplectically aspherical,
i.e.\,\,$\omega$ vanishes on $\pi_2(V)$, and $\Sigma \subset V$ is a
closed hypersurface with the property that
$(\Sigma,\omega|_{\Sigma})$ is virtually contact and the inclusion
homomorphism $i_* \colon \pi_1(\Sigma) \to \pi_1(V)$ is injective.
Then $\Sigma$ is of virtual restricted contact type.
\end{lemma}

\begin{proof} Let $\pi \colon \widetilde{V} \to V$ be the
universal cover of $V$. Since $\widetilde{V}$ is simply connected we
have by Hurewicz theorem that
$$H_2(\widetilde{V};\mathbb{Z})=\pi_2(\widetilde{V})=\pi_2(V).$$
Since $(V,\omega)$ is symplectically aspherical we deduce that
$\pi^*\omega$ vanishes on $H_2(\widetilde{V};\mathbb{Z})$ and hence
$$[\pi^*\omega]=0 \in H^2_{dR}(\widetilde{V}).$$
We conclude that there exists $\lambda_0 \in
\Omega^1(\widetilde{V})$ such that
$$d\lambda_0=\pi^* \omega.$$
Because $(\Sigma,\omega|_{\Sigma})$ is virtually contact, there
exists a bounded primitive $\lambda_1$ on a cover of $\Sigma$. Since
$\lambda_{1}$ is bounded, there exists a
bounded lift $\widetilde{\lambda}_1$ to $\pi^{-1}(\Sigma) \subset
\widetilde{V}$. Note that
$$d\lambda_0|_{\pi^{-1}(\Sigma)}-d\widetilde{\lambda}_1=
\pi^*\omega|_{\pi^{-1}(\Sigma)}-\pi^*\omega|_{\pi^{-1}(\Sigma)}=0.$$
We conclude that
$\lambda_0|_{\pi^{-1}(\Sigma)}-\widetilde{\lambda}_1$ defines a
class
$$\big[\lambda_0|_{\pi^{-1}(\Sigma)}-\widetilde{\lambda}_1\big]
\in H^1_{dR}\big(\pi^{-1}(\Sigma)\big).$$ However, since
$i_* \colon \pi_1(\Sigma) \to \pi_1(V)$ is injective,
each connected
component of $\pi^{-1}(\Sigma)$ is simply connected. Therefore
$H^1_{dR}\big(\pi^{-1}(\Sigma)\big)=\{0\}$ and consequently there
exists $f \in C^\infty\big(\pi^{-1}(\Sigma),\mathbb{R}\big)$ such
that
$$df=\lambda_0|_{\pi^{-1}(\Sigma)}-\widetilde{\lambda}_1.$$
Extend $f$ to a smooth map $\widetilde{f} \in
C^\infty(\widetilde{V},\mathbb{R})$ and set
$$\lambda=\lambda_0-d\widetilde{f}.$$
Noting that
$$\lambda|_{\pi^{-1}(\Sigma)}=\lambda_0|_{\pi^{-1}(\Sigma)}-df=\widetilde{\lambda
}_1$$
we conclude that $\lambda$ is a primitive of $\pi^*\omega$ whose
restriction to $\pi^{-1}(\Sigma)$ is bounded. This finishes the
proof of the lemma.
\end{proof}

Theorem \ref{thm:RFH} continues to hold with the same proof if we assume
that $\Sigma$ is of virtual restricted contact type.
The hypersurfaces of Lemma \ref{virtc} are of virtual restricted contact type,
but when $\dim M=2$, $\pi_{1}(\Sigma)$ does not inject into $\pi_{1}(T^*M)$.

\end{remark}

\subsection{Proof of Theorem \ref{thm:sigma=0}}
Let $\sigma=0$. By Lemma~\ref{lemma:contact}, the Ma\~n\'e critical
value is given by $c=\max\, U$ and all regular level sets $\Sigma_k$
are of restricted contact type, so $RFH(\Sigma_k)$ is defined.

For $k>c$ each $\Sigma_k$ is regular. By the invariance of
Rabinowitz Floer homology under contact homotopies, we can compute
$RFH(\Sigma_k)$ for zero potential.
In this case it is non-zero by Corollary 1.12 in~\cite{CFO}.
In particular, $\Sigma_k$ is non-displaceable for $k\geq
c$. The dynamics on $\Sigma_k$, $k>c$, is given by the geodesic flow
of the Jacobi metric $g/(k-U)$. The level set $\Sigma_c$ is singular.

For $k<c$, $\Sigma_k$ is displaceable because it misses one fibre.

\subsection{Proof of Theorem \ref{thm:non-displaceable}}
For $k>c$, $\Sigma_k$ is virtually contact by Lemma \ref{virtc} (in fact it
 is of virtual restricted contact type) thus
$RFH(\Sigma_k)$ is defined and invariant under virtually contact homotopies.
In particular the isomorphism type of $RFH(\Sigma_{k})$ is independent
of $k$.

Consider a path of Riemannian metrics $[0,1]\ni t\mapsto g_{t}$
such that $g_0=g$ and $g_1$ is a metric of negative curvature.
By taking $k>\max_{t\in [0,1]}c_{t}$ and using that the isomorphism
type of $RFH(\Sigma_{k,t})$ is independent of $t$ as well, we may suppose
that $g$ itself is negatively curved. But it is well
known that the geodesic flow of a negatively curved metric
is Anosov and does not have non-trivial
closed contractible orbits. By structural stability the same is
true for the orbits of $\Sigma_{k}$ provided $k$ is sufficiently large.
It follows from Theorem \ref{thm:RFH} that $RFH(\Sigma_{k})$ does not
vanish and thus $\Sigma_k$ is non-displaceable.


\subsection{Proof of Theorem~\ref{thm:pinched}}
This is just Theorem~\ref{insthe}.

\subsection{Proof of Theorem \ref{thm:crit}}
We argue by contradiction. Suppose the energy level $\Sigma_{c}$
is stable. We will show that $\Sigma_{c}$ must contain a contractible
periodic orbit.
By Lemma \ref{lem:HZ-stable} there exists a tubular neighborhood
$(-\varepsilon,\varepsilon)\times \Sigma_{c}$ of $\Sigma_{c}$
such that the characteristic foliations on $\{r\}\times \Sigma_{c}$
are conjugate for all $r\in (-\varepsilon,\varepsilon)$.
But by hypothesis, every compact set contained
in the set $H^{-1}(-\infty,c)$ is displaceable, hence we may apply
the main result of Schlenk \cite{Sch}, to conclude that
for almost every $r$, with $r<0$, $\{r\}\times \Sigma_{c}$ carries
a closed contractible orbit. By stability, $\Sigma_{c}$ must
also carry a closed contractible orbit.

\subsection{Proof of Theorems \ref{thm:nil}, \ref{thm:PSL} and
  \ref{thm:torus}}
Theorem \ref{thm:nil} follows from the lemmas in
Section \ref{sub:heisenberg} and Theorem \ref{thm:RFH}.
Theorem \ref{thm:PSL} follows from the lemmas in
Section \ref{ss:PSL} and Theorem \ref{thm:RFH}.
Theorem \ref{thm:torus} follows from the results
in Section \ref{sub:tori} and Theorem \ref{thm:RFH}.


\bigskip

Kai Cieliebak, Ludwig-Maximilians-Universit\"at, 80333 M\"unchen, Germany\\
E-mail: kai@math.lmu.de
\smallskip

Urs Frauenfelder, Department of Mathematics and Research Institute of
Mathematics, Seoul National University\\
E-mail: frauenf@snu.ac.kr
\smallskip

Gabriel P.~Paternain, University of Cambridge, Cambridge CB3 0WB, UK\\
E-mail: g.p.paternain@dpmms.cam.ac.uk

\end{document}